\documentclass[11pt,twoside]{article}

\usepackage{graphicx}
\usepackage{styleset}
\usepackage{macros}

\title	 {\normalfont\headfam\itshape\LARGE
          Knot homology groups from instantons}
\author	 {P. B. Kronheimer and T. S. Mrowka%
            
            \footnote{%
             The work of the first author
            was supported by the National Science Foundation through
            NSF grant number DMS-0405271. The work of
            the second author was supported by NSF grants DMS-0206485,
            DMS-0244663 and DMS-0805841.} }
\address {Harvard University, Cambridge MA 02138 \\
          Massachusetts Institute of Technology, Cambridge MA 02139}

\begin{document}

\pagestyle{numbered}	
\maketitle

\section{Introduction}

\subsection{An observation of some coincidences}

For a knot or link  $K$ in $S^{3}$, the Khovanov homology $\kh(K)$  
is a bigraded abelian group whose construction can be described in entirely
combinatorial terms \cite{Khovanov-1}. If we forget the bigrading, then as
abelian groups we have, for example,
\[
        \kh(\mathrm{unknot}) = \Z^{2}
\]
and
\[
        \kh(\mathrm{trefoil}) = \Z^{4} \oplus \Z/2.
\]
The second equality holds for both the right- and left-handed
trefoils, though the bigrading would distinguish these two cases.

The present paper was motivated in large part by the
observation that the group $\Z^{4}\oplus \Z/2$ arises in a different
context.  Pick a basepoint $y_{0}$ in the complement of the knot or
link,  and consider the space of all homomorphisms
\[
            \rho : \pi_{1}\bigl(S^{3}\sminus K, y_{0} \bigr) \to \SU(2)
\]
satisfying the additional constraint that
\begin{equation}\label{eq:conjugacy}
            \text{$\rho(m)$ is conjugate to $\begin{pmatrix} -i & 0 \\
            0 & i \end{pmatrix}$}
\end{equation}
for every $m$ in the conjugacy class of a meridian of the link. (There
is one such conjugacy class for each component of $K$, once the
components are oriented. The orientation does not matter here, because
the above element of $\SU(2)$ is conjugate to its inverse.) Let us write
\[
            R(K) \subset\Hom\bigr(
            \pi_{1}\bigl(S^{3}\sminus K, y_{0} \bigr) , \SU(2)\bigr)
\]
for the set of these homomorphisms. Note that we are not defining
$R(K)$ as a  set of
\emph{equivalence classes} of such homomorphisms under the action of
conjugation by  $\SU(2)$. The observation, then, is the following:

\begin{observation}\label{obs:startingpoint}
    In the case that $K$ is either the unknot or the trefoil, the
    Khovanov homology of $K$ is isomorphic to the ordinary homology of
    $R(K)$, as an abelian group. That is,
    \[
                \kh(K) \cong H_{*}(R(K)).
    \]
    This observation extends to all the torus knots of type $(2,p)$.
\end{observation}

To understand this observation, we can begin with the case of the
unknot, where the fundamental group of the complement is $\Z$. After
choosing a generator, we have a correspondence between
$R(\mathrm{unknot})$ and the conjugacy class of the distinguished
element of $\SU(2)$ in \eqref{eq:conjugacy} above.  This conjugacy
class is a $2$-sphere in $\SU(2)$, so we can write
\[
        R(\mathrm{unknot}) = S^{2}.
\]
For a non-trivial knot $K$, we always have homomorphisms $\rho$ which
factor through the abelianization $H_{1}(S^{3}\sminus K) = \Z$, and
these are again parametrized by $S^{2}$. Every other homomorphism has
stabilizer $\{\pm 1\} \subset \SU(2)$ under the action by conjugation,
so its equivalence class contributes a copy of $\SU(2)/\{\pm 1\} =
\RP^{3}$ to $R(K)$. In the case of the trefoil, for example, there is
exactly one such conjugacy class, and so
\[
                R(\mathrm{trefoil}) = S^{2} \amalg \RP^{3};
\]
and the homology of this space is indeed $\Z^{4}\oplus \Z/2$, just
like the Khovanov homology. This explains why the observation holds
for the trefoil, and the case of the $(2,p)$ torus knots is much the
same: for larger odd $p$, there are $(p-1)/2$ copies of $\RP^{3}$ in
the $R(K)$. In unpublished work, the above observation has been shown
to extend to all $2$-bridge knots by Sam Lewallen \cite{Lewallen}.

The homology of the space $R(K)$, while it is certainly an invariant
of the knot or link, should not be expected to behave well or share
any of the more interesting properties of Khovanov homology; no should
the coincidence noted above be expected to hold. A better way to
proceed is instead to imitate the construction of Floer's instanton
homology for 3-manifolds, by constructing a framework in which $R(K)$
appears as the set of critical points of a Chern-Simons functional on
a space of $\SU(2)$ connections on the complement of the link.  One
should then construct the Morse homology of this Chern-Simons
invariant. In this way, one should associate a finitely-generated
abelian group to $K$ that would coincide with the ordinary homology of
$R(K)$ in the very simplest cases. The main purpose of the present
paper is to carry through this construction. The invariant that comes
out of this construction is certainly not isomorphic to Khovanov
homology for all knots; but it does share some of its formal
properties.  The definition that we propose is a variant of the
orbifold Floer homology considered by Collin and Steer in
\cite{Collin-Steer}.

In some generality, given a knot or link $K$ in a $3$-manifold $Y$,
we will produce an ``instanton Floer homology group'' that is an
invariant of $(Y,K)$. These groups will be functorial for oriented
cobordisms of pairs. Rather than work only with $\SU(2)$, we will
work of much of this paper with a more general compact Lie group $G$,
though in the end it is only for the case of $\SU(N)$ that we are able
to construct these invariants. 

\subsection{Summary of results}

\paragraph{The basic construction.}

Let $Y$ be a closed oriented $3$-manifold, let $K\subset Y$ be an
oriented link, and let $P\to Y$ be a principal $U(2)$-bundle.
Let $K_{1},\dots,K_{r}$ be the components of $K$. We will say that
$(Y,K)$ and $P$ satisfies the \emph{non-integrality condition} if
none of the $2^{r}$ rational cohomology classes
\begin{equation}\label{eq:non-integrality}
                    \tfrac{1}{2} c_{1}(P) \pm\tfrac{1}{4} \PD[K_{1}]
                    \pm \dots \pm \tfrac{1}{4} \PD[K_{r}]
\end{equation}
is an integer class.
When the non-integrality condition holds, we will define a
finitely-generated abelian group $\I_{*}(Y,K,P)$. This group has a
canonical $\Z/2$ grading, and a relative grading by $\Z/4$.

In the case that $K$ is empty, the group $\I_{*}(Y,P)$ coincides with the
familiar variant of Floer's instanton homology arising from a $U(2)$
bundle $P\to Y$ with odd first Chern class \cite{Donaldson-book}. We 
recall, in outline, how this group is constructed. One considers the
space $\cA(Y,P)$ of all connections in the $\SO(3)$ bundle $\ad(P)$. This
affine space is acted on by the ``determinant-1 gauge group'': the
group $\G(Y,P)$ of automorphisms of $P$ that have determinant $1$
everywhere. Inside $\cA(Y,P)$ one has the flat connections: these can be
characterized as critical points of the Chern-Simons functional,
\[
\CS:\cA(Y,P)\to\R.
\]
The Chern-Simons functional descends to a
circle-valued function on the quotient space
\[
   \bonf(Y,P) = \cA(Y,P)/\G(Y,P).
\]
The image of
the set of critical points in $\bonf(Y,P)$ is compact,
and after perturbing $\CS$ carefully by a term that is invariant under
$\G(Y,P)$, one obtains a function whose set of critical points has finite
image in this quotient. If $(1/2)c_{1}(P)$ is not an integral class
and the perturbation is small,
then the critical points in $\cA(Y,P)$ are all irreducible connections.
One can arrange also a Morse-type non-degeneracy condition: the
Hessian if $\CS$ can be assumed to be non-degenerate in the directions
normal to the gauge orbits. The group $\I_{*}(Y,P)$ is then constructed
as the Morse homology of the circle-valued Morse function on
$\bonf(Y,P)$.

In the case that $K$ is non-empty, the construction of $\I_{*}(Y,K,P)$
mimics the standard construction very closely. The difference is that
we start not with the space $\cA$ of all smooth connections in $\ad(P)$,
but with a space $\cA(Y,K,P)$ of connections in the restriction of
$\ad(P)$
to $Y\sminus K$ which have a singularity along $K$. This space is
acted on by a group $\G(Y,K,P)$ of determinant-one gauge
transformations, and we have a quotient space
\[
            \bonf(Y,K,P) = \cA(Y,K,P)/\G(Y,K,P).
\]
In the case that
$c_{1}(P)=0$, the singularity
is such that the flat connections in the quotient space
$\bonf(Y,K,P)$ correspond to conjugacy classes of
homomorphisms from the fundamental group of $Y\sminus K$ to $\SU(2)$
which have the behavior \eqref{eq:conjugacy} for meridians of the
link. Thus, if we write $\Crit(Y,K,P) \subset \bonf(Y,K,P)$ for this set of
critical points of the Chern-Simons functional, then we have
\begin{equation}\label{eq:crit-is-R}
            \Crit(Y,K,P) = R(Y,K) / \SU(2)
\end{equation}
where $R(Y,K)$ is the set of homomorphisms $\rho: \pi_{1}(Y\sminus
K) \to \SU(2)$ satisfying \eqref{eq:conjugacy} and $\SU(2)$ is acting
by conjugation.
The non-integrality of the classes \eqref{eq:non-integrality} is
required in order to ensure that there will be no reducible flat
connections.

\paragraph{Application to classical knots.}
Because of the non-integrality requirement, the construction of
$\I_{*}$ cannot be applied directly when the $3$-manifold $Y$ has
first Betti number zero. In particular, we cannot apply this
construction to ``classical knots'' (knots in $S^{3}$).
However, there is a simple
device we can apply. Pick a point $y_{0}$ in $Y\sminus K$, and form
the connected sum at $y_{0}$ of $Y$ and $T^{3}$, to obtain a new pair
$(Y\# T^{3}, K)$. Let $P_{0}$ be the trivial $U(2)$ bundle on $Y$,
and let $Q$ be the $U(2)$ bundle on $T^{3}= S^{1}\times T^{2}$ whose
first Chern class is Poincar\'e dual to $S^{1}\times
\{\mathrm{point}\}$. We can form a bundle $P_{0} \# Q$ over $Y\#
T^{3}$. This bundle satisfies the non-integral condition, so we define
\[
                    \tI_{*}(Y,K) = \I_{*}(Y\# T^{3}, K, P_{0} \# Q).
\]
and call this the \emph{framed instanton homology} of the pair $(Y,K)$.
In the special case that $Y=S^{3}$, we write
\[
                \tI_{*}(K) = \tI_{*}(S^{3}, K).
\]

To get a feel for $\tI_{*}(Y,K)$ for knots in $S^{3}$, and to
understand the reason for the word ``framed'' here, it is first necessary to understand
that the adjoint bundle $\ad(Q)\to T^{3}$ admits only irreducible flat
connections, and that these form two orbits under the determinant-one
gauge group. (See section~\ref{subsec:config-spaces}. Under the
\emph{full} gauge group of all automorphisms of
$\ad(Q)$, they form a single orbit.) When we form a connected sum,
the fundamental group becomes a free product, and we have a general
relationship of the form
\begin{equation}\label{eq:connected-sum-reps}
                        R(Y_{0}\# Y_{1})/\SU(2)
                        = R(Y_{0}) \times_{\SU(2)} R(Y_{1}).
\end{equation}
Applying this to the connected sum $(Y\# T^{3}, K)$ and recalling
\eqref{eq:crit-is-R}, we find that
 the
flat connections in the quotient space $\bonf(Y\#
T^{3},K,P_{0}\#Q)$ form two disjoint copies of the space we called
$R(Y,K)$ above: that is,
\[
\begin{aligned}
        \Crit (Y\#T^{3}, K, P_{0}\# Q) 
                 &= R(Y,K) \amalg R(Y,K).
\end{aligned}
\]
Note that, on the right-hand side, we no longer have the quotient of
$R(Y,K)$ by $\SU(2)$ as we did before at \eqref{eq:crit-is-R}. The
space $R(Y,K)$ can be thought of as parametrizing isomorphism classes
of flat connections on $Y\sminus K$ with a \emph{framing} at the
basepoint $y_{0}$: that is, an isomorphism of the fiber at $y_{0}$
with $U(2)$.

For a general knot, as long as the set of critical points is
non-degenerate in the Morse-Bott sense, there will be a spectral
sequence starting at the homology of the critical set, $\Crit$, and
abutting to the framed instanton homology.  In the case of the unknot
in $S^{3}$, the spectral sequence is trivial and the group
$\tI_{*}(K)$ is the homology of the two copies of $R(K)$ which
comprise $\Crit$. Thus,
\[
\begin{aligned}
\tI_{*}(\mathrm{unknot})
&= H_{*}(\Crit) \\
&= H_{*}\bigl(R(\mathrm{unknot}) \amalg R(\mathrm{unknot})\bigr)\\
&= H_{*}(S^{2} \amalg S^{2})
                \\
                &= \Z^{2} \oplus \Z^{2}.
                \end{aligned}
\]
Noting again Observation~\ref{obs:startingpoint}, we can say that
$\tI_{*}(\mathrm{unknot})$ is isomorphic to two copies of the Khovanov
homology of the unknot. It is natural to ask whether this isomorphism
holds for a larger class of knots:

\begin{question}\label{conj:tI-Kh}
    Is there an isomorphism of abelian groups
    \[
                    \tI_{*}(K) \cong \kh(K) \oplus \kh(K)
    \]
    for all alternating knots?
\end{question}

There is evidence for an affirmative answer to this question for
the torus knots of type $(2,p)$, as well as for the (non-alternating)
torus knots of type $(3,4)$ and $(3,5)$. It also seems likely that the
answer is in the affirmative for all alternating knots if we use
$\Z/2$ coefficients instead of integer coefficients.  Already for the
$(4,5)$ torus knot, however, it is clear from an examination of $R(K)$
that the framed instanton homology $\tI_{*}(K)$ has smaller rank than
two copies of the Khovanov homology, so the isomorphism does not extend
to all knots.

For a general knot, we expect that $\kh(K) \oplus \kh(K)$
is related to $\tI_{*}(K)$ through a spectral sequence. There is a
similar spectral sequence (though only with $\Z/2$ coefficients)
relating (reduced) Khovanov homology to the Heegaard Floer homology of
the branched double cover \cite{OS-spectral}. The argument of
\cite{OS-spectral} provides a potential model for a similar argument
in the case of our framed instanton homology. An important ingredient
is to show that there is a long exact sequence relating the framed
instanton homologies for $K$, $K_{0}$ and $K_{1}$, when $K_{0}$ and
$K_{1}$ are obtained from a knot or link $K$ by making the two
different smoothings of a single crossing. It seems that a proof of
such a long exact sequence can be given using the same ideas that were
used in \cite{KMOS} to prove a surgery exact sequence for
Seiberg-Witten Floer homology, and the authors hope to return to this
and other related issues in a future paper.

\paragraph{Other variations.} Forming a connected sum with $T^{3}$, as
we just did in the definition of $\tI_{*}(Y,K)$, is one way to take an
arbitrary pair $(Y,K)$ and modify it so as to satisfy the
non-integrality condition; but of course there are many other ways.
Rather than using $T^{3}$, one can use any pair satisfying the
non-integrality condition; and an attractive example that -- like the
$3$-torus -- carries isolated flat connections, is the pair
$(S^{1}\times S^{2}, L)$, where $L$ is the $3$-component link formed
from three copies of the $S^{1}$ factor.

We shall examine this and some other
variations of the basic construction
in section~\ref{sec:classical}. Amongst these is a
``reduced'' version of $\tI_{*}(K)$ that appears to bear the same
relation to reduced Khovanov homology as $\tI_{*}(K)$ does to the
Khovanov homology of $K$. (For this reduced variant, the homology of
the unknot is $\Z$.)  Another variant arises if, instead of forming a
connected sum, we perform $0$-surgery on a knot (in $S^{3}$, for
example) and apply the basic construction to the core of the surgery
torus in the resulting $3$-manifold. The group obtained this way is
reminiscent of the ``longitude Floer homology'' of Eftekhary
\cite{Eftekhary}. It is trivial for the ``unknot'' in any $Y$ (i.e.~a
knot that bounds a disk), and is $\Z^{4}$ for the trefoil in $S^{3}$.

In another direction, we can alter the construction of $\tI_{*}(K)$ by
dividing the relevant configuration space by a slightly larger gauge
group, and in this way we obtain a variant of $\tI_{*}(K)$ which we
refer to as $\bar\tI_{*}(K)$ and which is (roughly speaking) half the
size of $\tI_{*}(K)$. For this variant, the appropriate modification
of Question~\ref{conj:tI-Kh} has only one copy of $\kh(K)$ on the
right-hand side. (For example, if $K$ is the unknot, then
$\bar\tI_{*}(K)$ is the ordinary homology of $S^{2}$, which coincides
with the Khovanov homology.)

\paragraph{Slice-genus bounds.}
A very interesting aspect of Khovanov homology was discovered by
Rasmussen \cite{Rasmussen-slice}, who showed how the Khovanov homology
of a knot can be used to provide a lower bound for the knot's
slice-genus. An argument with a very similar structure can be
constructed using the framed instanton homology $\tI_{*}(K)$. The
construction begins by replacing $\Z$ as the coefficient group with a
certain system of local coefficients $\Gamma$ on $\bonf(S^{3},K)$. In this
way we obtain a new group $\tI_{*}(K;\Gamma)$ that is
finitely-presented module over the ring $\Z[t^{-1},t]$ of
finite Laurent series in a variable $t$. We shall show that
$\tI_{*}(K;\Gamma)$ modulo torsion is essentially independent of the knot $K$:
it is always a free module of rank $2$. On the other hand, $\tI_{*}(K;\Gamma)$ comes with a
descending filtration, and we can define a knot invariant by
considering the level in this filtration at which the two generators lie.
From this knot invariant, we obtain a lower bound for the slice genus.

Although the formal aspects of this argument are modeled on
\cite{Rasmussen-slice}, the actual mechanisms behind the proof are the same
ones that were first used in \cite{KM-gtes-II} and
\cite{K-obstruction}.

\paragraph{Monotonicity and other structure groups.}
The particular conjugacy class chosen in \eqref{eq:conjugacy} is a
distinguished one. It might seem, at first, that the definition of
$\I_{*}(Y,K,P)$ could be carried out without much change if instead we
used any of the non-central conjugacy classes in $\SU(2)$ represented
by the elements
\[
            \exp 2\pi i 
            \begin{pmatrix}
                -\lambda & 0 \\ 0 & \lambda
            \end{pmatrix}
\]
with $\lambda$ in $(0,1/2)$. This is not the case, however, because
unless $\lambda=1/4$ we cannot establish the necessary finiteness
results for the spaces of trajectories in the Morse theory that
defines $\I_{*}(Y,K,P)$. The issue is what is usually called
``monotonicity'' in a similar context in symplectic topology. In
general, a Morse theory of Floer type involves a circle-valued Morse
function $f$ on an infinite-dimensional space $\bonf$ whose periods define a map
\[
             \Delta_{f} : \pi_{1}(\bonf) \to \R.      
\]
The Hessian of $f$ may also have spectral flow, defining another map,
\[
                \sflow : \pi_{1}(\bonf) \to \Z.
\]
The theory is called monotone if these two are proportional. Varying
the eigenvalue $\lambda$ varies the periods of the Chern-Simons functional in
our theory, and it is only for $\lambda=1/4$ that we have
monotonicity. In the non-monotone case, one can still define a Morse
homology group, but it is necessary to use a local system that has a
suitable completeness \cite{Novikov}.

A related issue is the question of replacing $\SU(2)$ by a general
compact Lie group, say a simple, simply-connected Lie group $G$. The
choice of $\lambda$ above now becomes the choice of an element $\Phi$
of the Lie algebra of $G$, which will determine the leading term in
the singularity of the connections that we use.  If we wish to
construct a Floer homology theory, then the choice of $\Phi$ is
constrained again by the monotonicity requirement. It turns out that
the monotonicity condition is equivalent to requiring that the adjoint
orbit of $\Phi$ is K\"ahler-Einstein manifold with Einstein constant
$1$, when equipped with the K\"ahler metric corresponding to the
Kirillov-Kostant-Souriau $2$-form.  We shall develop quite a lot of
the machinery in the context of a general $G$, but in the end we find
that it is only for $\SU(N)$ that we can satisfy two competing
requirements: the first requirement is monotonicity; the second
requirement is that we avoid connections with non-trivial stabilizers
among the critical points of the perturbed Chern-Simons functional.
Using $\SU(N)$, we shall construct a variant of $\tI_{*}(K)$ that
seems to bear the same relation to the $\sl(N)$ Khovanov-Rozansky
homology \cite{Khovanov-Rozansky} as $\tI_{*}(K)$ does to the original
Khovanov homology.

\subsection{Discussion}

Unlike $\tI_{*}(K)$, the Khovanov homology of a knot is bigraded. If
we write $\KPol_{K}(q,t)$ for the $2$-variable Poincar\'e polynomial
of $\kh(K)$, then $\KPol_{K}(q,-1)$ is (to within a standard factor)
the Jones polynomial of $K$ \cite{Khovanov-1}. Relationships between
the Jones polynomial and gauge theory can be traced back to Witten's
reinterpretation of the Jones polynomial as arising from a
$(2+1)$-dimensional topological quantum field theory
\cite{Witten-jones-tqft}. What we are exploring here is in one higher
dimension: a relationship between the Khovanov homology and gauge
theory in $3+1$ dimensions.

Our definition of $\tI_{*}(Y,K)$ seems somewhat unsatisfactory, in
that it involves a rather unnatural-looking connected sum with
$T^{3}$. As pointed out above, one can achieve apparently the same
effect by replacing $T^{3}$ here with the pair $(S^{1}\times S^{2},
L)$ where $L$ is a standard $3$-stranded link. The unsatisfactory
state of affairs is reflected in the fact that we are unable to prove
that these two choices would lead to isomorphic homology groups. The
authors feel that there should be a more natural construction,
involving the Morse theory of the Chern-Simons functional on the
``framed'' configuration space $\tilde\bonf(Y,K) =
\cA(Y,K)/\G^{o}(Y,K)$, where $\G^{o}\subset \G$ is the subgroup
consisting of gauge transformations that are $1$ at a basepoint
$y_{0}\in (Y\sminus K)$.  The reduced homology theory would be
constructed in a similar manner, but using a basepoint $k_{0}$ on $K$.
A related construction, for homology $3$-spheres, appears in an
algebraic guise in \cite[section 7.3.3]{Donaldson-book}.

This idea of using the framed configuration space $\tilde\bonf(Y,K)$
and dispensing with the connect-sum with $T^{3}$ is attractive: it would enable one to work with a general $G$ without concern for avoiding reducible solutions. However, it cannot be carried
through without overcoming obstacles involving bubbles in
the instanton theory: the particular issue is bubbling at the chosen
base-point. 

\subparagraph{Acknowledgments.} The development of these ideas was
strongly influenced by the paper of Seidel and Smith
\cite{Seidel-Smith}. Although gauge theory as such does not appear
there, it does not seem to be far below the surface.  The first author
presented an early version of some of the ideas of the present paper
at the Institute for Advanced Study in June 2005, and learned there
from Katrin Wehrheim and Chris Woodward that they were pursuing a very
similar program (developing from \cite{Wehrheim-Woodward} in the
context of Lagrangian intersection Floer homology). Ciprian Manolescu
and Chris Woodward have described a similar program, also involving
Lagrangian intersections, motivated by the idea of using the framed
configuration space. The idea of using a 3-stranded link in
$S^{1}\times S^{2}$ as an alternative to $T^{3}$ was suggested to us
by Paul Seidel and also appears in the work of Wehrheim and Woodward.
The authors have been informed that Magnus Jacobsson and Ryszard
Rubinsztein independently noticed the coincidence described in
Observation~\ref{obs:startingpoint}, for various knots
\cite{Jacobsson-Rubinsztein}.

\section{Instantons with singularities}

For the case that the structure group is $\SU(2)$ or $\PSU(2)$,
instantons with codimension-2 singularities were studied in
\cite{KM-gtes-I,KM-gtes-II} and related papers. Our purpose here is to
review that material and at the same time to generalize some of the
constructions to the case of more general compact groups $G$.
In the next section, we will be considering cylindrical $4$-manifolds
and Floer homology; but in this section we begin with the closed case.
We find it convenient to work first with the case that $G$ is simple and
simply-connected. Thus our discussion here applies directly to
$\SU(N)$ but not to $U(N)$. Later in this section we will
indicate the adjustments necessary to work with other Lie groups,
including $U(N)$.

For instantons with codimension-2 singularities and arbitrary
structure group, the formulae  for the energy of solutions and the
dimension of moduli spaces which we examine here are closely related
to similar formulae for non-abelian Bogomoln'yi monopoles. See
\cite{Murray-1, Murray-Singer, Weinberg}, for example.

\subsection{Notation and root systems}
\label{subsec:root-systems}

For use in the rest of the paper, we set down some of the
notation we shall use for root systems and related matters.
Fix $G$, a compact connected Lie group that is both simple
and simply connected.  We will fix a
maximal torus $T\subset G$ and denote by $\ft \subset \g$ its Lie
algebra. Inside $\ft$ is the integer lattice consisting of points
$x$ such that $\exp(2\pi x)$ is the identity. The dual lattice is the
lattice of weights: the elements in $\ft^{*}$ taking integer values on
the integer lattice. We denote by $\Roots\subset\ft^{*}$ the set of
roots. We choose a set of positive roots $\Roots^{+}\subset \Roots$, so that
$\Roots = \Roots^{+}\cup \Roots^{-}$, with $\Roots^{-}=-\Roots^{+}$.
The set of simple roots
corresponding to this choice of positive roots will be
denoted by $\Sroots^{+}\subset \Roots^{+}$. We denote by $\weyl$ half
the sum of the positive roots, sometimes called the \emph{Weyl vector},
\[
        \weyl = \frac{1}{2}\sum_{\beta\in \Roots^{+}}\beta,
\]
and we write $\theta$ for the highest root.

We define the Killing form on $\g$
with a minus sign, as
 \[
 \langle a, b\rangle = -\tr(\ad(a)\ad(b)),
\]
 so
that it is positive definite. The corresponding map $\g^{*}\to \g$
will be denoted $\alpha\mapsto\alpha^{\dag}$, and the inverse map is
$\psi \mapsto \psi_{\dag}$. If $\alpha$ is a root, we denote by
$\alpha^{\vv}$ the coroot
\[
                \alpha^{\vv} = \frac{2 \alpha^{\dag}}{\langle
                \alpha,\alpha\rangle}.
\]
The simple coroots, $\alpha^{\vv}$ for $\alpha\in \Sroots^{+}$, form an
integral basis for the integer lattice in $\ft$. The \emph{fundamental
weights} are the elements of the dual basis $\w_{\alpha}$, $\alpha\in \Sroots^{+}$ for
the lattice of weights. The \emph{fundamental Weyl chamber} is the cone in
$\ft$ on which all the simple roots are non-negative.\footnote{Note
that our convention is that the fundamental Weyl chamber is closed: it is not the
locus where the simple roots are \emph{strictly} positive.}
This is the cone spanned by the duals of the fundamental
weights, $\w_{\alpha}^{\dag}\in \ft$.

The highest root $\theta$ and the corresponding coroot $\theta^{\vv}$
can be written as positive integer
combinations of the simple roots and coroots respectively: that is,
\[
                \begin{aligned}
                    \theta &= \sum_{\alpha\in \Sroots^{+}} n_{\alpha} \alpha
                    \\
                    \theta^{\vv} &= \sum_{\alpha\in \Sroots^{+}}
                    n^{\vv}_{\alpha} \alpha^{\vv}
                    \\
                \end{aligned}
\]
for non-negative integers $n_{\alpha}$ and $n^{\vv}_{\alpha}$. The
numbers
\begin{equation}\label{eq:sum-is-coxeter}
                \begin{aligned}
                    h &=1+ \sum_{\alpha\in \Sroots^{+}} n_{\alpha}
                    \\
                    h^{\vv} &= 1 + \sum_{\alpha\in \Sroots^{+}}
                    n^{\vv}_{\alpha} 
                    \\
                \end{aligned}
\end{equation}
are the Coxeter number and dual Coxeter number respectively. The
squared length of the highest root is
equal to $1/h^{\vv}$:
\begin{equation}\label{eq:theta-length}
                \langle \theta,\theta\rangle = 1/h^{\vv}.
\end{equation}
(The inner product on $\g^{*}$ here is understood to be the dual
inner product to the Killing form.) We also record here the relation
\begin{equation}\label{eq:magic-delta-formula}
                \weyl  = \sum_{\alpha\in \Sroots^{+}} \w_{\alpha};
\end{equation}
this and the previous relation have the corollary
\begin{equation}\label{eq:delta-theta}
                2 \langle \weyl, \theta \rangle  = 1 - 1/h^{\vv}.
\end{equation}

For each root $\alpha$, there is a preferred homomorphism
\[
            j_{\alpha} : \SU(2) \to G
\]
whose derivative maps
\[
   dj_{\alpha}:
        \begin{pmatrix}
            i & 0 \\
            0 & -i
        \end{pmatrix} \mapsto \alpha^{\vv}.
\]
In the case that $G$ is simply laced, or if $\alpha$ is a long root in
the non-simply laced case, then the map $j_{\alpha}$ is injective and
represents a generator of $\pi_{3}(G)$. In
particular, this applies when $\alpha$ is the highest root $\theta$.
 Under the adjoint action of
$j_{\theta}(\SU(2))$, the Lie algebra $\g$ decomposes as one copy of
the adjoint representation of $\SU(2)$, a number of copies of the
defining $2$-dimensional representation of $\SU(2)$, and a number of
copies of the trivial representation. The pair
$(G,j_{\theta}(\SU(2)))$ is $4$-connected. 

\subsection{Connections and moduli spaces}
\label{subsec:connections-and-moduli}

Let $X$ be a closed, connected, oriented, Riemannian
$4$-manifold, and let $\Sigma\subset X$ be a smoothly embedded,
compact oriented $2$-manifold. We do not assume that $\Sigma$ is
connected. We take a principal $G$-bundle $P\to X$, where $G$ is
as above. Such a $P$ is classified by an element of $\pi_{3}(G)$, and
hence by a single
characteristic number: following \cite{AHS}, we choose to normalize
this characteristic number by defining
\[
            k = -\frac{1}{(4 h^{\vv})} p_{1}(\g)[X],
\]
where $h^{\vv}$ is the dual Coxeter number of $G$. This normalization
is chosen as in \cite{AHS} so that $k$ takes all values in $\Z$ as $P$
ranges through all bundles on $X$. If the structure group of $P$ is
reduced to the subgroup $j_{\theta}(\SU(2))$, the $k$ coincides with
the second Chern number of the corresponding $\SU(2)$ bundle.

Fix
$\Phi \in \ft$ belonging to the fundamental Weyl chamber, and suppose that
\begin{equation}\label{eq:highest-root-bound}
        \theta(\Phi) < 1
\end{equation}
where $\theta$ is the highest root. This condition is
equivalent to saying that
\[
            -1 < \alpha(\Phi) < 1
\]
for all roots $\alpha$. This in turn means that an element $U \in \g$
is fixed by the adjoint action of $\exp(2\pi\Phi)$ if and only if
$[U,\Phi]=0$. In a simply-connected group, the commutant of any
element is always connected, and it therefore follows that the
subgroup of $G$ which commutes with $\exp(2\pi\Phi)$ coincides with the
stabilizer of $\Phi$ under the adjoint action. We call this group
$G_{\Phi}$. We write $\g_{\Phi}\subset \g$ for its Lie algebra, and we
let $\fo$ stand for the unique $G$-invariant complement, so that
\begin{equation}\label{eq:fo-def}
        \g = \g_{\Phi} \oplus \fo.
\end{equation}
 The set of roots can be decomposed
according to the sign of $\alpha(\Phi)$, as
\[
           \Roots = \Roots^{+}(\Phi) \cup \Roots^{0}(\Phi) \cup
           \Roots^{-}(\Phi).
\]
 Similarly, the set
of simple roots decomposes as
\[
                \Sroots^{+} = \Sroots^{+}(\Phi) \cup
                \Sroots^{0}(\Phi),
\]
where $\Sroots^{0}(\Phi)$ are the simple roots that vanish on $\Phi$.
Knowing $\Sroots^{0}(\Phi)$, we can recover $\Roots^{0}(\Phi)$ as the set of
those positive roots lying in the span of $\Sroots^{0}(\Phi)$.
We can
decompose the complexification $\fo\otimes \C$ as
\begin{equation}\label{eq:fo-decomp}
        \fo\otimes\C = \fo^{+} \oplus \fo^{-}
\end{equation}
where $\fo^{+}\subset \g\otimes \C$ is the sum of the weight spaces
for all roots $\alpha$ in $\Roots^{+}(\Phi)$, and $\fo^{-}$ is the sum of
the weight spaces for roots in $\Roots^{-}(\Phi)$.

Choose a reduction of the structure group of $P|_{\Sigma}$ to the
subgroup $G_{\Phi} \subset G$. Extend this 
arbitrarily to a tubular neighborhood  $\nu\supset\Sigma$, so that we
have a reduction of $P|_{\nu}$. If $O \cong G/G_{\Phi}$ is the
adjoint orbit of $\Phi$ in the Lie algebra $\g$ and if $O_{P}\subset
\g_{P}$ is the corresponding subbundle of the adjoint bundle, then
such a
reduction of structure group to $G_{\Phi}\subset G$ is determined by
giving a section $\varphi$ of $O_{P}$ over the neighborhood $\nu$. We
denote the  principal $G_{\Phi}$-bundle resulting from this reduction
by $P_{\varphi}\subset P|_{\nu}$. There are corresponding reductions of
the associated bundle $\g_{P}$ with fiber $\g$ and its
complexification, which we write as
\begin{equation}\label{eq:gphi-decomp}
\begin{aligned}
\g_{P}|_{\nu} &= \g_{\varphi} \oplus \fo_{\varphi}\\
\end{aligned}
\end{equation}
and
\[
\begin{aligned}
(\g_{P}\otimes\C)|_{\nu} &= (\g_{\varphi}\otimes\C) \oplus
(\fo_{\varphi}\otimes\C)\\
              &= (\g_{\varphi}\otimes\C) \oplus \fo^{+}_{\varphi} \oplus
              \fo_{\varphi}^{-}.
\end{aligned} 
\]

We identify $\nu$ diffeomorphically with the disk bundle of the
oriented normal bundle to $\Sigma$, we let $i\eta$ be a connection
$1$-form on the circle bundle $\partial \nu$, and we extend $\eta$ by
pull-back to the deleted tubular neighborhood $\nu\sminus \Sigma$.
Thus $\eta$ is a $1$-form that coincides with $d\theta$ in polar
coordinates $(r,\theta)$ on each normal disk under the chosen
diffeomorphism.

The data $\varphi$ and $\eta$ together allow us to define the model for
our singular connections. We choose a smooth $G$-connection
$A^{0}$ on $P$. We take  $\beta(r)$ to be a
cutoff-function equal to $1$ in a neighborhood of $\Sigma$ and
supported in the neighborhood  $\nu$. Then we define
\begin{equation}\label{eq:singular-model-A}
                 A^{\varphi} =  A^{0} + \beta(r) \varphi\otimes \eta
\end{equation}
as a connection in $P$ over $X\sminus \Sigma$. (The form
$\beta(r)\varphi\eta$ has been extended by zero to all of
$X\sminus\Sigma$.)  The holonomy of $A^{\varphi}$ around a
loop $r=\epsilon$ in a normal disk to $\Sigma$ (oriented with the
standard $\theta$ coordinate increasing) defines an automorphism of
$P$ over $\nu\sminus\Sigma$ which is asymptotically equal to
\begin{equation}\label{eq:exp-phi-holonomy}
                \exp( -2\pi \varphi)
\end{equation}
when $\epsilon$ is small. 

Following \cite{KM-gtes-I}, we fix a $p>2$ and consider a  space of
connections on $P|_{X\sminus \Sigma}$ defined as
\begin{equation}\label{eq:cA-def}
            \cA^{p}(X,\Sigma,P,\varphi) = \{\,  A^{\varphi} + a \mid
                                                a, \nabla_{A^{\varphi}}
                                                a \in L^{p}(X\sminus
                                                \Sigma) \,\}.
\end{equation}
As in \cite[Section 3]{KM-gtes-I}, the definition of this space of
connections can be reformulated to make clear that it depends only on
the reduction of structure group defined by $\varphi$, and does not
otherwise depend on $\Phi$. To do this,
extend
the radial distance function $r$ as a positive function on $X\sminus
\Sigma$ and 
define a Banach space $W^{p}_{k}(X)$
for $k\ge 1$ by taking the completion of the
compactly supported smooth functions on
$X\sminus\Sigma$
with respect to norm
\[
            \|f\|_{W^{p}_{k}} = \left \| \frac{1}{r^{k}} f \right
            \|_{p} + \left \| \frac{1}{r^{k-1}} \nabla f \right
            \|_{p} + \dots + \left \|\nabla^{k} f \right
            \|_{p}.
\]
(For $k=0$, we just define $W^{p}_{k}$ to be $L^{p}$.)
The essential point then is that
 the condition on $a|_{\nu}$ that arises from the definition
\eqref{eq:cA-def} can be equivalently written (using the decomposition
\eqref{eq:gphi-decomp}) as
\[
            a|_{\nu} \in L^{p}_{1}(\nu;\g_{\varphi}) \oplus
            W^{p}_{1}(\nu; \fo_{\varphi}).
\]
This shows that the space $\cA^{p}(X,\Sigma,P,\varphi)$ depends only
on the decomposition of $\g_{P}$ as $\g_{\varphi}\oplus
\fo_{\varphi}$.
It is important here that the condition
\eqref{eq:highest-root-bound} is satisfied: this condition
ensures that the eigenvalues of the bundle automorphism
\eqref{eq:exp-phi-holonomy} acting on $\fo_{\varphi}$ are all different
from $1$, for these eigenvalues are $\exp(\pm 2\pi i \alpha(\Phi))$ as
$\alpha$ runs through $\Roots^{+}(\Phi)$.
This space of connections is acted on by the
gauge group
\[
                \G^{p}(X,\Sigma,P,\varphi) = \{\,
                g \in\Aut(P|_{X\sminus\Sigma}) \mid
                \text{ $\nabla_{A^{\varphi}}g,
                \nabla^{2}_{A^{\varphi}}g \in
                L^{p}(X\sminus\Sigma)$}\,\}.
\]
The fact that this is a Banach Lie group acting smoothly on
$\cA^{p}(X,\Sigma,P,\varphi)$ is a consequence of multiplication theorems,
such as the continuity of the multiplications $W^{p}_{2}\times
W^{p}_{1}\to W^{p}_{1}$, just as in \cite{KM-gtes-I}.

The center of the gauge group $\G^{p}=\G^{p}(X,\Sigma,P,\varphi)$ is
canonically isomorphic to the center $Z(G)$ of $G$, a finite group. This
subgroup $Z(\G^{p})$ acts trivially on the $\cA^{p}(X,\Sigma,P,\varphi)$, so
the group that acts effectively is the quotient $\G^{p}/Z(\G^{p})$. Some
connections $A$ will have larger stabilizer; but there is an important
distinction here that is not present in the most familiar case, when
$G=\SU(N)$. In the case of $\SU(N)$ if the stabilizer of $A$ is larger
than $Z(\G^{p})$, then the stabilizer has positive dimension, but for other
$G$ the stabilizer may be a finite group strictly larger than $G$.
To understand this point, 
recall that the stabilizer of a  connection $A$ in the gauge group is
isomorphic to the centralizer $C_{G}(S)$ where $S$ is the set of
holonomies around all loops based at some chosen basepoint. So the
question of which stabilizers occur is equivalent to the question of
which subgroups of $G$ arise as $C_{G}(H)$ for some $H\subset G$,
which we may as well take to be a closed subgroup. In the case of
$\SU(N)$, the only \emph{finite} group that arises this way is the
center. But for other simple Lie groups $G$, there may be a
semi-simple subgroup $H\subset G$ of the same rank as $G$; and in this
case the centralizer $C_{G}(H)$ is isomorphic to the center of $H$,
which may be strictly larger than the center of $G$.
Examples of this phenomenon include the case where $G=\Spin(2n+1)$ and
$H$ is the subgroup $\Spin(2n)$. In this case $C_{G}(H)$ contains
$Z(G)$ as a subgroup of index $2$. Another case is $G=G_{2}$ and
$H=\SU(3)$: in this case $C_{G}(H)$ has order $3$ while the center of
$G$ is trivial. 

We reserve the word \emph{reducible} for connections $A$ whose
stabilizer has positive dimension:

\begin{definition}
We will say that
a connection $A$ is \emph{irreducible} if its stabilizer in the gauge
group is finite.
\end{definition}

The homotopy type of $\G^{p}(X,\Sigma,P,\varphi)$
coincides with that of the group of all smooth automorphisms of $P$
which respect the reduction of structure group along $\Sigma$. The
bundle $P$ is classified its characteristic number $k$, and the section
$\varphi$ is determined up to homotopy by the induced map on
cohomology,
\begin{equation}\label{eq:H2-map}
                \varphi^{*} : H^{2}(O_{P}) \to H^{2}(\Sigma).
\end{equation}
Because the restriction of $P$ to $\Sigma$ is trivial, and because the
choice of trivialization is unique up to homotopy, we can also
think of the reduction of structure group along $\Sigma$ as being
determined simply by specifying the isomorphism class of the principal
$G_{\Phi}$-bundle $P_{\varphi}\to\Sigma$. In this way, when $\Sigma$ is connected,
the classification is by
$\pi_{1}(G_{\Phi})$. The inclusion $T\to G_{\Phi}$ induces a
surjection on $\pi_{1}$, so we can lift to an element of $\pi_{1}(T)$,
which we can reinterpret as an integer lattice point $\xi$ in $\ft$.
Let $Z(G_{\Phi})\subset T$ be the center of $G_{\Phi}$,  let
$\fz(G_{\Phi})$ be its Lie algebra, and let $\Pi$ be the orthogonal
projection
\[
                \Pi : \ft \to \fz(G_{\Phi}).
\]
We can describe $\fz(G_{\Phi})$ as
\[
\begin{aligned}
  \fz(G_{\Phi}) &= \bigcap_{\alpha\in S^{0}(\Phi)} \ker\alpha \\
             &= \mathrm{span}\{ \,\w_{\beta}^{\dag} \mid \beta\in
             S^{+}(\Phi)\,\}.  
\end{aligned}
\]
The projection $\Pi(\xi)$ may not be an integer  lattice point, but the
image under $\Pi$ of the integer lattice in $\ft$ is isomorphic
to $\pi_{1}(G_{\Phi})$, and the reduction of structure group is
determined up to homotopy by the element
\[
                \Pi(\xi) \in \fz(G_{\Phi}).
\]
We give the image of the integer lattice under $\Pi$ a name:
\begin{definition}\label{def:lattice}
    We write $\LL(G_{\Phi}) \subset \fz(G_{\Phi})$ for the image under $\Pi$ of the
    integer lattice in $\ft$. Thus $\LL(G_{\Phi})x$ is isomorphic both to
    $H_{2}(O;\Z)$ and to $\pi_{1}(G_{\Phi})$, and classifies the
    possible reductions of structure group of $P\to\Sigma$ to the
    subgroup $G_{\Phi}$, in the case that $\Sigma$ is connected.
    If the reduction of structure group determined by $\varphi$ is
    classified by the lattice element $\ll \in \LL(G_{\Phi})$, we refer
    to $\ll$ as the \emph{monopole charge}. If $\Sigma$ has more than
    one component, we define the monopole charge by summing over the
    components of $\Sigma$.
\end{definition}

We now wish to define a moduli space of anti-self-dual connections as
\[
            M(X,\Sigma,P,\varphi) =
                    \{ \, A \in \cA^{p} \mid F^{+}_{A} = 0 \, \}
                    \bigm/ \G^{p}.
\]
As shown in \cite{KM-gtes-I}, there
is a Kuranishi model for the neighborhood of a  connection $[A]$
in $M(X,\Sigma,P,\varphi)$ described by a Fredholm complex, as long as $p$ is
chosen sufficiently close to $2$. Specifically, if $x$ denotes the
smallest of all the real numbers $\alpha(\Phi)$ and $1-\alpha(\Phi)$
as $\alpha$ runs through $\Roots^{+}(\Phi)$, then
$p$ needs to be in the
range
\begin{equation}\label{eq:p-close-to-2}
                2 < p < 2 + \epsilon(x)
\end{equation}
where $\epsilon$ is a continuous function which is positive for $x>0$
but has $\epsilon(0)=0$. We suppose henceforth that $p$ is in this
range. The Kuranishi theory then tells us,
in particular, that if $A$ is irreducible and the
operator
\[
            d^{+}_{A} : L^{p}_{1,A}(X\sminus\Sigma;\g_{P}\otimes
            \Lambda^{1}) \to  L^{p}(X\sminus\Sigma;\g_{P}\otimes
            \Lambda^{+})
\]
is surjective, then a neighborhood of $[A]$ in $M(X,\Sigma,P,\varphi)$ is a
smooth manifold (or orbifold if the finite stabilizer is larger than
$Z(G)$),
and its dimension is equal to minus the index of the
Fredholm complex
\[
            L^{p}_{2,A}(X\sminus\Sigma;\g_{P}\otimes\Lambda^{0})
            \stackrel{d_{A}}{\to}
            L^{p}_{1,A}(X\sminus\Sigma;\g_{P}\otimes\Lambda^{1})
            \stackrel{d^{+}_{A}}{\to}
            L^{p}(X\sminus\Sigma;\g_{P}\otimes\Lambda^{+}) .           
\]
No essential change is needed to carry over the proofs from
\cite{KM-gtes-I}. The non-linear aspects come down to the
multiplication theorems for the $W^{p}_{k}$ spaces, while the Fredholm theory for
the linear operators reduces in the end to the case of a line bundle
with the weighted norms.

We refer to minus the index of the above complex as the \emph{formal
dimension} of the moduli space $M(X,\Sigma,P,\varphi)$. We can
write the formula for the formal dimension as
\begin{multline}\label{eq:dimension-fmla}
 ( 4 h^{\vv}) k + 2 \bigl\langle c_{1}(\fo^{-}_{\varphi}), [\Sigma]
    \bigr\rangle 
     + \frac{(\dim O)}{2} \chi(\Sigma)  - (\dim G) (b^{+}-b^{1}+1).
\end{multline}
Here $\dim O$ denotes the dimension of $O$ as a  real manifold: an
even number, because $O$ is also a complex manifold. The proof of this
formula can be given following \cite{KM-gtes-I} by using excision to
reduce to the simple case that the reduction of structure group can
extended to all of $X$. In this way, the calculation can eventually be
reduced to calculating the index of a Fredholm complex of the same
type as above, but with $\g_{P}$ replaced by a complex line-bundle
$\fo_{\mu}$ on $X$,
equipped with a singular connection $d_{\mu}$ of the form
\[
                \nabla + i \beta(r) \mu\eta
\]
with $\mu\in(0,1)$ (\emph{cf.} \eqref{eq:singular-model-A}). The
index calculation in the case of such a line bundle is given in
\cite{KM-gtes-I}.

We can express the characteristic class $c_{1}(\fo_{\varphi}^{-})$ that
appears in the formula \eqref{eq:dimension-fmla} in slightly different
language. As was just mentioned, the manifold $O$ is naturally a
complex manifold. To define the standard complex structure, we
identify the tangent space to $O$ at $\Phi$ with $\fo$, and we give
$\fo$ a complex structure $J$ by identifying it with $\fo^{-}\subset
\fo\otimes\C$ using the linear projection. This gives a
$G_{\Phi}$-invariant complex structure on $T_{\Phi}O$ which can be
extended to an integrable complex structure on all of $O$ using the
action of $G$. The bundle $O_{P}\to X$ is now a bundle of complex
manifolds, and we use $c_{1}(O_{P}/X)$ to denote the first Chern
class of its vertical tangent bundle. Then we can rewrite
$c_{1}(\fo^{-}_{\varphi})$ as
\[
          c_{1}(\fo^{-}_{\varphi}) = \varphi^{*}(c_{1}(O_{P}/X))
\]
in $H^{2}(\Sigma)$. Using again the canonical trivialization of
$P|_{\Sigma}$ up to homotopy, we
can also think of $\varphi$ as simply a map $\Sigma\to O$ up to homotopy,
and we can think of the characteristic class as $\varphi^{*}(c_{1}(O))$.

We can summarize the situation with a lemma:

\begin{lemma}\label{lem:dimension-fmla}
    Let
    the reduction of structure group of $P\to\Sigma$ to the
    subgroup $G_{\Phi}$ have monopole charge $\ll$.
    Then
    the formula \eqref{eq:dimension-fmla} for
    formal dimension of the moduli space can
    be rewritten as
\begin{equation}\label{eq:dimension-fmla-2}
 4 h^{\vv} k + 4 \weyl(\ll)
     + \frac{(\dim O)}{2} \chi(\Sigma)  - (\dim G) (b^{+}-b^{1}+1),
\end{equation}
where $\weyl$ is, as above, the Weyl vector.
\end{lemma}

\begin{proof}
    The difference between this expression and the previous formula
    \eqref{eq:dimension-fmla} is the replacement of the term involving
    $2 c_{1}(\fo^{-}_{\varphi})[\Sigma]$ by the term involving
    $4\weyl(\ll)$. To see that these are equal, we may
    reduce the structure group of the $G_{\Phi}$ bundle over $\Sigma$
    to the torus $T$, and again write $\xi$ for the vector in the
    integer lattice of $T$ that classifies this $T$-bundle. The bundle
    $\fo^{-}_{\varphi}$ now decomposes as a direct sum according to
    the positive roots in $R^{+}(\Phi)$:
    \begin{equation*}
            \fo^{-}_{\varphi} = \bigoplus_{\alpha\in R^{+}(\Phi)}
            \fo^{-}_{\alpha}.
    \end{equation*}
    We have
    \[
    \begin{aligned}
    c_{1}(\fo^{-}_{\alpha})[\Sigma] &=
                      \alpha(\xi);
                      \end{aligned}
    \]
    so,
    \[
    \begin{aligned}
    c_{1}(\fo^{-}_{\varphi})[\Sigma] &=
                \sum_{\alpha\in R^{+}(\Phi)} \alpha(\xi) \\
                &=  \sum_{\alpha\in R^{+}(\Phi)} \langle
                \alpha^{\dag}, \xi\rangle.
                \end{aligned}
   \]                
    For each simple root $\beta$ in $S^{0}(\Phi)$, the Weyl group
    reflection $\sigma_{\beta}$ permutes the vectors
    \[
                  \{ \, \alpha^{\dag} \mid \alpha \in R^{+}(\Phi) \,
                  \}.
    \]
    It follows that when we write
    \[ 
                \sum_{\alpha\in R^{+}} \alpha^{\dag}
                = \sum_{\alpha\in R^{+}(\Phi)} \alpha^{\dag} +
                \sum_{\alpha\in R^{0}(\Phi)\cap R^{+}} \alpha^{\dag},
    \]
    the first term on the right is in the kernel of $\beta$ for all
    $\beta$ in $S^{0}(\Phi)$; i.e.~the first term belongs to $\fz(G_{\Phi})$.
    The second term on the right belongs to the orthogonal complement of $\fz(G_{\Phi})$,
    because it is in the span of the elements $\alpha^{\vv}$ as
    $\alpha$ runs through $S^{0}(\Phi)$. Recalling the definition of
    the Weyl vector $\weyl$, we therefore deduce
    \[
             \sum_{\alpha\in R^{+}(\Phi)}
                \alpha^{\dag}        = 2 \Pi \weyl^{\dag}.
    \]
    Thus
        \[
    \begin{aligned}
    c_{1}(\fo^{-}_{\varphi})[\Sigma] &=
                \sum_{\alpha\in R^{+}(\Phi)} \langle
                \alpha^{\dag}, \xi\rangle \\
                &= 2 \langle \Pi \weyl^{\dag} , \xi \rangle \\
                &= 2 \langle \weyl, \Pi\xi \rangle \\
                &= 2 \weyl(\ll)
                \end{aligned}
   \]
   as desired, because $\ll = \Pi\xi$.
\end{proof}

\subsection{Energy and monotonicity}

Along with the formula \eqref{eq:dimension-fmla-2} for the dimension of
the moduli space, the other important quantity is the
\emph{energy} of a solution $A$ in $\cA^{p}(X,\Sigma,P,\varphi)$ to the equations
$F^{+}_{A}=0$, which we define as
\begin{equation}\label{eq:cE-def}
\begin{aligned}
\cE &= 2 \int_{X\sminus\Sigma} | F_{A}
                |^{2}\,d\mathrm{vol} \\
                    &= 2 \int_{X\sminus\Sigma}-\tr(
                    \ad(*F_{A})\wedge\ad(F_{A})).
                    \end{aligned}
\end{equation}
(Note that  the norm on $\g_{P}$ in the first line is again defined using the Killing form,
$-\tr(\ad(a)\ad(b))$.) The reason for the factor of two is to fit with
our use of the path energy in the context of Floer homology later.

This quantity depends only on $P$ and $\varphi$, and can be calculated in
terms of the instanton number $k$ and the monopole charges $\ll$.
 To do this, we again suppose that the structure group of
 $P_{\varphi}\to \Sigma$ is reduced to the torus $T$, and we 
decompose the bundle $\fo^{+}_{\varphi}$ again as
\eqref{eq:fo-decomp}. The formula for $\cE$ as a function of $P$ and
$\varphi$ can then be written, following the argument of
\cite{KM-gtes-I}, as
\begin{equation}
\begin{aligned}
\cE(X,\Sigma,P,\varphi) & =  32\pi^{2} \Bigl( h^{\vv}k + \sum_{\beta\in
          R^{+} }\beta(\Phi) c_{1}(\fo^{-}_{\beta})[\Sigma]
          - \frac{1}{2}\sum_{\beta\in
          R^{+}}\beta(\Phi)^{2}(\Sigma\cdot\Sigma)\Bigr) \\
          &=  32\pi^{2} \Bigl( h^{\vv}k + \sum_{\beta\in
          R^{+} }\beta(\Phi)\beta(\xi)
          - \frac{1}{2}\sum_{\beta\in
          R^{+}}\beta(\Phi)^{2}(\Sigma\cdot\Sigma)\Bigr)
          \end{aligned}
          \label{eq:proto-E}
\end{equation}
where $\xi$ is again the integer vector in $\ft$ classifying the
$T$-bundle to which $P_{\varphi}$ has been reduced.
The two sums involving the positive roots in this
formula each be interpreted as half the Killing form, which leads to
the more compact formula
\begin{equation*}
\cE = 8\pi^{2} \Bigl(4 h^{\vv}k + 2 \langle
                    \Phi,\xi\rangle
          - \langle
                    \Phi,\Phi\rangle(\Sigma\cdot\Sigma)\Bigr) . 
                    \end{equation*}
Finally, using the fact that $\Phi$ belongs to $\fz(G_{\Phi})$, we can replace
$\xi$ by its projection $\ll$ and write
\begin{equation}\label{eq:energyl-fmla-2}
\cE = 8\pi^{2} \Bigl(4 h^{\vv}k + 2 \langle
                    \Phi,\ll\rangle
          - \langle
                    \Phi,\Phi\rangle(\Sigma\cdot\Sigma)\Bigr) . 
                    \end{equation}

An important comparison to be made is between the linear terms
in $k$ and $l$ in the formula for $\cE$ and the linear terms in the same variables in
the formula \eqref{eq:dimension-fmla-2} for the formal dimension of the moduli
space.
 In the dimension formula, these linear terms are
\begin{equation}\label{eq:linear-d}
            4  h^{\vv} k + 4 \weyl(\ll)
\end{equation}
while in the formula for $\cE$ they are
\begin{equation}\label{eq:linear-E}
 8\pi^{2} \Bigl(4 h^{\vv}k + 2 \langle
                    \Phi,\ll\rangle\Bigr).
\end{equation}

\begin{definition}\label{def:monotone}
    We shall say that the choice of $\Phi$ is \emph{monotone} if the
    linear forms \eqref{eq:linear-E} and \eqref{eq:linear-d}
    in the variables $k$ and $\ll$ are
    proportional.
\end{definition}

\begin{proposition}\label{prop:Einstein}
    Let $\Phi_{0}$ be any element in the fundamental Weyl chamber. Then
    there exists a unique $\Phi$ in the same Weyl chamber such that the
    stabilizers of $\Phi_{0}$ and $\Phi$ coincide, and such that
    the monotonicity condition holds. Furthermore, this $\Phi$
    satisfies the constraint \eqref{eq:highest-root-bound}.
\end{proposition}

\begin{proof}
We are seeking a $\Phi$ with $G_{\Phi}=G_{\Phi_{0}}$. The Lie algebras
of these groups therefore have the same center, in which both $\Phi$
and $\Phi_{0}$ lie, and we can write $\Pi$ for the projection of $\ft$
onto the center. 
From the formulae, we see that the monotonicity condition requires
that $\langle \Phi, \ll \rangle = 2 \weyl(\ll)$ for all $\ll$ in the
image of $\Pi$. This condition is satisfied only by the element
     \begin{equation}\label{eq:monotone-Phi}
                    \Phi= 2 \Pi(\weyl^{\dag}).
     \end{equation}
     We can rewrite this as
     \begin{equation}\label{eq:Phi-roots}
                \Phi = \sum_{\alpha\in \Roots^{+}(\Phi_{0})
                }\alpha^{\dag}.
    \end{equation}
    It remains only to verify the bound
    $\theta(\Phi)<1$. But we have
    \[
    \begin{aligned}
    \theta(\Phi)
                &< \sum_{\alpha\in\Roots^{+}}
                \theta(\alpha^{\dag}) \\
                            &= 2 \langle \theta,
                            \weyl\rangle\\
                            &= 1 - 1/h^{\vv}\\
                            &< 1
                            \end{aligned}
    \]
     by \eqref{eq:delta-theta}, as desired.
\end{proof}

\subsection{Geometric interpretation of the monotonicity condition}
                    
We can reinterpret these formulae in terms of cohomology classes on
$O$. As a general reference for the following material, we cite
\cite[Chapter 8]{Besse}.
In the first line of \eqref{eq:proto-E}, we see the
characteristic class
\begin{equation}\label{eq:char-class-1}
    \sum_{\beta\in
          R^{+} }\beta(\Phi) c_{1}(\fo^{-}_{\beta}).
\end{equation}
The decomposition of $\fo^{-}_{\varphi}$ as the sum of
$\fo^{-}_{\beta}$ reflects our reduction of the structure group to
$T$. A more invariant way to decompose this bundle is as follows. We
write $E^{+}\subset \fz(G_{\Phi})^{*}$ for the set of non-zero linear forms on $\fz(G_{\Phi})$
arising as $\beta|_{\fz(G_{\Phi})}$ for $\beta$ in $R^{+}$. The elements of
$E^{+}$ are weights for the action of $Z(G_{\Phi})$ on $\fo$, and we have
a corresponding decomposition of the vector bundle $\fo^{-}_{\varphi}$ into
weight spaces,
\[
                \fo^{-}_{\varphi} = \bigoplus_{\gamma\in E^{+}}
                \fo^{-}_{\varphi}(\gamma).
\]
The characteristic class \eqref{eq:char-class-1} can be written more
invariantly as
\begin{equation}\label{eq:char-class-2}
    \sum_{\gamma\in
          E^{+} }\gamma(\Phi) c_{1}(\fo^{-}_{\varphi}(\gamma)).
\end{equation}
The tangent bundle $TO$ has a decomposition of the same form, as a
complex vector bundle,
\[
                TO = \bigoplus_{\gamma\in E^{+}} TO(\gamma).
\]
Using the canonical trivialization of $P\to\Sigma$ up
to homotopy, interpret $\varphi$ again as a map
\[
            \varphi : \Sigma \to O,
\]
and then rewrite the characteristic class \eqref{eq:char-class-2} as
$\varphi^{*}(\Omega_{\Phi})$, where
\begin{equation}\label{eq:Omega-fmla}
                \Omega_{\Phi} =  \sum_{\gamma\in
          E^{+} }\gamma(\Phi) c_{1}(TO(\gamma)) \in H^{2}(O;\R).
\end{equation}
In this way, we can rewrite the linear form  \eqref{eq:linear-E} in
$k$ and $l$ as
\begin{equation}\label{eq:linear-E-Omega}
        32\pi^{2} \Bigl( h^{\vv}k + \bigl\langle \varphi^{*}\Omega_{\Phi},
        [\Sigma]\bigr\rangle \Bigr)
\end{equation}

Using this last expression, we see that the monotone condition
simply requires that
\begin{equation}\label{eq:G-monotone-formula}
               c_{1}(\fo^{-}) = 2 \Omega_{\Phi}
\end{equation}
in $H^{2}(O;\R)$. This equality has a geometric interpretation in
terms of the geometry of the orbit $O$ of $\Phi$ in $\g$.
Recall that
we have equipped $O$ with a complex structure $J$, so that its complex
tangent bundle is isomorphic to $\fo^{-}$. There is also the
    Kirillov-Kostant-Souriau $2$-form on  $O$, which is the
    $G$-invariant form $\omega_{\Phi}$ characterized by the condition that at
    $\Phi\in O$ it is given by
    \[
                    \omega_{\Phi} ( [U,\Phi], [V,\Phi] ) = \bigl\langle \Phi,
                    [U,V] \bigr\rangle
    \]
     where the angle brackets denote the Killing form. Together, $J$
     and $\omega_{\Phi}$ make $O$ into a homogeneous K\"ahler manifold. The
     cohomology class $[\omega_{\Phi}]$ of the Kirillov-Kostant-Souriau form
     is $4 \pi \Omega_{\Phi}$, so the monotonicity condition can be
     expressed as
     \[
                [\omega_{\Phi}] = 2\pi c_{1}(O).
     \]
     The fact that the K\"ahler class and the first Chern class are
     proportional means, in particular, that $(O,\omega_{\Phi})$ is a
     monotone symplectic manifold, in the usual sense of symplectic
     topology.
     In the homogeneous case,
     this
     proportionality (with a specified constant) between the classes
     $[\omega_{\Phi}]$ and  $c_{1}(O)$ on $O$ implies a
     corresponding relation between their natural geometric
     representatives, namely $\omega_{\Phi}$ itself and the Ricci form. That is to
     say, our monotonicity condition is equivalent to the equality
     \[
                 g_{\Phi} =\mathrm{Ricci}(g_{\Phi}) 
     \]
     for the K\"ahler metric $g_{\Phi}$ corresponding to $\omega_{\Phi}$. Thus in
     the monotone case, $O$ is a K\"ahler-Einstein manifold with
     Einstein constant $1$.

\subsection{The case of the special unitary group}

We now look at  the case of the special unitary group $\SU(N)$.
We continue to suppose
that $X\supset \Sigma$ is a pair consisting of a $4$-manifold and an
embedded surface, as in the previous subsections.
Let $P\to X$ be a given
principal $\SU(N)$ bundle. An element $\Phi$ in the standard
fundamental Weyl chamber in the Lie algebra
$\su(N)$ has the form
the form
\[
            \Phi = i \diag(\lambda_{1},\dots,
            \lambda_{1},\lambda_{2},\dots,\lambda_{2}, \dots
            ,\lambda_{m},\dots,
            \lambda_{m})
\]
where $\lambda_{1} > \lambda_{2} > \dots > \lambda_{m}$. Let the
multiplicities of the eigenspaces be $N_{1},\dots, N_{m}$, so that
$N=\sum N_{s}$. We will suppose that
\begin{equation}\label{eq:lambda-bound}
                \lambda_{1} - \lambda_{N} < 1,
\end{equation}
this being the analog of the condition \eqref{eq:highest-root-bound}.
Let $O\subset \su(N)$ be the orbit of $\Phi$ under the
adjoint action. Choose a
section $\varphi$ of the associated bundle $O_{P}|_{\Sigma}$, so
defining a reduction of the structure group of $P|_{\Sigma}$ to the
subgroup
\[
           S( U(N_{1}) \times \dots \times U(N_{m})).
\]
Let $P_{s}\to \Sigma$ be the principal $U(N+s)$-bundle arising from the
$s$'th factor in this reduction and define
\[
                l_{s} = - c_{1}(P_{s})[\Sigma].
\]
Because we have a special unitary bundle, we have
\[
\begin{aligned}
\sum N_{s}\lambda_{s} &= 0 \\
\sum l_{s} &= 0.
\end{aligned}
\]
The integers $l_{s}$ are equivalent data to the monopole charge $\ll$
in $\LL(G_{\Phi})$ as defined in Definition~\ref{def:lattice} for the
general case: more precisely, the relationship is
\[
                \ll = i \diag( l_{1}/N_{1}, \dots, l_{1}/N_{1},
                     \dots , l_{m}/N_{m}, \dots, l_{m}/N_{m} ).
\]
The Killing form is $2N$ times the standard trace norm on $\su(N)$, so
we have for example
\[
            \langle \Phi, \ll \rangle = 2N
            \sum_{s=1}^{m} \lambda_{s}l_{s}
\]
The dual Coxeter number is $N$ and the energy formula becomes:
\[
\cE(X,\Sigma,P,\varphi) = 32\pi^{2}N\Bigr(
                k +  \sum_{s=1}^{m} \lambda_{s}l_{s} -
                \tfrac{1}{2}\bigl( \sum_{s=1}^{m}
               \lambda_{s}^{2}
                N_{s}\bigr)\Sigma\cdot\Sigma \Bigr).
\]

If $E_{s}$ is the vector bundle associated to $P_{s}$ by the standard
representation, then the bundle denoted $\fo_{\varphi}^{-}$ in the
previous sections (the pull-back by $\varphi$ of the vertical tangent
bundle of $O_{P}$, equipped with its preferred complex structure) can
be written as
\[
            \fo_{\varphi}^{-} = \bigoplus_{s<t} E_{s}^{*}\otimes
            E^{t},
\]
and the first Chern class of this bundle evaluates on $\Sigma$ as
\[
            c_{1}(\fo_{\varphi}^{-})[\Sigma]
            = \sum_{s=1}^{m} \sum_{t=1}^{m} \mathrm{sign} (t-s)
            N_{t}l_{s} .
\]
The dimension formula becomes
    \begin{equation}\label{eq:SUN-dim}
                4Nk + 2 \sum_{s,t}\mathrm{sign} (t-s)
            N_{t}l_{s} + \Bigl(\sum_{s<t}N_{s}N_{t}\Bigr)\chi(\Sigma)
            - (N^{2}-1)(b^{+}-b^{1}+1).
    \end{equation}
    where  $k$ is the second Chern number of $P$.
Thus the monotone condition simply requires that
\begin{equation}\label{eq:lamdba-monotone-formula}
                \lambda_{s} =
                \frac{1}{2N} \sum_{t=1}^{m}
             \mathrm{sign} (t-s) N_{t}
\end{equation}
for all $s$. Notice that the $\Phi$ whose eigenvalues are
given by the formula \eqref{eq:lamdba-monotone-formula} with
multiplicities $N_{s}$ is already traceless, and satisfies the
requirement \eqref{eq:lambda-bound}, which we can take as confirming
Proposition~\ref{prop:Einstein} in the case of $\SU(N)$.

\begin{examples}
    (i) The simplest example occurs when $P$ is an $\SU(2)$ bundle and
    the reduction of structure group is to $U(1)$. In
    this case, we can write
    \[
                    \Phi = i
                    \begin{pmatrix}
                        \lambda & 0 \\
                        0     & -\lambda
                    \end{pmatrix}
    \]
    with $\lambda\in (0,1/2)$,
    and the holonomy of $A^{\varphi}$ along small loops linking $\Sigma$
    is asymptotically
    \[
                \exp 2\pi i \begin{pmatrix}
                        -\lambda & 0 \\
                        0     & \lambda
                    \end{pmatrix}.
    \]
    We can write $(l_{1}, l_{2})$ as $(l,-l)$ and
    formula for the index becomes
    \[
                8k + 4 l + \chi(\Sigma) - 3(b^{+}-b^{1}+1).
    \]
    The action is given by
    \[
                64 \pi^{2} ( k + 2\lambda l - \lambda^{2}
                \Sigma\cdot\Sigma).
    \]
    These are the formulae from \cite{KM-gtes-I}. The monotone
    condition requires that $\lambda=1/4$, and in this case the
    asymptotic holonomy on small loops is
    \[
            \begin{pmatrix}
                        -i & 0 \\
                        0     & i
                    \end{pmatrix}.
    \]

    (ii) The next simplest case is that of $\SU(N)$ with two
    eigenspaces, of dimensions $N_{1}=1$ and $N_{2}=N-1$. In this
    case, $O$ becomes $\CP^{N-1}$. We can again write $(l_{1}, l_{2})$
    as $(l,-l)$ and we can write
    \[
    \begin{aligned}
        \lambda_{1} &= \lambda\\
        \lambda_{2} &= - \lambda/(N-1)
    \end{aligned}
    \]
    with $\lambda$ in the interval $(0,1)$. The formula for the
    dimension is
    \[
                    4N k + 2N l + (N-1) \chi(\Sigma) -
                    (N^{2}-1)(b^{+}-b^{1}+1)
    \]
    and the energy is given by
    \[
                32 \pi^{2} N \Bigl( k + \frac{N}{N-1}\bigl(\lambda l -
                \tfrac{1}{2} \lambda^{2} \Sigma\cdot\Sigma\bigr)\Bigr).
    \]
    The monotone condition requires $\lambda=(N-1)/(2N)$, so that
\[
    \begin{aligned}
        \lambda_{1} &= (N-1)/(2N)\\
        \lambda_{2} &= -1/(2N).
    \end{aligned}
\]    
The asymptotic holonomy on small loops is
\[
                \zeta \mathop{\mathrm{diag}}(-1,1,\dots,1)
\]
where $\zeta = e^{\pi i /N}$ is a  $(2N)$'th root of unity.

\end{examples}

\subsection{Avoiding reducible solutions}
\label{subsec:reducibles}

We return to the case of a  general  $G$ (still simple and simply
connected). Recall that a connection is reducible
if its stabilizer in $\G^{p}(X,\Sigma,P,\varphi)$ has positive
dimension. This is equivalent to saying that there is a non-zero section $\psi$ of
the bundle $\g_{P}$ on $X\sminus\Sigma$ that is parallel for the
connection $A$.

Let us ask under what conditions this can occur for an $[A]$ belonging
to the moduli space $M(X,\Sigma,P,\varphi)$. For simplicity, we suppose for
the moment that $\Sigma$ is connected. If $\psi$ is parallel, then
it determines a single orbit in $\g$; we take
$\Psi\in \g$ to be a representative. The structure group of the bundle
then reduces to the subgroup $G_{\Psi}$, the stabilizer of $\Psi$, which is
a connected proper subgroup of $G$. Because $\Phi$ and $\Psi$ must
commute, we may suppose they both belong to the Lie algebra $\ft$ of
the maximal torus. As before, we shall suppose $\Phi$ belongs to the
fundamental Weyl chamber.

The center of $G_{\Psi}$ contains a torus of dimension at least one,
because $\Psi$ itself lies in the Lie algebra of the center. So there
is a non-trivial character,
\[
                s : G_{\Psi} \to U(1).
\]
Because $G_{\Psi}$ contains $T$, this character corresponds to a weight
for this maximal torus: there is an element $\w\in \ft^{*}$ in the
lattice of weights such that
\begin{equation}\label{eq:character}
                s( \exp (2\pi x)) = \exp(2\pi\w( x))
\end{equation}
for $x$ in $\ft$.
The fundamental group of $T$ maps
onto that of $G_{\Psi}$, so we may assume that $\w$ is a primitive
weight (i.e. is not a non-trivial multiple of another integer weight).

In addition to taking $\w$ to be primitive, we can further narrow
down the possibilities as follows. Suppose first that
$\Psi$ lies (like $\Phi$) in the fundamental Weyl chamber. In the
complex
group $G^{c}$, there are $r=\rank(G)$
different maximal parabolic subgroups which contain the standard Borel
subgroup corresponding to our choice of positive roots. These maximal
parabolics are indexed by the set of simple roots $\Sroots^{+}$; we let
$G(\alpha)$, for $\alpha\in \Sroots^{+}$, denote the intersection of these
groups with the compact group $G$. Each group $G(\alpha)\subset G$ has
$1$-dimensional center, and the weight $\w$ corresponding to a
primitive character of $G(\alpha)$ is the fundamental weight
$\w_{\alpha}$.
 The group
$G_{\Psi}$ lies inside one of the $G(\alpha)$, so the same fundamental
weight
$\w_{\alpha}$ defines a character of $G_{\Psi}$.
If $\Psi$ does not lie in the fundamental Weyl chamber, then we
need to apply an element of the Weyl group $\mathcal{W}$; so in general, we can
always take $\w$ to have the form
\[
            \w = \w_{\alpha} \comp \sigma
\]
where $\sigma\in \mathcal{W}$ and $\w_{\alpha}$ is one of the
fundamental weights.

Applying the character $\w$ to the singular connection  $A$, we obtain
a singular $U(1)$ connection on $X\sminus\Sigma$, carried by the
bundle obtained by applying $s$ to the principal $G_{\Psi}$-bundle
$P_{\psi}\to X$.  That is, we have a $U(1)$ connection differing by
terms of regularity $L^{p}_{1,A}$ from the singular model connection
\[
                \nabla + i \beta(r) \w(\Phi) \eta.
\]
Here it is important that $\Sigma$ is connected: what we have really
done, is picked a base-point near $\Sigma$, and used the values of
$\phi$ and $\psi$ at that base-point to determine $T$ and $s$.
When we apply the Chern-Weil to obtain an expression for $c_{1}$ of
the line bundle $s(P_{\psi})$, we obtain an additional contribution
from the singularity, equal to the Poincar\'e dual of
\[
            2\pi \w(\Phi) [\Sigma].
\]
More precisely, if $F$ denotes the curvature of this $U(1)$ connection
on $X\sminus\Sigma$ (as an $L^{p}$ form, extended by zero to all of
$X$), then
\[
            \frac{i}{2\pi}[F] = c_{1}(s(P_{\psi})) -  \w(\Phi)
            \mathrm{P.D.}[\Sigma].
\]
If $\Sigma$ has more than one component, say
$\Sigma=\Sigma_{1}\cup\dots\Sigma_{r}$, then the only change is that
we will see a different element of the Weyl group for each component
of $\Sigma$: so if we have a reducible solution then there will be a
fundamental weight $w_{\alpha}$ and elements
$\sigma_{1},\dots,\sigma_{r}$ in $\mathcal{W}$ such that the 
curvature $F$ of the corresponding $U(1)$ connection satisfies
\begin{equation}\label{eq:de-Rham-class}
            \frac{i}{2\pi}[F] = c_{1}(s(P_{\psi})) -
            \sum_{j=1}^{r}(\w_{\alpha}\comp\sigma_{j})(\Phi)
            \mathrm{P.D.}[\Sigma_{j}].
\end{equation}

Because the connection $A$ is anti-self-dual, the $2$-form $F$ is also
anti-self-dual, and is therefore $L^{2}$-orthogonal to every closed,
self-dual form $h$ on $X$. By the usual argument
\cite{Donaldson-omega-argument}, we deduce:

\begin{proposition}\label{prop:no-reducibles}
    Suppose that $b^{+}(X)\ge 1$, and let the components of $\Sigma$
    be $\Sigma_{1},\dots,\Sigma_{r}$. Suppose that for every fundamental
    weight $\w_{\alpha}$ and every choice of elements
    $\sigma_{1},\dots,\sigma_{r}$ in the Weyl group, the real
    cohomology class
\begin{equation}\label{eq:is-non-integral}
            \sum_{j=1}^{r}(\w_{\alpha}\comp\sigma_{j})(\Phi)
            \mathrm{P.D.}[\Sigma_{j}]
\end{equation}
    is not integral. Then
   for generic choice of Riemannian metric on $X$, there are
       no reducible solutions in the moduli space
       $M(X,\Sigma,P,\varphi)$.
\end{proposition}

\begin{examples}
As an illustration,  in the case $G=\SU(N)$,
if $\Sigma$ is connected and $[\Sigma]$ is primitive, then
for \eqref{eq:is-non-integral} to be an integral class
means that the sum of some proper subset of the eigenvalues
of $\Phi$ (listed with repetitions) is equal to an integer multiple of
$i$. 
If there are only two distinct eigenvalues $i\lambda_{1}$ and
$i\lambda_{2}$ of multiplicities $N_{1}$ and $N_{2}$, and if we are in
the monotone case, so that
\[
\begin{aligned}
    \lambda_{1} = N_{2}/(2N)\\
    \lambda_{2} = -N_{1}/(2N),
\end{aligned}
\]
then this integrality means that
\[
                (2N) \bigm| (a N_{2} - b  N_{1})
\]
for some non-negative with $a\le N_{1}$, $b\le N_{2}$ and $0 < a+b <
N$. This cannot happen if $N_{1}$ and $N_{2}$ are coprime.

Another family of examples satisfying the monotone condition occurs
when $\Phi=2\weyl^{\dag}$ (the case where the reduction of structure
group is to the maximal torus, $T\subset G$) and the group $G$ is
small. For example, if $\Phi=2\weyl^{\dag}$ and $G$ is either the
group $G_{2}$ or the simply-connected group of type $B_{2}$, $B_{3}$
or $B_{4}$, then \eqref{eq:is-non-integral} is never an integer.
Thus we have:

\begin{corollary}\label{cor:irreducibles-examples}
    Suppose that $b^{+}(X)$ is positive and that $\Sigma$ is connected and lies
    in a primitive homology class. Suppose $\Phi$ is chosen to satisfy
    the monotone condition and that we are in one of the following cases:
    \begin{enumerate}
        \item $G$ is the group $\SU(N)$, and $\Phi$ has two distinct
        eigenvalues whose multiplicities are coprime.

        \item $G$ is the group $G_{2}$ and $\Phi$ is the regular
        element $2\weyl^{\dag}$.

        \item \label{item:G2-irreducible}
        $G$ is the group $\Spin(5)$, $\Spin(7)$ or $\Spin(9)$
        and $\Phi$ is $2\weyl^{\dag}$.
    \end{enumerate}
    Then for generic choice of
    Riemannian metric on $X$, there are no reducible solutions in the
    moduli space
    $M(X,\Sigma,P,\varphi)$.
\end{corollary}

\begin{proof}
    We illustrate the calculation in the case of $\Spin(9)$. The
    $B_{4}$ root system is the following collection of integer vectors
    in $\R^{4}$: the vectors $\pm e_{i} \pm e_{j}$ for $i\ne j$, and
    the vectors $\pm e_{i}$. The simple roots are
    \[
    \begin{aligned}
    \alpha_{i} &= e_{i}-e_{i+1},\qquad (i=1,2,3) ,\\
                \alpha_{4} &= e_{4}
                \end{aligned}
                \]
    and the fundamental weights are
    \[
    \begin{aligned}
    w_{1} & = (1,0,0,0) \\
                    w_{2} & = (1,1,0,0) \\
                    w_{3} & = (1,1,1,0) \\
                     w_{4} & = (1/2, 1/2,1/2,1/2) .
                     \end{aligned}
    \]
    The orbit of these under the Weyl group consists of all vectors of
    the form
    \begin{equation}\label{eq:orbits-weights}
             \begin{gathered}
             \pm e_{i}\\
                \pm e_{i} \pm e_{j} \\
                \pm e_{i} \pm e_{j} \pm e_{k} \\
                (1/2)(\pm e_{1} \pm e_{2} \pm e_{3} \pm e_{4})                
             \end{gathered}
    \end{equation}
    with $i,j,k$ distinct.
    We identify $\ft$ with its dual using the Euclidean inner product
    on $\R^{4}$, so that the coroots are $\alpha^{\vv}_{i}=\alpha_{i}$
    for $i=1,2,3$ and $\alpha^{\vv}_{4}=2\alpha_{4}$;
    and we note that the Killing form on $\ft$ is $14$
    times the Euclidean inner product, because the dual Coxeter number
    is $7$. Then we calculate the sum of the positive roots to obtain
    \[
                2\weyl = (7,5,3,1)
    \]
    and we deduce
    \[
                2\weyl^{\dag} = (1/14) (7,5,3,1).
    \]
    It is straightforward to see that $2\weyl^{\dag}$ does not have
    integer pairing with any of the vectors in
    \eqref{eq:orbits-weights}, with respect to the Euclidean inner
    product on $\R^{4}$.
\end{proof}

The examples in this corollary are not meant to be exhaustive: the
authors have not attempted a  complete classification of the monotone
cases. Note that we cannot extend the above proof for $\Spin(9)$ to
the case of $\Spin(11)$ (the $B_{5}$ Dynkin diagram) because the
vector $2\weyl^{\dag}$ is then \[2\weyl^{\dag}(1/18)(9,7,5,3,1)\] which has inner
product $0$ in $\R^{5}$ with the vector $(1,0,-1,-1,-1)$, which
belongs to the Weyl orbit of the fundamental weight $(1,1,1,1,0)$.

\end{examples}

\begin{remark}
    We recall again from section~\ref{subsec:connections-and-moduli} that
    in cases other than $\SU(N)$, it is possible for a connection to
    be irreducible and yet have finite stabilizer strictly larger than
    the center of the group $G$. The examples of this phenomenon that
    were illustrated previously include the case that $G$ is
    $\Spin(2n+1)$ as well as the case of $G_{2}$.  Note that these
    examples include the examples mentioned in part
    \ref{item:G2-irreducible} of
    Corollary~\ref{cor:irreducibles-examples}; so although we can
    avoid stabilizers of positive dimension in those cases, we will
    still be left with finite stabilizers larger than the center.
\end{remark}

The examples in Corollary~\ref{cor:irreducibles-examples} are cases
where $\Sigma$ is connected. The next corollary
exhibits an interesting case where
Proposition~\ref{prop:no-reducibles} can be applied to a disconnected
$\Sigma$:

\begin{corollary}\label{cor:irreducibles-comp}
    Suppose $b^{+}(X)$ is positive.
    Let $G=\SU(N)$ and 
    suppose $\Sigma$ has $N+1$ components, all belonging to the same
    primitive homology class. Let $\Phi$ be the element
    \[
                    \Phi = \left(\frac{i}{2N}\right)\diag\bigl((N-1),
                    -1,\dots,-1\bigr)
    \]
    so that the monotone condition holds. Then for a generic choice
    of Riemannian metric on $X$, the moduli space
    $M(X,\Sigma,P,\varphi)$ contains no reducible solutions.
\end{corollary}

\begin{proof}
    From Proposition~\ref{prop:no-reducibles}, we see that we must
    check that the rational number
    \[
                            \sum_{j=1}^{N-1}(\w_{\alpha}\comp\sigma_{j})(\Phi)
    \]
    is never an integer. As a function on the maximal torus,
    the fundamental weight $\alpha$  can be taken to be the sum of the
    first $k$ eigenvalues for some $k$ with $1\le k \le N-1$, and so
    \begin{equation}\label{eq:weyl-cases}
            (\w_{\alpha}\comp\sigma_{j})(\Phi)
            =
            \begin{cases}
                -k/(2N),&\text{or}\\
                (N-k)/(2N),&
            \end{cases}
    \end{equation}
    according to which Weyl group element $\sigma_{j}$ is involved.
    The above sum is therefore
    \[
                   (s/2) - k(N-1)/(2N)
    \]
    where $k$ depends on the choice of $\alpha$, and $s$ is the number
    of components for which the second case of \eqref{eq:weyl-cases}
    occurs. This quantity differs from an element of $(1/2)\Z$ by
    $k/(2N)$, so it cannot be an integer.
\end{proof}

\subsection{Bubbles}

Uhlenbeck's compactness theorem for instanton moduli spaces on a
closed $4$-manifold $X$ carries over to the case of instantons with
codimension-2 singularities along a surface $\Sigma\subset X$.
In the case of $\SU(2)$, the proof can
again be found in \cite{KM-gtes-I}. The proof carries over without
substantial change to the case of a general group. We shall state
here the version appropriate for a simply-connected simple Lie group
$G$ and a moduli space $M(X,\Sigma,P,\varphi)$ of anti-self-dual
connections with reduction along $\Sigma$.

\begin{proposition}\label{prop:Uhlenbeck}
    Let $[A_{n}]$ be a sequence of gauge-equivalence classes of
    connections in the moduli space $M(X,\Sigma,P,\varphi)$. Then,
    after replacing this sequence by a subsequence, we can find a
    bundle $P'\to X$, a section $\varphi'$ of the bundle
    $O_{P'}\to\Sigma$ defining a reduction of structure group to the
    same subgroup $G_{\Phi}\subset G$,  an element $[A]$ in
    $M(X,\Sigma,P',\varphi')$ and a finite set of point
    $\mathbf{x}\subset X$ with the following properties.
    \begin{enumerate}
        \item There is a sequence of isomorphisms of bundles $g_{n}:
        P'|_{X\sminus\mathbf{x} } \to P$ such that
        $g_{n}^{*}(\varphi)=\varphi'|_{\Sigma\sminus\mathbf{x}}$ and such
        that
        \[
                            g_{n}^{*}(A_{n}) \to
                            A|_{X\sminus\mathbf{x}}
        \]
        on compact subsets of $X\sminus\mathbf{x}$.

        \item In the sense of measures on $X$, the energy densities $2
        |F_{A_{n}}|^{2}$ converge to
        \[
                           2 |F_{A}|^{2} + \sum_{x\in
                           \mathbf{x}}\mu_{x}
                           \delta_{x}
        \]
        where $\delta_{x}$ is the delta-mass at $x$ and $\mu_{x}$ are
        positive real numbers.

        \item\label{item:charge-change}
        For each $x\in \mathbf{x}$, we can find an integer $k_{x}$
        and an  $\ll_{x}$ in the lattice
        $\LL(G_{\Phi})\subset\fz(G_{\Phi})$ such that
        \[
                        \mu_{x} = 8\pi^{2}(4 h^{\vv} k_{x} +
                        2 \langle \Phi, \ll_{x} \rangle )
        \]
        If $x\not\in\Sigma$, we can take $\ll_{x}=0$ here. Furthermore, if
        $(k,\ll)$ and $(k',\ll')$ are the instanton numbers and
        monopole charges for $(P,\varphi)$ and
        $(P',\varphi')$ respectively, then we can arrange that
        \[
                    \begin{aligned}
                        k &= k' + \sum_{x\in\mathbf{x}} k_{x} \\
                        \ll &= \ll' + \sum_{x\in\mathbf{x}} \ll_{x} \\
                    \end{aligned}
        \]

        \item \label{item:on-S4}
        For each such pair $(k_{x}, \ll_{x})$ with $x\in\Sigma$
        we can find an
        expression for these as finite sums,
        \[
\begin{aligned}
           k_{x} &= k_{x,1} + \dots + k_{x,m}\\
            \ll_{x} &= \ll_{x,1} + \dots + \ll_{x,m}
\end{aligned}
        \]
         and solutions $[A_{x,i}]$ in moduli spaces $M(S^{4},S^{2},
         P_{x,i}, \varphi_{x,i})$ for the round metric on $(S^{4},
         S^{2})$, where $P_{x,i}$ is the $G$-bundle on $S^{4}$ with
         $k(P_{x,i})=k_{x,i}$ and $\varphi_{x,i}$ is the reduction of
         structure group along $S^{2}$ classifies by the element
         $\ll_{x,i})$ in $\LL(G_{\Phi})$.        
    \end{enumerate}
\end{proposition}

The content of the last three parts of the proposition is that the
energy $\mu_{x}$ that is ``lost'' at each of the point $x$ in
$\mathbf{x}$ is accounted for by the energy of a collection of
solutions on $(S^{4},S^{2})$ that have bubbled off. (The expression
\begin{equation}\label{eq:energy-S4}
            8\pi^{2}( 4h^{\vv} k + 2\langle \Phi,\ll \rangle)
\end{equation}
is the formula for the energy in the case of $(S^{4},S^{2})$.) In
general, if no multiple of $\Phi$ is an integer point, then the set of
values realized by this function of $k$ and $\ll$ is dense in the real
line; and while the proof of Uhlenbeck's theorem does provide us with
a constant $\eta$ and a guarantee that $\mu_{x}\ge \eta$ in all cases,
we need better information than this to make use of the compactness
theorem in applications. For example, the statement of the result as
given does not guarantee that the formal dimension of
$M(X,\Sigma,P,\varphi)$ is not larger than that of
$M(X,\Sigma,P',\varphi')$.

The essential matter is to know which pairs $(k,\ll)$ in $\Z
\times \LL(G_{\Phi})$ are realized by solutions on $(S^{4},S^{2})$.

\begin{proposition}\label{prop:energy-inequalities}
    Let $\Phi$ as usual lie in the fundamental Weyl chamber and
    satisfy the necessary constraint $\theta(\Phi)<1$.
    Then for any solution $[A]$ on $(S^{4},S^{2})$
    with the round metric, the corresponding topological invariants
    $(k,\ll)$ in $\Z\times \LL(G_{\Phi})$ must satisfy the inequalities
    \begin{equation}\label{eq:inequality-1}
                    k \ge 0
    \end{equation}
    and
    \begin{equation}\label{eq:inequality-2}
                    n^{\vv}_{\alpha}k + \w_{\alpha}(\ll) \ge 0
    \end{equation}
    for all simple roots $\alpha$.
\end{proposition}

\begin{remark}
    We will see in the course of the proof that the above inequalities
    are equivalent to a smaller set, namely the set consisting of the
    inequality
    \eqref{eq:inequality-1} together with the inequalities
    \eqref{eq:inequality-2} taken \emph{only} for those $\alpha$
    belonging to the set of simple roots $\alpha$ in $S^{+}(\Phi)$
    (the simple roots which are positive
    on $\Phi$).
\end{remark}

Before proving the proposition, we note an important corollary for the
formal dimensions of the non-empty moduli spaces on $(S^{4},S^{2})$.
The dimension formula in this case can be written
\[
                4 h^{\vv} l + 4 \langle \weyl^{\dag}, \ll \rangle -
                \dim G_{\Phi}.
\]
We can interpret the first two terms
\begin{equation}\label{eq:framed-dim}
  4 h^{\vv} l + 4 \langle \weyl^{\dag}, l \rangle
\end{equation}
as the dimension of a  framed moduli space, as follows. The gauge
group $\G^{p}(X,\Sigma, P,\varphi)$ consists of continuous
automorphisms of the bundle $P\to S^{4}$ which preserve the section
$\varphi$ of the adjoint bundle; so if we pick a point $s\in
S^{2}\subset S^{4}$ then there is  a closed subgroup $\G^{p}_{1}$
consisting of elements with $g(s)=1$. The formula
\eqref{eq:framed-dim} can be interpreted as the formal dimension of a
moduli space $\tilde{M}(S^{4}, S^{2})$ where we divide the space of
anti-self-dual singular connections by the smaller group $\G^{p}_{1}$
instead of the full gauge group. We then have:

\begin{corollary}\label{cor:at-least-4}
    For any non-empty moduli space on $(S^{4},S^{2})$ with the round
    metric, the corresponding instanton number $k$ and monopole charge are
    either both zero (in which case the moduli space contains only the flat
    connection) or satisfy
    \[
                    4 h^{\vv} k + 4 \langle \weyl^{\dag}, \ll \rangle
                    \ge 4.
    \]
\end{corollary}

\begin{proof}[Proof of Corollary~\ref{cor:at-least-4}]
    We recall the relation \eqref{eq:sum-is-coxeter} and the fact that
    $\weyl$ is the sum of the $\w_{\alpha}$ (taken over all
    simple roots $\alpha)$. Using these, we see that the sum of all the
    inequalities in Proposition~\ref{prop:energy-inequalities} gives
    us
    \[
                    h^{\vv} k +  \langle \weyl^{\dag}, \ll \rangle
                    \ge 0.
    \]
    To refine this a little, let us break up the sum into two parts
    according to whether $\alpha$ lies in $S^{+}(\Phi)$ or
    $S^{-}(\Phi)$: we obtain
    \[
    \begin{aligned}
        4 h^{\vv} k + 4 \langle \weyl^{\dag}, \ll \rangle
                    &=
                    4k + 4 \sum_{\alpha\in S^{+}(\Phi)}\bigl( h^{\vv}
                    k +  \langle \w_{\alpha}^{\dag}, \ll \rangle \bigr)
                   + 4 \sum_{\beta\in S^{0}(\Phi)}\bigl( h^{\vv} k +
                    \langle \w_{\beta}^{\dag}, \ll \rangle \bigr) \\
                    & \ge
                    4k + 4 \sum_{\alpha\in S^{+}(\Phi)}\bigl( h^{\vv}
                    k +  \langle \w_{\alpha}^{\dag}, \ll \rangle
                    \bigr)  .                  
    \end{aligned}
    \]
    As well as being non-negative by the proposition, the terms under
    the final summation sign are all integers: this is because $\ll$
    is the projection in $\fz(G_{\Phi})$ of an integer vector $\xi\in
    \ft$ and $\xi-\ll$ lies in the kernel of $\w_{\alpha}$ for
    $\alpha$ in $S^{+}(\Phi)$. (The similar terms involving the
    $\beta$ in $S^{0}(\Phi)$ need not be integers.)

    This shows that $4 h^{\vv} k + 4 \langle \weyl^{\dag}, \ll
    \rangle$ is at least $4$ unless $k$ is zero and
    $\w_{\alpha}(\ll)$ is zero for all $\alpha$ in $S^{+}(\Phi)$.
    These are independent linear conditions which imply that $k$ and
    $\ll$ are both zero.
\end{proof}

\begin{corollary}\label{cor:lose-4}
    In the situation of Proposition~\ref{prop:Uhlenbeck}, if the set
    of bubble-points $\mathbf{x}$ is non-empty, then the formal
    dimension of the moduli space $M(X,\Sigma,P',\varphi')$
    is smaller than the dimension of
    $M(X,\Sigma,P,\varphi)$, and the difference is at least $4$.
\end{corollary}

\begin{proof}[Proof of Corollary~\ref{cor:lose-4}]
     The difference in the dimensions is equal to a sum
     \[
          \sum_{x}\sum_{i=1}\left(
          4 h^{\vv} k_{x,i} + 4 \langle \weyl^{\dag}, l_{x,i}
          \rangle\right)
     \]
     and each term is at least four by the previous corollary and the
     condition in part \ref{item:on-S4} of
     Proposition~\ref{prop:Uhlenbeck}.
\end{proof}

\begin{proof}[Proof of Proposition~\ref{prop:energy-inequalities}]
    The proof rests on a theorem of Munari \cite{Munari} which provides
    a correspondence between moduli spaces of singular instantons in
    $(S^{4},S^{2})$ and certain complex-analytic moduli spaces for
    holomorphic data on $\CP^{2}$. To state Munari's theorem, fix
    $\Phi$ as usual, let $P\to S^{4}$ be a $G$-bundle and $\varphi$ a
    section of $O_{P} \to S^{2}$ defining a  reduction of structure
    group. Pick a point $s\in S^{2}$ and let $\tilde{M}(S^{4}, S^{2},
    P,\varphi)$ be the corresponding framed moduli space. Let $\pi :
    \CP^{2} \to S^{4}$ be a map which collapses the line at infinity
    $\ell_{\infty}\subset \CP^{2}$ to the point $s$ and which maps
    another complex line $\Sigma$ to $S^{2}\subset S^{4}$. Write
    $s_{\infty}\in \CP^{2}$ for the point where $\Sigma$ and
    $\ell_{\infty}$ meet. Then we
    have:
    \begin{theorem}[\cite{Munari}; see also \cite{Biquard}]\label{thm:Munari}
        There is a bijection between the moduli space of singular
        anti-self-dual connections $\tilde{M}(S^{4}, S^{2},
        P,\varphi)$ on the one hand, and on the other, the set of
        isomorphism classes of collections $(\mathcal{P},
        \psi, \tau)$ where
        \begin{itemize}
            \item $\mathcal{P}\to \CP^{2}$ is a holomorphic principal
            $G^{c}$-bundle topologically isomorphic to $\pi^{*}(P)$,
            \item $\psi : \Sigma \to O_{\mathcal{P}}$ is a holomorphic
            section of the associated bundle on $\Sigma=\CP^{1}$ with
            fiber $O$, homotopic to the section $\pi^{*}(\varphi)$.
            \item $\tau$ is a  holomorphic trivialization of the
            restriction of $\mathcal{P}$ to $\ell_{\infty}$,
            satisfying the constraint that the induced trivialization
            of the adjoint bundle carries $\psi(s_{\infty})$ to
            $\Phi$.
        \end{itemize}
    \end{theorem}

    A special case of this theorem, which may help to understand the
    statement, is the case that $k=0$ and the bundle $P$ on $S^{4}$ is
    trivial. In this case $\mathcal{P}$ is topologically trivial; and
    the trivialization on the line at infinity forces $\mathcal{P}$ to
    be analytically trivial also, so that $\tau$ extends uniquely to
    a holomorphic trivialization of $\mathcal{P}\to\CP^{2}$. The data
    $\psi$ then becomes a based rational map: a holomorphic map from
    $\Sigma = \CP^{1}$ to $O$ sending $s_{\infty}$ to $\Phi$.

    Staying with this special case, the inequalities of
    Proposition~\ref{prop:energy-inequalities} have a straightforward
    interpretation. For a holomorphic map $\psi$ from
    $\CP^{1}$ to $O$, the pairing of $\psi(\CP^{1})$ with any class in
    the closure of  the K\"ahler
    cone of $O$ must be non-negative. The inequalities of the
    proposition when $k=0$ can be seen as consequences of this
    statement. This is essentially the same argument that was used by
    Murray \cite{Murray-1} to constrain the possible charges
    of monopoles on $\R^{3}$.

    For the general case, the strategy is similar, but we use the
    energy $\cE$ of the anti-self-connection, rather than the energy
    of a holomorphic map. The essential point, which the following immediate
    corollary of Theorem~\ref{thm:Munari} above:

    \begin{corollary}
        Suppose $\Phi$ and $\Phi_{1}$ are two elements of the
        fundamental Weyl chamber with the same stabilizer, so that
        $\fz(G_{\Phi})=\fz(G_{\Phi_{1}})$. Suppose both satisfy the
        constraint \eqref{eq:highest-root-bound}.
        Then $M(S^{4},S^{2},P,\varphi)$ is homeomorphic to
        $M(S^{4},S^{2},P,\varphi_{1})$ when
        $\varphi$ and $\varphi_{1}$ are sections of the bundles
        associated to the adjoint action of $G$ on the orbits of
        $\Phi_{1}$ and $\Phi_{2}$ respectively, with the same homotopy
        class. In particular, one of these moduli spaces is non-empty
        if and only if the other is.
    \end{corollary}

    To apply this corollary, suppose that a moduli space $M(S^{4},
    S^{2}, P,\varphi)$ is non-empty. Let $k\in\Z$ and $\ll \in
    \LL(G_{\Phi})\subset \fz(G_{\Phi})$ be the topological invariants of $P$
    and $\varphi$. Let $\Alcove$ be the alcove in $\ft$ defined as the
    intersection of the fundamental Weyl chamber with the half-space
    $\theta \le 1$, where $\theta$ is the highest root. Thus $\Alcove$
    is a closed simplex. The intersection $\Alcove\cap \fz(G_{\Phi})$ is a
    simplex with possibly smaller dimension. In applying the
    corollary, the admissible values for $\Phi_{1}$ are precisely the
    interior points of the simplex $\Alcove\cap \fz(G_{\Phi})$.
    A necessary condition for  a moduli space to be non-empty is that
    the associated topological energy
    $\cE(S^{4},S^{2},P,\varphi_{1})$ is non-negative, so the corollary
    tells us that
    \begin{equation}\label{eq:simplex-inequality}
                        2 h^{\vv} k + \langle \psi ,
                        \ll
                        \rangle \ge 0
    \end{equation}
    for all interior points of $\Alcove\cap \fz(G_{\Phi})$, and hence for
    all points in the closed simplex $\Alcove\cap \fz(G_{\Phi})$, by
    continuity. If $\Pi : \ft \to \fz(G_{\Phi})$ again denotes the
    orthogonal projection, then it is a fact about the geometry of
    $\Alcove$ that
    \[
                    \Pi (\Alcove) \subset \Alcove \cap \fz(G_{\Phi}).
    \]
    As $\ll$ itself lies in $\fz(G_{\Phi})$, we deduce that the
    inequality \eqref{eq:simplex-inequality} holds not just for $\psi$
    in $\Alcove\cap \fz(G_{\Phi})$, but for all $\psi$ in $\Alcove$.

    The vertices of the simplex $\Alcove$ are the point $0$ and the
    points $\psi =
    \w_{\alpha}^{\dag}/\theta(\w_{\alpha}^{\dag})$, as
    $\alpha$ runs through the simple roots. Applying
    \eqref{eq:simplex-inequality} with $\psi$ at these vertices, we
    obtain $k\ge0$ and
    \[
                            2 h^{\vv}\theta(\w_{\alpha}^{\dag}) k
                            + \w_{\alpha}(\ll) \ge 0
    \]
    for all simple roots $\alpha$.
    To complete the proof of the inequality \eqref{eq:inequality-2},
    we calculate, using \eqref{eq:theta-length} and the definition of
    the coroots,
    \[
\begin{aligned}
                    2 h ^{\vv}\theta(\w_{\alpha}^{\dag}) &=
                      2 \langle \theta, \theta\rangle^{-1}
                      \w_{\alpha}(\theta^{\dag}) \\
                      &=\w_{\alpha}(\theta^{\vv}) \\
                      &= n^{\vv}_{\alpha}.
\end{aligned}
    \]
    This completes the proof of the proposition.
\end{proof}

\begin{example}
We illustrate the $\SU(N)$ case. Arrange the eigenvalues
of $\Phi$ as usual, as $i\lambda_{1}$, \dots $i\lambda_{m}$ with
$\lambda_{1}> \dots >\lambda_{m}$ and $\lambda_{1}-\lambda_{m}<1$.
Let $N_{s}$ be the multiplicity of
the eigenspace for $\lambda_{s}$, so that $\varphi$ defines a reduction of
$P|_{S^{2}}$ to the subgroup
\[
            S(U(N_{1}) \times \dots \times U(N_{m})).
\]
Let $k$ be $c_{2}(P)[S^{4}]$, and let
$l_{1},\dots, l_{m}$ be the first Chern numbers,
\[
            l_{s} = -c_{1}(E_{s})[S^{2}]
\]
where $E_{s}$ is the associated $U(N_{s})$ bundle. 
We have $\sum l_{s} =0$. Then the inequalities of
Proposition~\ref{prop:energy-inequalities}, taken just for the extreme
cases when $\alpha$ is in $S^{+}(\Phi)$, become
\begin{equation}\label{eq:monopole-number-inequalities}
\begin{gathered}
     k \ge 0 \\
     k + l_{1} \ge 0 \\
     \dots  \\
     k + l_{1} + l_{2} + \dots + l_{m-1} \ge 0.
\end{gathered}
\end{equation}
Note that the first inequality can also be written as the
non-negativity of $k + l_{1} + \dots + l_{m}$, because the $l_{s}$ add
up to
zero. Let us write
\[
                K_{s} = k + \sum_{t < s} l_{t},
\]
so that the above inequalities assert $K_{s}\ge0$.
Then we observe that the formal dimension of the framed moduli space,
given by the formula \eqref{eq:SUN-dim}, can be written
as
\[
            2(N_{m} + N_{1}) K_{1} + 2 (N_{1}+N_{2}) K_{2} + \dots +
            2(N_{m-1}+N_{m}) K_{m}.
\]
This is bounded below by
\[
                    4 ( K_{1} + \dots + K_{m}).
\]
In particular the dimension of the moduli space is at least $4$,
unless $k$ and the $l_{s}$ are all zero. Slightly more precisely, we
can state:

\begin{corollary}\label{cor:at-least-2N}
For $\Phi$ as above, and $G=\SU(N)$,
    the minimum possible formal dimension of any non-empty framed moduli space of
    positive formal dimension on
    $(S^{4}, S^{2})$ is
    \[
                \min \{ \, 2(N_{s-1} + N_{s}) \mid s=1,\dots,  m \,\}
    \]
   where we interpret $N_{0}$ as a synonym for $N_{m}$. In particular,
   no moduli space has dimension less than $4$, except for the trivial
   zero-dimensional moduli space; and in the special case that there
   are only two distinct eigenvalues, the smallest
   positive-dimensional moduli space has dimension $2N$.
\end{corollary}

\end{example}

\subsection{Orbifold metrics and connections}
\label{subsec:orbifold-metrics}

Up until this point, we have considered a moduli space
$M(X,\Sigma,P,\varphi)$ of singular instantons defined using a space of
connections $\cA^{p}(X,\Sigma,P,\varphi)$ modeled on an $L^{p}_{1}$ Sobolev
space, with $p$ a little bigger than $2$. There are disadvantages
associated with having to use such a weak Sobolev norm: for example,
these connections $A$ are not continuous, which creates difficulties
if we want to use holonomy perturbations later. There is also a
difficulty with proving the sort of vanishing theorems that are
usually used to show that the moduli spaces of solutions on $S^{4}$,
for example, are smooth.

Something that was exploited in \cite{KM-gtes-I} is that we \emph{can} use
stronger Sobolev norms if we first make a slight change to the
geometry of our picture. We will explain this here.

We shall equip $X$ with a singular metric $g^{\nu}$ which has an
orbifold-type singularity along the surface $\Sigma$, with cone-angle
$2\pi/\nu$ for some integer $\nu>0$. This means that at each point of
$\Sigma$ there is a neighborhood $U$ such that $(U\sminus\Sigma,
g^{\nu})$ is isometric to the quotient of a smooth Riemannian
manifold by a cyclic group of order $\nu$: the model for such a metric
in the flat case is the metric
\[
            du^{2} + dv^{2} + dr^{2} + \left( \frac{r^{2}}{\nu^{2}}
            \right) d\theta^{2}.
\]
As motivation, if $\Phi$ is an element of $\ft\subset\g$ with the
property that $\nu\Phi$ is integral, then our model singular
connection $A^{\varphi}$ from \eqref{eq:singular-model-A} can be
constructed so that it becomes a
smooth connection on passing to  $\nu$-fold branched cover; so if we
use the metric $g^{\nu}$, then we can regard $A^{\varphi}$ as an
orbifold connection, and we can reinterpret it as being a smooth
connection in the orbifold sense.

We shall not use the orbifold language here, except in referring to
the metric $g^{\nu}$ as having an ``orbifold singularity''. Also,
when using the metric $g^{\nu}$, we shall not require that $\nu\Phi$
be integral. What we will exploit is that, by making $\nu$ sufficiently
large, the ``Fredholm package'' that is used in constructing the
moduli spaces can be made to work in Sobolev spaces with any desired
degree of regularity. More precisely, let $A^{\varphi}$ be the model
singular connection on $(X,\Sigma)$ equipped with the metric
$g^{\nu}$, and let $d_{A^{\varphi}}^{+}$ be the linearized
anti-self-duality operator acting on $\g_{P}$-valued $1$-forms,
defined using the metric $g^{\nu}$. On differential forms on
$X\sminus\Sigma$, define the norms $\check{L}^{p}_{k,A^{\varphi}}$
using the Levi-Civita derivative of $g^{\nu}$ and the covariant
derivative of $A^{\varphi}$ on $\g_{P}$. Then let
$\mathcal{D}_{\varphi}$ be the operator
\begin{equation}\label{eq:D-4d}
                \mathcal{D}_{\varphi} = -d^{*}_{A^{\varphi}} \oplus
                d^{+}_{A^{\varphi}}
\end{equation}
acting on the spaces
\begin{equation}\label{eq:check-Sobolev}
                \check{L}^{p}_{k,A^{\varphi}}(X\sminus\Sigma, \g_{P}\otimes
                \Lambda^{1})
                \to
                \check{L}^{p}_{k-1,A^{\varphi}}(X\sminus\Sigma, \g_{P}\otimes
                (\Lambda^{0}\oplus\Lambda^{+}))               
\end{equation}
Then $\mathcal{D}_{\varphi}$ is Fredholm, as shown in
\cite[Proposition 4.17]{KM-gtes-I}:

\begin{proposition}[\cite{KM-gtes-I}]\label{prop:Fredholm-cone}
    Given any compact subinterval $I\subset (0,1)$ and any $p$ and
    $m$, there exists a $\nu_{0}=\nu_{0}(I,p,m)$ such that for all
    $\nu\ge\nu_{0}$, all $k\le m$ and all $\Phi$ in the
    fundamental Weyl chamber satisfying
    \[
                        \alpha(\Phi) \in I, \forall \alpha\in
                        R^{+}(\Phi),
    \]
    the operator $\mathcal{D}_{\varphi}$ acting on the spaces
    \eqref{eq:check-Sobolev} is Fredholm, as is its formal adjoint, and
    the Fredholm alternative holds.
\end{proposition}

This proposition gives us the linear part of the theory needed for the
gauge theory; the non-linear aspects are the multiplication theorems
and the Rellich lemma, which also go through in this setting: see
\cite{KM-gtes-I} for details.  When using the orbifold-type metric, we
will fix an integer $m>2$ and define our space of connections as
\[
            \cA(X,\Sigma,P,\varphi) = \{ \, A \mid A - A^{\varphi} \in
            \check{L}^{2}_{m,A^{\varphi}} \,\}.
\]
We write $\G(X,\Sigma,P,\varphi)$ for the corresponding gauge group, whose
Lie algebra is $L^{2}_{m+1,A^{\varphi}}(X\sminus\Sigma, \g_{P})$,
and we let
\[
                M(X,\Sigma,P,\varphi) \subset \cA(X,\Sigma,P,\varphi)/ \G
\]
be the moduli space of singular anti-self-dual connections for the
metric $g^{\nu}$. The formula for the dimension of the moduli space
at an irreducible regular point (i.e. the index of the operator
$\mathcal{D}_{\varphi}$ above) is given by the same formula
\eqref{eq:dimension-fmla-2} as before, as is the energy $\cE$ of a
solution.

\subsection{Orienting moduli spaces}
\label{subsec:orienting-4d}

We next show that the moduli spaces of singular instantons are
orientable, and discuss how to orient them. Again, for the case
$G=\SU(2)$, the necessary material is in \cite{KM-gtes-I}. In the case
that the $K$ is absent, the orientability of the moduli spaces for a
general simple Lie group $G$ and simply-connected $X$ is
explained in
\cite{Donaldson-Kronheimer}. For the case of $\SU(N)$ and arbitrary
$X$, a proof is given in \cite{Donaldson-orientations}. In the
following proposition, we treat a simple, simply-connected group $G$.
Recall that $\Sigma$ is an \emph{oriented} surface.

\begin{proposition}\label{prop:orientable-sing}
    In the moduli space $M(X,\Sigma,P,\varphi)$, the set of regular
    points $M^{\reg}$ is an orientable manifold. If the dimension of $G$ is even,
    then $M^{\reg}$ has a  canonical orientation; while if $G$ is
    odd-dimensional, the manifold $M^{\reg}$ can be canonically
    oriented once a homology orientation for $X$ is given.
\end{proposition}

\begin{proof}
We first deal with the case that $\Sigma$ is absent: we consider the
orientability of the set of regular points in a moduli space $M(X,P)$
of (non-singular) anti-self-dual connections. Following the usual
argument, we consider the space of all connections modulo gauge,
$\bonf(X,P)$, and also the space of framed connections $\tbonf(X,P)$:
the quotient of the space of connections by the based gauge group.
Over $\tbonf(X,P)$ one has a real determinant line bundle
$\Omega(X,P)$, the determinant of the family of operators obtained by
coupling $-d^{*}\oplus d^{+}$ to the family of connections in the
adjoint bundle. To show that $M^{\reg}(X,P)$ is orientable, we will
show that $\Omega(X,P)$ is trivial.

We will reduce the problem to the known case of an $\SU(2)$ bundle by
applying (in the reverse direction) the same stabilization argument
used in \cite{Donaldson-orientations}. 
Pick any long
root for $G$, say the highest root $\theta$, and let $j: \SU(2) \to G$
be the corresponding copy of $\SU(2)$.
The structure group of $P$ can
be reduced to the subgroup $j(\SU(2))$, giving us an $\SU(2)$ bundle
$Q\subset P$, and we have a map
\[
                j_{*} : \tbonf(X,Q) \to \tbonf(X,P)
\]
whose domain is the space of based $\SU(2)$ connections. From
\cite{Donaldson-orientations}, we know that the corresponding line
bundle $\Omega(X,Q)$ on $\tbonf(X,Q)$ is trivial.

The pair
$(G,j(\SU(2))$ is $4$-connected, so the inclusion of based gauge
groups is surjective on $\pi_{0}$; and the map $j_{*}$ above is
therefore surjective on $\pi_{1}$. To show that the determinant line
$\Omega(X,P)$ on $\tilde\bonf(X,P)$ is trivial, it is therefore enough to show that
its pull back by $j_{*}$ is trivial. 
As stated in section~\ref{subsec:root-systems},
the adjoint representation of $G$
decomposes as a representation of $j(\SU(2))$ as one copy of the adjoint
representation of $\SU(2)$, a number of copies of the
$2$-dimensional representation of $\SU(2)$, and a number of copies of
the trivial representation. Accordingly, the pull-back of $\Omega(X,P)$ by
$j_{*}$ is a tensor product of a number of real line bundles: the one
corresponding to the adjoint representation is a copy of
$\Omega(X,Q)$;  the determinant lines corresponding to
the $2$-dimensional representations are orientable using the complex
orientations; and the remaining factors are trivial. Thus the
triviality of $\Omega(X,P)$ is reduced to the known case of
$\Omega(X,Q)$.

Once one knows that the moduli space is orientable, the next issue is
to specify a standard orientation. We stay with the case that $\Sigma$ is
absent. By ``addition of instantons'', the matter is reduced to
specifying a trivialization of $\Omega(X,P)$ in the case that $P$ is
trivial. In this case we can look at the fiber of $\Omega(X,P)$ at the
trivial connection, where operator is the standard operator
$-d^{*}\oplus d^{+}$ coupled to the trivial bundle $\g$. In the case
that $\g$ is even-dimensional, the determinant line can be canonically
oriented; while in the case that $\g$ is odd-dimensional, we need to
specify and orientation
the determinant of the operator $-d^{*}\oplus d^{+}$ with real
coefficients, i.e.~a homology orientation for $X$.
Conventions for these choices can be set up so that
the orientation of the moduli space agrees with its complex
orientation when $X$ is K\"alher:
the arguments from \cite{Donaldson-orientations} are
adapted to the case of $\SU(N)$ in  \cite{K-higherrank}, and the case
of a general simple $G$ is little different.

We now consider how the determinant line changes when we introduce a
codimension-2 singularity along $\Sigma\subset X$. There is again a
determinant line bundle $\Omega(X,\Sigma,P,\varphi)$ over the space
$\tbonf(X,\Sigma,P,\varphi)$ of singular connections. We must show
this line bundle is trivial. If we consider
the restriction of the line bundle to a compact family $S\subset
\tbonf(X,\Sigma,P,\varphi)$  of singular
connections, then the data is a family of $G$ bundles $P_{s}$ with a
family of reductions of structure group, $\varphi_{s}$. Let $A_{s}$
be a family of smooth connections in the bundles $P_{s}$. We may
suppose that $A_{s}|_{\Sigma}$ respects the reduction $\varphi_{s}$ to
the group $G_{\Phi}\subset G$. We may also suppose that $A_{s}$ can be
identified with pull-back of $A_{s}|_{\Sigma}$ in a tubular
neighborhood of $\Sigma$. Let $A^{\varphi}_{s}$ be constructed from
$A_{s}$ by adding the singular term in the usual way, as at
\eqref{eq:singular-model-A}. Over $S$, we can consider two line
bundles: first the line bundle $\Omega(X,P)$, which we have already
seen is trivial; and second the line bundle
$\Omega(X,\Sigma,P,\varphi)$, the determinant line of the deformation
complex for the singular instantons. We must examine the ratio
\begin{equation}\label{eq:Omega-ratio}
            \Omega(X,\Sigma,P,\varphi) \otimes \Omega(X,P)^{-1}
\end{equation}
and show that it is trivial.

By excision, we can replace $X$ now by
the sphere-bundle over $\Sigma$ which is obtained by doubling the
tubular neighborhood, and we can replace the family of connections
$A_{s}$ by the family of $G_{\Phi}$ connections obtained by pulling
back $A_{s}|_{\Sigma}$.  In this setting,
the adjoint bundle $\g_{P}$ over $S\times
\Sigma$ decomposes as a sum of two sub-bundles, associated to the
decomposition of $\g$ as $\g_{\Phi} \oplus \fo$ in \eqref{eq:fo-def}.
There is a corresponding tensor product decomposition of each of the
determinant lines in \eqref{eq:Omega-ratio} above. On the summand
$\g_{\Phi}$, the two operators agree; so the ratio
\eqref{eq:Omega-ratio} is isomorphic to the ratio of determinant lines
for the same operators coupled only to the subbundle coming from
$\fo$ instead of to all of $\g$. Since $\fo$ is complex, the ratio of
determinant lines can be given its complex orientation, which
completes the proof that the moduli space is orientable. This argument
also shows that a choice of orientation for $\Omega(X,P)$ gives rise
to a preferred orientation for $\Omega(X,\Sigma,P,\varphi)$.
So the data needed to orient the moduli space is the same in the
singular and non-singular cases.
\end{proof}

\subsection{The unitary and other non-simple groups}
\label{subsec:unitary-4d}

Up until this point, $G$ has always been a simple and simply-connected
group. One very straightforward generalization is allow $G$ to be
semi-simple and still simply-connected. In this case $G$ is a product
of simple groups $G_{1}\times \dots \times G_{m}$, and our
configuration spaces of connections are simply products.  We can
define the lattice $\LL(G_{\Phi})$ and the the monopole charge $\ll$ just
as before, with the understanding that we are now dealing with a root
system that is reducible. The only new feature here is that the
instanton charge $k$ is now an $m$-tuple, $k=(k_{1},\dots, k_{m})$,
and in the dimension and action formulae we see
\[
            \sum h_{i}^{\vv} k_{i}
\]
with $h^{\vv}_{i}$ the dual Coxeter number of $G_{i}$, where
previously we had just $h^{\vv}k$. The monotone condition and
Proposition~\ref{prop:Einstein} are unchanged.

There is a more interesting variation to consider when $G$ has center
of positive dimension, such as in the case of the unitary group. We
will make the assumption that $G$ is connected and that
the commutator subgroup $[G,G]$ is
simply connected. We shall write $Z(G)$ for the center of $G$, and we
set
\[
\begin{aligned}
\bar{Z}(G) &= G/[G,G] \\
                                &= Z(G)/(Z(G)\cap [G,G]).
                                \end{aligned}                                
\]
The quotient map will be written
\begin{equation}\label{eq:g-det}
            \gdet : G\to\bar{Z}(G),
\end{equation}
and we use the same notation also for the corresponding map on the Lie
algebras. The abelian group $Z(G)$ may not be connected, but
$\bar{Z}(G)$ is a  torus. We will run through some of the points to
show how the theory adapts to this case.

\paragraph{Instanton moduli spaces.}
Let us temporarily omit the codimension-2 singularity along $\Sigma$
from our discussion. An appropriate setting for gauge theory in a
$G$-bundle $P\to X$ when $Z(G)$ has positive dimension is not to consider
the space of all $G$-connections in $P$, but instead to fix a
connection $\Theta$ in the associated $\bar Z(G)$-bundle, $\gdet(P)
\to X$, and to consider the space $\cA(X,P)$ of connections $A$ in $P$ which
induce the given connection $\Theta$ in $\gdet(P)$:
\[
            \cA(X,P) = \{\, A \mid \gdet(A) = \Theta \,\}
\]
(In the case of
the unitary group $U(N)$, this means looking at the unitary connections in a
rank-$N$ vector bundle $E$ inducing a given connection $\Theta$ in
$\Lambda^{N}E$.)
Such a
$G$-connection $A$ in $P$ is entirely determined by the induced
connection $\bar{A}$ in the associated $(G/Z(G))$-bundle $\bar P$.
The
appropriate gauge group in this context is not the group of
all automorphisms of $P$, but instead the group $\G(X,P)$ consisting
of automorphisms which take values in $[G,G]$ everywhere. That is, an
element of $\G(X,P)$ is a section of associated fiber bundle arising
from the adjoint action of $G$ on the subgroup $[G,G]$. In the case of
a  unitary vector bundle, this is the group of unitary automorphisms
of a vector bundle $E\to X$ having determinant $1$ at every point.
The moduli space $M(X,P)$ is the subspace of the quotient $\bonf(X,P)
= \cA(X,P)/\G(X,P)$ consisting of all $[A]$ such that the curvature of
$\bar{A}$ (not the curvature of $A$) is anti-self-dual:
\[
            M(X,P) = \{ \, A \in \cA(X,P) \mid F^{+}_{\bar{A}}=0 \,\}
            / \G(X,P).
\]

Note that the chosen connection $\Theta$ really plays no role here,
and we could equally well regard $\cA(X,P)$ as paramatrizing the
connections $\bar{A}$ in $\bar{P}$. (Later, however, when we introduce
holonomy perturbations in section~\ref{subsec:holomy-perts}, a choice
of $\Theta$ will be important.) Indeed, $\bonf(X,P)$ and $M(X,P)$
really depend only on the adjoint group $\bar{G}$ and the bundle
$\bar{P}$, because both the adjoint bundle with fiber $[\g,\g]$ and
the bundle of groups with fiber $[G,G]$ (whose sections are the gauge
transformations) are bundles associated to $\bar{P}$. The choice of
$G$ therefore only affects which bundles $\bar{P}$ can arise.

Because of this last observation, we can if we wish start with the
simply-connected semi-simple group $G_{1}$ with adjoint group
$\bar{G}_{1}$ and construct a group $G$ with
$[G,G]\cong G_{1}$ as an extension. Such a $G$ can be described by taking a
subgroup $H_{1}$ of the finite group $Z(G_{1})$ together with a torus
$S$ and an
injective  homomorphism $a : H_{1} \to S$; one then defines
\begin{equation}\label{eq:G-from-G1}
                    G = (G_{1} \times S) / H
\end{equation}
where $H \subset H_{1} \times S$ is the graph of $a$. A given bundle
$\bar{P}$ with structure group $\bar{G}=\bar{G_{1}}$ lifts to a
$G$-bundle $P$ if and only if its characteristic class
\[
            \bar{c}(\bar{P}) \in H^{2}(X; \pi_{1}(\bar{G}_{1}) )
\]
lifts to a class
\[
            c \in H^{2}(X; \pi_{1}(G) ).
\]
The image of the map $\pi_{1}(G) \to \pi_{1}(\bar{G}_{1}) \cong
Z(G_{1})$ is our chosen subgroup $H_{1}$ and $\pi_{1}(G)$ is
torsion-free; so the bundle $\bar{P}$ has
a lift to a $G$-bundle if and only if $\bar{c}(\bar{P})$ lies in the
subgroup $H^{2}(X; H_{1})$ and admits an ``integer lift'' to
$\Z^{k}$ for one (and hence any) presentation of $H_{1}$ as
\[
                   0\to \Z^{k} \to \Z^{k} \to H_{1} \to 0.
\]
This discussion shows us that, in order to allow the largest possible
collection of $\bar{G_{1}}$-bundles to lift, and to avoid redundancy,
we may impose the following conditions.

\begin{condition}\label{cond:G-conditions}
    In the construction of the non-semi-simple group $G$ from $G_{1}$
    in \eqref{eq:G-from-G1}, we may require
    \begin{enumerate}
        \item the subgroup $H_{1} \subset Z(G_{1})$ is the whole of
        $Z(G_{1})$;
        \item the rank $k$ of the torus $S$ is chosen to be equal to
        the number of generators in a smallest possible generating set
        of $H_{1}$; or equivalently, the image of $H_{1}$ in $S$ is
        not contained in any proper sub-torus.
    \end{enumerate}
\end{condition}

These conditions mean that if $G_{1}$ is simple, then $G=G_{1}$ in the
case of type $E_{8}$, $F_{4}$ or $G_{2}$ (the simply connected cases);
while $S$ will be a circle group in all other cases except $D_{2r}$,
where $S$ will be a $2$-torus. These conditions do not determine $G$
uniquely, in general. For example, in the case that $G_{1}=\SU(N)$,
the condition allows that $G$ is $U(N)$ but also allows $G$ to be
$U(N)/C_{m}$, where $C_{m}$ is a cyclic central subgroup of order
prime to $N$.

\paragraph{Singular instantons.}
We fix a maximal torus and a set of positive roots
for $G$. Because $G$ is not semi-simple, the fundamental Weyl chamber
in $\ft = \mathrm{Lie}(T)$ is the product of the fundamental Weyl
chamber for $[\g,\g]$ with $\fz(G) = \mathrm{Lie}(Z(G))$.
The bundle $P$ is no longer trivial on the $3$-skeleton of $X$,
because $\pi_{1}(G)$ is non-trivial. We can identify $\pi_{1}(G)$
with the lattice
\[
            \LL(G) \subset \fz(G)
\]
obtained as the projection of the integer lattice in $\ft$. The
bundle $P$ has a $2$-dimensional characteristic class which we write
as
\begin{equation}\label{eq:c-char-class}
                c(P)  \in H^{2}(X; \LL(G)).
\end{equation}

Now we introduce the codimension-$2$ singularity along a surface
$\Sigma\subset X$. 
Fix an element $\Phi$ in the Lie algebra $\g$
belonging to the fundamental Weyl chamber and satisfying
$\theta(\Phi)<1$, where $\theta$ is the highest root.
Let $O\subset \g$ be the orbit of $\Phi$ under the
adjoint action. This lies in a translate  of $[\g,\g]$ inside $\g$
consisting of all elements with the same trace as $\Phi$.
Choose a
section $\varphi$ of the associated bundle $O_{P}|_{\Sigma}$, so
defining a reduction of the structure group of $P|_{\Sigma}$ to the
subgroup $G_{\Phi}$.

   Let $\Theta$ again be a fixed
    connection in the associated bundle $\gdet(P)$ on $X$, and
    let ${A}^{0}$ be a $G$-connection with $\gdet(A^{0})=\Theta$.
    Choose an extension of the section $\varphi$ to the tubular
    neighborhood, and define a singular $G$ connection on
    the restriction of $P$ to $X\sminus\Sigma$ by the same formula
    as in the previous case:
\begin{equation}\label{eq:singular-model-A-tilde}
            {A}^{\varphi} =  {A}^{0} + \beta(r) \varphi\otimes \eta
\end{equation}
   The induced connection on $\gdet(P)$ is
   \begin{equation}\label{eq:Theta-varphi}
                \Theta^{\varphi} = \Theta + \beta(r)
                \gdet(\varphi)\otimes \eta.
   \end{equation}
 We
    define a space of $G$-connections
     modeled on $A^{\varphi}$ and having the same induced connection
     on $\gdet(P)$:
    \[
            \cA^{p}(X,\Sigma,P,\varphi) = \{\,  A^{\varphi} + a \mid
                                                a, \nabla_{A^{\varphi}}
                                                a \in L^{p}(X\sminus
                                                \Sigma;
                                                [\g,\g]\otimes\Lambda^{1}(X)) \,\}.
    \]  
    Our gauge group will consist of $[G,G]$-valued automorphisms of
    $P|_{X\sminus\Sigma}$:
    \[
                    \G^{p}(X,\Sigma,P,\varphi) = \{\,
                g  \mid
                \text{$\nabla_{A^{\varphi}}g,
                \nabla^{2}_{A^{\varphi}}g \in
                L^{p}(X\sminus\Sigma; [G,G]_{P})$}\,\}.
    \]
    For $p$ sufficiently close to $2$, as in \eqref{eq:p-close-to-2},
    the moduli space $M(X,\Sigma,P,\varphi)$ is defined as the quotient by
    $\G^{p}$ of the set of solutions to $F^{+}_{\bar{A}}=0$ in
    $\cA^{p}(X,\Sigma,P,\varphi)$.

    \paragraph{Monopole charges.}
    We continue to write $Z(G_{\Phi})$ for the center of the commutant
    $G_{\Phi}$, and $\LL(G_{\Phi}) \subset \fz(G_{\Phi})$ for the image of the
    integer lattice in $\ft$ under the projection $\Pi : \ft \to
    \fz(G_{\Phi})$.  The structure group of the bundle $P_{\varphi} \to \Sigma$
    can still be reduced to the subgroup $T$, and the resulting
    $T$-bundle is classified by an element $\xi$ in the integer
    lattice. The element $\ll = \Pi(\xi)$ determines $P_{\varphi}$ up
    to isomorphism. The bundle $P$ itself may be non-trivial on
    $\Sigma$, and $\ll$ is constrained by the requirement that
    \[
    \gdet(\ll) = \langle c(P) , [\Sigma] \rangle
    \]
      in  $\LL(G) \cong \pi_{1}(G)$.

      \paragraph{Dimension and energy.}
      The formula for the formal dimension of the moduli space can be
      written in the same was as before (see
      \eqref{eq:dimension-fmla-2}), except for the term involving the
      $4$-dimensional characteristic class:
      \begin{equation}\label{eq:dimension-non-simple}
               -2 p_{1}(\g_{P})[X] + 4 \weyl(\ll)
     + \frac{(\dim O)}{2} \chi(\Sigma)  - (\dim G) (b^{+}-b^{1}+1)
      \end{equation}
      Note that the term $\weyl(\ll)$ depends only on the projection
      of $\ll$ into $[\g,\g]$. The term involving $p_{1}(\g_{P})$
      satisfies a congruence depending on  the associated
      $\bar{Z}(G)$-bundle $\gdet(P)$, via the $2$-dimensional
      characteristic class $c(P)$. In the case that the structure
      group of $P$ reduces to the maximal torus, we obtain a lift of
      $c(P)$ to a class $\hat{c} \in H^{2}(X; L(T))$, and we then have
      \[
                    -2 p_{1}(\g_{P}) = -\langle \hat{c}, \hat{c} \rangle
      \]
      where the quadratic form on the right is defined using the
      semi-definite Killing form on the lattice $L(T)$ and the
      cup-square on $X$. Modulo $4h^{\vv}$,
      the quantity on the right depends only on the image of $c(P)$
       in the finite group $H^{2}(X; \pi(\bar{G}))$. Furthermore,
      in the case that $[G,G]$ is simple,
      if $P$ is altered on a $4$-cell in $X$, then
      $p_{1}(\g_{P})$ changes by a multiple of $2h^{\vv}$ (as in the case
      of the $4$-sphere); and since a general $P$ can be reduced to
      the maximal torus on the complement of a $4$-cell, we have in
      general
      \[
            -2 p_{1}(\g_{P})[X]  = -\langle\hat{c}, \hat{c} \rangle[X] \pmod
            {4h^{\vv}}
      \]
       where $\hat{c}$ is any lift of $c(P)$ to $H^{2}(X;L(T))$. In
       the case of the unitary group $U(N)$, we identify $c(P)$ as the first
       Chern class, and the formula becomes
       \begin{equation}\label{eq:p1-versus-c1-squared}
                  -2 p_{1}(\g_{P})[X] = - 2(N-1) c_{1}^{2}(P)[X]
                  \pmod{4N}.
       \end{equation}
      
      The appropriate definition of the
      energy $\cE$ is still the formula \eqref{eq:cE-def}, with the
      understanding that the norm defined by the Killing form is now
      only a semi-norm, so that the formula for the energy actually
      involves only
      $\bar{A}$. In a similar manner, the formula
      \eqref{eq:energyl-fmla-2} now becomes
      \begin{equation}\label{eq:energy-non-simple}
      \cE = 8\pi^{2} \Bigl(-2 p_{1}(\g_{P})[X] + 2 \langle
                    \Phi,\ll\rangle
          - \langle
                    \Phi,\Phi\rangle(\Sigma\cdot\Sigma)\Bigr),
       \end{equation}
      where the inner products are now only semi-definite.

      In these two formulae, as elsewhere, the component of $\Phi$ in
      the center of $\g$ is immaterial. The monotone condition (i.e.
      the condition that the terms in \eqref{eq:dimension-non-simple}
      and \eqref{eq:energy-non-simple} which are linear in $p_{1}(\g)$
      and $\ll$ are proportional) is a constraint only on the
      component $\bar\Phi$ which lies in $[\g,\g]$. In formulae, the
      monotone condition still amounts to requiring that
      \begin{equation}\label{eq:monotone-non-simple}
                \langle \Phi , \ll \rangle = 2 \weyl(\ll)
      \end{equation}
       for all $\ll$ in $\fz(G_{\Phi})$. Given any $\Phi_{0}$ in the
       fundamental Weyl chamber, there is a unique $\Phi$ satisfying
       this condition with the additional constraints that (i)
       $\Phi_{0}$ and $\Phi$ have the same centralizer and (ii)
       $\Phi_{0}$ and $\Phi$ have the same central component in
       $\fz(G)$.
      
      \paragraph{Isomorphic moduli spaces.}

      For the following discussion, we return temporarily to
      considering the moduli spaces $M(X,P)$ of non-singular
      connections, in the absence of the
      embedded surface $\Sigma$.
      If $P$ and $P'$ are isomorphic $G$-bundles, then the moduli
      spaces $M(X,P)$ and $M(X,P')$ are certainly homeomorphic also;
      but a particular identification $M(X,P)\to M(X,P')$ depends on a
      choice of bundle isomorphism $f : P \to P'$. Because we are
      dividing out by the action of the gauge group $\G(X,P)$
      consisting of all $[G,G]$-valued automorphisms, the map of
      moduli spaces $M(X,P)\to M(X,P')$ depends on $f$ only through
      the corresponding  isomorphism of $\bar{Z}(G)$-bundles,
      $\gdet(f):\gdet(P)\to \gdet(P')$.

      A convenient viewpoint on this is to fix a principal
      $\bar{Z}(G)$-bundle $\delta$ on $X$, together with a connection
      $\Theta$ on $\delta$, and then regard $\bonf$ as parametrizing
      isomorphism classes of triples consisting of:
      \begin{enumerate}
        \item a principal $G$-bundle $P\to X$ with specified instanton
        charges $k=(k_{1},\dots, k_{m})$;
        \item an isomorphism of $G$-bundles $q: \gdet(P)\to \delta$;
        \item a connection $A$ in $P$ with $\gdet(A) =
        q^{*}(\Theta)$, or equivalently just a connection $\bar{A}$ in
        $\bar{P}$.
       \end{enumerate}
        Two such triples $(P,q,\bar{A})$ and $(P',q',\bar{A}')$ are isomorphic if
        there is an isomorphism of $G$-bundles, $f:P\to P'$, with
        $f^{*}(\bar{A}')=\bar{A}$ and $q = q' \comp \gdet(f)$. From
        this point of view, it is natural to write the configuration
        space as $\bonf_{k}(X)_{\delta}$, indicating its dependence on
        $k$ and the $\bar{Z}(G)$-bundle. The corresponding moduli
        space can be written
        \[
                    M_{k}(X)_{\delta} \subset \bonf_{k}(X)_{\delta}
        \]

        It is clear that the automorphisms  $g:\delta\to\delta$ act on
        $\bonf_{k}(X)_{\delta}$, preserving the moduli space, by
        $(P,q,\bar{A})\mapsto (P,g\comp q,\bar{A})$.
        Furthermore the action of $g$ on $\bonf_{k}(X)_{\delta}$
        is trivial if and only if
        $g=\gdet(f)$ for some bundle isomorphism $f : P \to P$ with
        $f^{*}(\bar{A})=\bar{A}$ for all $A$. This last condition on
        $f$ requires that $f$ take values in $Z(G)$. Thus the group
        which acts effectively is the quotient of $\Map(X,\bar{Z}(G))$
        by the image of $\Map(X,Z(G))$ under the map
        $\gdet:Z(G)\to\bar{Z}(G)$. This is a finite group. For
        example, in the case of $U(N)$, this is the quotient of
        $H^{1}(X;\Z)$ by the image of multiplication by $N$; this is
        isomorphic to the subgroup of $H^{1}(X;\Z/N)$ consisting of
        elements with an integer lift.

        There is another way in which isomorphisms arise between
        moduli spaces of this sort.
       The group operation provides a homomorphism of groups,
      \[
                 G \times Z(G) \to G.
      \]
      Given a $G$-bundle $P$ and a $Z(G)$-bundle $\epsilon$ on $X$, we
      can use this homomorphism to obtain a ``product'' $G$-bundle,
      which we will denote by $P \bb\epsilon$.
      If
      we fix a connection $\omega$ in $\epsilon$, then to each connection
      $A$ in $\conf(X,P)$ with
      $\gdet(A)=\Theta$, we can associate a connection
      \[A'=A+\omega\] in $\conf(X,P\bb\epsilon)$ with
      $\gdet(A')=\Theta+\gdet(\omega)$.  This operation
      descends to the quotient space $\bonf$ and preserves the
      locus of connections which satisfy the equations
      $F^{+}_{\bar{A}}=0$. It therefore gives an identification of moduli spaces
      \[
                    \mu_{\epsilon} : M(X,P)
                            \to M(X,P\bb\epsilon).
      \]
      In terms of the data $k$ and $\delta$ which determine the moduli
      space up to isomorphism, this is a map
      \[
                    \mu_{\epsilon} : M_{k}(X)_{\delta} \to
                     M_{k}(X)_{\delta\bb\gdet(\epsilon)},
      \]
      where we have extended the use of $\bb$ to denote also the
      product of two $\bar{Z}(G)$-bundles. (In the case of $U(N)$, the
      operation $\bb$ in both cases becomes the tensor product by a
      line bundle.)

      All of this discussion can be carried over to the case of
      connections with singularities along $\Sigma\subset X$. Once
      $\Phi$ is given, the moduli space $M(X,\Sigma,P,\varphi)$ is
      determined up to isomorphism by $\delta = \gdet(P)$ together
      with the instanton charges $k$ and monopole charges $l$. (In the
      case that $\Sigma$ has more than one component, the monopole
      charges need to be specified for each component.) We can
      therefore write the moduli space as $M_{k,l}(X,\Phi)_{\delta}$.
      The automorphisms of $\delta$ then act on the moduli spaces, and
      it is again the quotient of $\Map(X,\bar{Z}(G))$ by the image of
      $\Map(X,Z(G))$ that acts effectively.  If $\epsilon$ is a
      $Z(G)$-bundle, we also have a corresponding isomorphism
      \[
                    \mu_{\epsilon} : M_{k,l}(X,\Phi)_{\delta}
                    \to
                    M_{k,l}(X,\Phi)_{\delta\otimes\gdet(\epsilon)}.
      \]
      
      \paragraph{Reducibles.}
      The main point at which the the present discussion of
      non-semi-simple groups diverges from the previous case is in the
      discussion of reducible connections, because the characteristic
      class $c(P)$ now plays a role and because the integer lattice is
      no longer generated by the coroots. Let us write $L(T)$ for
      the integer lattice of the maximal torus; the lattice of weights
      is the dual lattice in $\ft^{*}$. The simple
      coroots $\alpha^{\vv}$ are a basis for $L(T) \cap [\g,\g]$,
      because $[G,G]$ is simply connected. Let $\bar{w}_{\alpha}$
      denote the dual basis for $(\ft \cap [\g,\g])^{*}$. For each
      simple root $\alpha$, we can choose a weight $w_{\alpha}$ in
      $\ft^{*}$ such that the restriction of $w_{\alpha}$ to $\ft \cap
      [\g,\g]$ is $\bar{w}_{\alpha}$. The choice of $w_{\alpha}$ is
      uniquely determined by $\bar{w}_{\alpha}$ to within the addition
      of a weight that factors through $\gdet$. Associated to
      $\alpha$, as before, is a subgroup $G(\alpha) \subset G$, the
      centralizer of $w_{\alpha}^{\dag}$ (or equivalently of
      $\bar{w}_{\alpha}^{\dag}$). Note that we will not always be able
      to choose $w_{\alpha}$ in such a way that its restriction to
      $\fz(G)$ is zero, and nor will its restriction to the lattice
      $\LL(G)\subset\fz(G)$ be integral: it will define a map
      \[
                    w_{\alpha} : \LL(G) \to \Q.
      \]
       In the case of $U(2)$ for example, $w_{\alpha}$ will map the
       rank-$1$ lattice
       $\LL(G)$ onto $\frac{1}{2}\Z$.
      
      To say that $[A]$ is reducible
      still means that its stabilizer in the gauge group
      $\G^{p}(X,\Sigma,P,\varphi)$ has positive dimension, and this is
      equivalent to there being a non-zero covariant-constant section
      $\psi$ of the associated bundle $[\g,\g]_{P}\subset \g_{P}$. As
      in subsection~\ref{subsec:reducibles}, we obtain a reduction of
      the structure group of $P$ to a subgroup $G_{\Psi}$, and we
      write $P_{\psi}\subset P$ for this $G_{\Psi}$-bundle. For some
      element $\sigma$ in the Weyl group, $G_{\sigma(\Psi)}$ is
      contained in $G(\alpha)$ for some simple root $\alpha$, and it
      follows that $w_{\alpha}\comp\sigma$ defines a  non-central
      character of $G_{\Psi}$:
      \[
                    s : G_{\Psi} \to U(1).
      \]
       Applying $s$ to the connection $A$ gives a $U(1)$ connection
       $s(A)$ whose curvature $F_{s(A)}$ is an $L^{p}$ form defines
       a de Rham class
       \begin{equation}\label{eq:deRham-2}
                                \frac{i}{2\pi}[F_{s(A)}] = c_{1}(s(P_{\psi})) -
            \sum_{j=1}^{r}(\w_{\alpha}\comp\sigma_{j})(\Phi)
            \mathrm{P.D.}[\Sigma_{j}]
       \end{equation}
       where the $\Sigma_{j}$ are the components of $\Sigma$, just as
       at \eqref{eq:de-Rham-class}. Unlike the previous case, the
       curvature form $F_{s(A)}$ is not anti-self-dual, because
       $F_{A}^{+}$ has a non-zero central component. The Chern-Weil
       formula for the central component is
       \[
                                \frac{i}{2\pi}[F_{\gdet(A)}] =
                                c(P) -
            \sum_{j=1}^{r}\gdet(\Phi)
            \mathrm{P.D.}[\Sigma_{j}]
      \]
      as an equality of $\fz(G)$-valued cohomology classes. The
      self-dual part of $F_{A}$ coincides with the self-dual part of
      $F_{\gdet(A)}$. So, applying $w_{\alpha}\comp\sigma$ to this
      last formula and then subtracting it  from \eqref{eq:deRham-2},
      we learn that the class
      \[
                          c_{1}(s(P_{\psi})) - w_{\alpha}(c(P))
                             -
            \sum_{j=1}^{r}(\w_{\alpha}\comp\sigma_{j})(\Phi-\gdet(\Phi))
            \mathrm{P.D.}[\Sigma_{j}]
      \]
      is represented by an anti-self-dual form on $X$. The first term
      in this last formula is an integral class. The second term may
      not be integral: the class $c(P)$ takes values in
      $\LL(G)$, and $w_{\alpha}$ need not take integer values
      on this lattice. The terms in the last sum depend only on
      $\bar{w}_{\alpha}$, not on $w_{\alpha}$, because
      $\Phi-\gdet(\Phi)$ lies in $[\g,\g]$. This leads to
      the following variant of Proposition~\ref{prop:no-reducibles}.

      \begin{proposition}\label{prop:no-reducibles-2}
    Suppose that $b^{+}(X)\ge 1$, and let the components of $\Sigma$
    be $\Sigma_{1},\dots,\Sigma_{r}$. Write
    \[
                    \bar \Phi = \Phi - \gdet(\Phi)
    \]
    for the component of $\Phi$ in $[\g,\g]$.
    Suppose that for every fundamental
    weight $\w_{\alpha}$ and every choice of elements
    $\sigma_{1},\dots,\sigma_{r}$ in the Weyl group, the real
    cohomology class
\begin{equation*}
            w_{\alpha}(c(P)) +
            \sum_{j=1}^{r}(\bar{w}_{\alpha}\comp\sigma_{j})(\bar\Phi)
            \mathrm{P.D.}[\Sigma_{j}]
\end{equation*}
    is not integral. Then
   for generic choice of Riemannian metric on $X$, there are
       no reducible solutions in the moduli space
       $M(X,\Sigma,P,\varphi)$.
\end{proposition}

  Note that, as a rational cohomology class, $w_{\alpha}(c(P))$ does
  depend on $w_{\alpha}$, not just on $\bar{w}_{\alpha}$. But
  different choices of how to extend $\bar{w}_{\alpha}$ will be
  reflected in a change to $w_{\alpha}(c(P))$ by an integral class.

  In the case that each component $\Sigma_{j}$ is null-homologous, the
  criterion in the above proposition reduces to the requirement:
\begin{equation}\label{eq:simple-coprime-condition}
  \text{\emph{ $w_{\alpha}(c(P))$ be non-integral for each simple root
  $\alpha$}}. 
\end{equation}
  To
  illustrate this, consider the familiar case of the unitary group
  $U(N)$. There are $(N-1)$ simple roots, $\alpha_{k}$, $k=1,\dots,
  N-1$, and if we write a typical Lie algebra element in $\u(N)$ as
  \[
                    x = i \diag(\epsilon_{1},\dots ,\epsilon_{N}),
  \]
  then we can choose $w_{\alpha_{k}}$ so that
  \[
            w_{\alpha_{k}}(x) = \epsilon_{1} + \dots + \epsilon_{k}.
  \]
   Then $c(P)$ is related to the usual first Chern class $c_{1}(P)$ in
   such a way that
   \[
                  w_{\alpha_{k}}(c(P)) = (k/N) c_{1}(P).      
   \]
   So the criterion in the proposition is that the rational class
   $(k/N) c_{1}(P)$ should not be integral, for any $k$ in the range
   $1\le k \le N-1$.  This is equivalent to requiring that the
   evaluation of $c_{1}(P)$ on some integral homology class should be
   coprime to $N$. We record this corollary, which is familiar from
   the case of non-singular instantons:

   \begin{corollary}\label{cor:example-coprime}
    Let $G$ be the unitary group $U(N)$.
    Suppose that $b^{+}(X)\ge 1$, and that all components of $\Sigma$
    are null-homologous. If there is an integral homology class in $X$
    whose pairing with $c_{1}(P)$ is coprime to $N$, then
     for generic choice of Riemannian metric on $X$, there are
       no reducible solutions in the moduli space
       $M(X,\Sigma,P,\varphi)$.
   \end{corollary}

\begin{remark}
One might hope that there would be something analogous to this
corollary in the case of other simply-connected groups with
non-trivial center, but unfortunately, the case of the unitary group
is rather special.
If the commutator subgroup $G_{1}=[G,G]$ is a simply-connected simple
group of any type other than $A_{r}$, there will always be a
fundamental weight $w_{\alpha}$ which takes integer values on
$\LL(G)$; so the criterion \eqref{eq:simple-coprime-condition} cannot
be met. This phenomenon is noted in
\cite[Proposition~7.8]{Ramanathan-1975}.
\end{remark}

\paragraph{Orientations.}

The discussion of orientations from
section~\ref{subsec:orienting-4d} adapts readily to the case of
non-simple groups $G$ with $[G,G]$ simply-connected. 
Recall that Proposition~\ref{prop:orientable-sing}  has two
parts: the first is the assertion that the moduli spaces are
orientable, and the second involves specifying a canonical
orientation. Adapting the first part is routine. For the second task,
we need to observe that the restriction of the homomorphism
\eqref{eq:g-det} to the maximal torus $T\subset G$ admits a right-inverse,
\begin{equation}\label{eq:gdet-inverse}
\begin{aligned}
\fe : \bar{Z}(G) \to T \\
\gdet \comp \fe = 1.
\end{aligned}
\end{equation}
Fix such an $\fe$ once and for all. From our chosen $\bar{Z}(G)$-connection $\Theta$
in $\gdet(P)$ we now obtain a $G$-connection $\fe(\Theta)$ on a
bundle isomorphic to $P$, with $\gdet(\fe(\Theta))=\Theta$. This
comes with a reduction of structure group (on the whole of $X$) to the
maximal torus, and a fortiori to $G_{\Phi}$. Adding the singular term
along $\Sigma$ in the usual way, we obtain a distinguished singular
connection $A^{\varphi}$. Because $A^{\varphi}$ respects a reduction
to the maximal torus, the adjoint bundle with fiber $[\g,\g]$
decomposes as a direct sum of a bundle with fiber $\ft\cap
[\g,\g]$ and a complex vector bundle, and there is a corresponding
decomposition of the operator \eqref{eq:D-4d} whose determinant line we
wish to orient. The induced connection on the first summand is
trivial. As in the previous argument, we can now proceed by
making use of the complex orientation on the second summand and the
homology orientation $o_{W}$ for the first summand. In this way, the
moduli spaces $M(X,\Sigma,P,\varphi)$ become canonically oriented at
all regular points.

Let us write $\delta$ again for the $\bar{Z}(G)$-bundle $\gdet(P)$,
and so denote the moduli space by $M_{k,l}(X,\Phi)_{\delta}$ as above.
Recall that in this setting, the automorphisms of $\delta$ act on the
moduli space.
The naturality of the construction of the orientation means that the
automorphisms of this moduli space arising in this way
are orientation-preserving diffeomorphisms on the regular part. A more
interesting question arises if we ask whether the map
\begin{equation}\label{eq:mu-epsilon}
            \mu_{\epsilon}: M_{k,l}(X,\Phi)_{\delta} \to
            M_{k,l}(X,\Phi)_{\delta\otimes \gdet(\epsilon)}
\end{equation}
preserves orientation. The singularity along $\Sigma$ plays no role in
the answer here, and we could equally well consider $M_{k}(X)_{\delta}$
instead. This is a question which was treated for
$G=U(2)$ in \cite{Donaldson-orientations}, and the argument was
adapted for $U(N)$ in \cite{K-higherrank}. The case of a general
non-semisimple group is little different, and we shall summarize the results.

We shall write $G_{1}=[G,G]$ again, and we shall
suppose that $Z(G_{1})$ is non-trivial (for otherwise $G$ is
simply a product). We will also require that $G_{1}$ is simple, and
that $G$ is obtained from
$G_{1}$ by the construction \eqref{eq:G-from-G1} and that
Condition~\ref{cond:G-conditions} holds. Then we have the following
result:

\begin{proposition}\label{prop:epsilon-signs}
    Under the above assumptions, the map $\mu_{\epsilon}$ in
    \eqref{eq:mu-epsilon} is orientation preserving if the simple
    group $G_{1}$ is any group other than $E_{6}$ or $\SU(N)$. In the
    case of $E_{6}$, the result depends on the choice of $\fe$; but
    $\fe: U(1)\to E_{6}$ may be chosen so that $\mu_{\epsilon}$ is
    orientation-preserving for all $\epsilon$ and $\delta$. In the
    case of $\SU(N)$, if $G$ is chosen to the standard $U(N)$ and
    $\fe :  U(1) \to U(N)$ is the
    inclusion in the first factor of the torus $U(1)^{N}$ in $U(N)$,
    then $\mu_{\epsilon}$ is orientation-preserving if $N$ is $0$, $1$
    or $3$ mod $4$. If $N$ is $2$ mod $4$ then $\mu_{\epsilon}$ is
    orientation-preserving if and only if $c_{1}(\epsilon)^{2}[X]$ is
    even.
\end{proposition}

\begin{proof}
Both $\fe(\delta)$ and $\fe(\gdet(\epsilon))$ are $T$-bundles on $X$.
Let $a$ and $b$ be their respective characteristic classes in $H^{2}(X; L(T))$.
The calculation for $U(N)$ from \cite{K-higherrank} adapts with little
change to show that $\mu_{\epsilon}$ preserves or reverses orientation according to
the parity of the quantity,
\begin{equation}\label{eq:sign-change-1}
     \left(
                  \frac{1}{2}  \langle a \cupprod b \rangle
                  + \frac{1}{4}  \langle  b\cupprod b \rangle
                  + \weyl(b)\cupprod \weyl(b)
    \right)[X]
 \end{equation}
in which $\langle a \cupprod b \rangle$ denotes the pairing in
$H^{4}(X;\Z)$
obtained from the semi-definite Killing form on $L(T)$ and the cup
product on $X$, and $\weyl(b)$ is to be interpreted as an element of
$H^{2}(X;\Z)$. The quantity above plainly depends only on the images of
$a$ and $b$ under the projection to $[\g,\g]$. Let us write $\bar{a}$,
$\bar{b}$ for these projections (with the torsion parts of the
cohomology dropped). We have
\begin{equation}
            \begin{aligned}
                \bar{a} &\in   H^{2}(X;
                L(\bar{T}_{1}))/\mathrm{torsion}
                \subset H^{2}(X; \ft_{1})\\
                \bar{b} &\in   H^{2}(X;  L(T_{1}))/\mathrm{torsion}
                \subset  H^{2}(X; \ft_{1})
            \end{aligned}
\end{equation}
where $L(T_{1})$ is the integer lattice for the maximal torus $T_{1}$
in the simply-connected group $G_{1}=[G,G]$ and and $L(\bar{T}_{1})$
is the integer lattice for $T_{1}/Z(G_{1})$ (the maximal torus of the
adjoint form of $G_{1}$). Let $\langle \dash,\dash \rangle_{2}$ denote
the inner product on $\ft_{1} = \Lie(T_{1})$ normalized so that the
coroots $\alpha^{\vv}$ corresponding to the long roots $\alpha$ for
$G_{1}$ have length $2$.  We then have
\[
                    \langle x,y\rangle = 2 h^{\vv} \langle
                    x,y\rangle_{2}
\]
where $h^{\vv}$ is the dual Coxeter number of $G_{1}$,
so the formula \eqref{eq:sign-change-1} can be rephrased as
\begin{equation}\label{eq:sign-change-2}
     \left(
                  h^{\vv} \langle \bar a \cupprod \bar b \rangle_{2}
                  + \frac{h^{\vv}}{2}  \langle  \bar{b}\cupprod
                  \bar{b} \rangle_{2}
                  + \weyl(\bar{b})\cupprod \weyl(\bar{b})
    \right)[X]
 \end{equation}
The pairing $\langle \dash,\dash \rangle_{2}$ gives a map
\[
                 L(\bar{T}_{1}) \times L(T_{1})    \to \Z
\]
and its restriction to the coarser lattice $L(T_{1})$ is an even form.
The Weyl vector $\weyl$ takes integer values on $L(T_{1})$, so each of
the three terms in \eqref{eq:sign-change-2} is an integer.

Let $p$ be the least
common multiple of the orders of the elements of $Z(G_{1})$: so $p$ is
$2$ in the case of $B_{n}$, $C_{n}$, $D_{2n}$ and $E_{7}$, is $4$ in
the case of $D_{2n+1}$, is $3$ in the case of $E_{6}$ and is $n+1$ in
the case of $A_{n}$. Condition~\ref{cond:G-conditions} ensures that the
map
\[
              \gdet:  Z(G) \to \bar{Z}(G)
\]
has the property
\[
               \gdet_{*}( \pi_{1}(Z(G)) = p \cdot
               \pi_{1}(\bar{Z}(G)).
\]
This condition tells us that $\bar{b}$ lies in
\[
 p \cdot H^{2}(X ;L(\bar{T}_{1}))/\mathrm{torsion}\subset
 H^{2}(X; L(T_{1}))/\mathrm{torsion}.
\]
We write
\[
\begin{aligned}
            \bar{b} &= p \bar{e} \\
                 \bar{e} &\in H^{2}(X;
                 L(\bar{T}_{1}))/\mathrm{torsion}.
\end{aligned}
\]
We also exclude the $A_{n}$ case from our discussion, because the
result of the Proposition for $\SU(N)$ is contained in
\cite{K-higherrank}. We examine the parity of the three terms in
\eqref{eq:sign-change-2}, beginning with the first term, the quantity
\[
                h^{\vv}\langle \bar{a}\cupprod \bar{b} \rangle_{2} [X].
\]
This is even if $h^{\vv}$ is even, and the remaining cases to look at
(with the above exclusions in mind) are $B_{n}$ for any $n$ and
$C_{n}$ for $n$ even. In both these cases, the pairing
$\langle\dash,\dash\rangle_{2}$ takes only even values on $L(\bar{T}_{1})
\times L(T_{1})$, so in all these cases this term is even. The
second term in \eqref{eq:sign-change-2} is also even if $h^{\vv}$ is
even. If $h^{\vv}$ is odd, then $p$ is even and the term can be expressed as
\[
            (h^{\vv}p/2) \langle \bar{e}\cupprod \bar{b}\rangle_{2} [X].
\]
As with the first term, we are dealing with $B_{n}$ or $C_{2m}$, and
the pairing is even by the same mechanism. The third term can be
written
\[
                p^{2} \weyl(\bar{e})\cupprod \weyl(\bar{e})[X].
\]
The case $A_{n}$ has been excluded, and in all other cases $\weyl$
takes integer values on $L(\bar{T}_{1})$. So $\weyl(\bar{e})$ is
an integral class. If $p$ is even, then this term is therefore even.
The only case where $p$ is odd is the case of $E_{6}$. In the $E_{6}$
case, one can check that there exists a coset representative $\bar{e}_{1}$ for a generator of
$L(\bar{T}_{1})/L(T_{1})\cong \Z/3$ such that $\weyl(\bar{e}_{1})$ is an
even integer. We can choose $\fe$ so that its image is spanned by this
representative, and with such a choice this third term is again even.
\end{proof}

\section{Instanton Floer homology for knots} 

\subsection{Configuration spaces and flat connections}
\label{subsec:config-spaces}

Let $Y$ be a closed, connected, oriented $3$-manifold, and let
$K\subset Y$ be an oriented knot or link. Take a simple,
simply-connected Lie group $G$, and let $P\to Y$ be a principal
$G$-bundle (necessarily trivial). Fix a maximal torus and a set of
positive roots as before, and choose $\Phi$ in the fundamental Weyl
chamber, satisfying the constraint \eqref{eq:highest-root-bound}.
Let $O \subset \g$ be its orbit, and let
$\varphi$ be a section of the associated bundle $O_{P}$ along $K$,
defining a  reduction of the structure group of $P|_{K}$ to the
subgroup $G_{\Phi}$. Any two choices of section $\varphi$ are
homotopic: the only topological data is in the pair $(Y,K)$ and the
choice of $G$ and $\Phi$.

We will equip $Y$ with a
Riemannian metric $g^{\nu}$ that is singular along $K$, as we did in dimension 4
in subsection~\ref{subsec:orbifold-metrics} above: the cone-angle will
be $2\pi / \nu$.
To ensure sufficient regularity, we 
let $I$ denote a compact interval in $(0,1)$ containing $\alpha(\Phi)$ for
all roots $\alpha$ in $R^{+}(\Phi)$, and take $\nu$ to  be at least as large as the
integer $\nu_{0}(I,2,m)$ supplied by
Proposition~\ref{prop:Fredholm-cone}, with $m$ a chosen Sobolev
exponent not less than $3$. (We will impose further restrictions on
$\nu$ shortly.)
We construct a model singular connection
$B^{\varphi}$ on the restriction of $P$ to $Y\sminus K$, just as in
the $4$-dimensional case (see \eqref{eq:singular-model-A}), and we
introduce the space of connections
\[
            \conf(Y,K,\Phi) = \{ \, B \mid B-B^{\varphi} \in
            \check{L}^{2}_{m,B^{\varphi}} \, \}.
\]
Here $\check{L}^{2}_{k,B^{\varphi}}$ denote the $3$-dimensional
Sobolev spaces defined just as in
subsection~\ref{subsec:orbifold-metrics}. Because they are trivial,
we omit $P$ and
$\varphi$ from our notation. There is a gauge group
\[
            \G(Y,K,\Phi) = \{ \, g \mid g \in
            \check{L}^{2}_{m+1,B^{\varphi}} \, \}
\]
and a quotient space
\[
                \bonf(Y,K,\Phi) = \conf(Y,K,\Phi) / \G(Y,K,\Phi).
\]
Two connections $B$ and $B'$ belonging to $\conf(Y,K,\Phi)$ are
gauge-equivalent as $G$-connections on $Y\setminus K$ if and only if
they differ by the action of an element of $\G(Y,K,\Phi)$.
As in the $4$-dimensional case, we call a connection $B$ reducible if
its stabilizer has positive dimension.

The space of connections $\conf(Y,K,\Phi)$ is an affine space, and on the
tangent space $T_{B}\conf$ we define an $L^{2}$ inner product
(independent of $B$) by
\begin{equation}\label{eq:inner-prod}
            \langle b, b' \rangle_{L^{2}} = \int_{Y}
             - \tr  (\ad(*b) \wedge \ad( b')),
\end{equation}
Thus we are using the Killing form to contract the
Lie algebra indices, and the Hodge star on $Y$ and the wedge product to
contract the form indices. The Hodge star is the one defined by the
singular metric $g_{^{\nu}}$.
We define the \emph{Chern-Simons functional} on
$\conf(Y,P,\Phi)$ to be the unique function
\[
            \CS : \conf(Y,K,\Phi) \to \R
\]
satisfying $\CS(B^{\varphi})=0$ and having formal gradient (with respect to
the above inner product)
\[
           ( \grad \CS)_{B} = *F_{B}.
\]
From this characterization, one can derive as usual the formula
\begin{equation}\label{eq:CS-formula}
                \CS(B^{\varphi}+b) = \bigl\langle *F_{B^{\varphi}}, b \bigr\rangle_{L^{2}}
                       +\frac{1}{2} \bigl\langle *d_{B^{\varphi}}b, b
                       \bigr\rangle_{L^{2}}
                       +
                       \frac{1}{3} \bigl\langle *[b\wedge b], b
                       \bigr\rangle_{L^{2}}.
\end{equation}
The Chern-Simons functional is independent of the choice of Riemannian
metric on $Y$, as can be seen by rewriting this formula using
\eqref{eq:inner-prod}.

The homotopy type of $\G(Y,K,\Phi)$ is that of the space of maps $g :
Y\to G$ with $g(K) \subset G_{\Phi}$. It follows that the the space of
components of the gauge group is
\[
            \pi_{3}(G) \times [ K , G_{\Phi} ]
\]
which is  isomorphic to
\begin{equation}\label{eq:pi1-of-bonf}
                \Z \oplus \LL(G_{\Phi})^{r}
\end{equation}
where $r$ is the number of components of the link $K$ and $\LL(G_{\Phi}) \subset
\fz(G_{\Phi})$ is the lattice of Definition~\ref{def:lattice}.  In
particular, there is a preferred homomorphism
\begin{equation}\label{eq:d-map}
            d:     \G(Y,K,\Phi)  \to \Z \oplus \LL(G_{\Phi}),
\end{equation}
where the map to the second factor is
obtained by taking the sum over all components of $K$.  An alternative
way to think of $d$ is to use a gauge transformation $g$ in
$\G(Y,K,\Phi)$ to form the bundle $S^{1} \times_{g} P$ over
$S^{1}\times Y$, together with its reduction $S^{1}\times_{g} \varphi$
over $S^{1}\times K$, defined by $\varphi$.  This data over
$(S^{1}\times Y, S^{1}\times \Sigma)$ has an instanton number $k$ and
monopole charge, as in the previous section (see
Definition~\ref{def:lattice} in particular). Then $d(g)$ can be
computed as $(k,l)$.

 The Chern-Simons
functional is invariant only under the identity component of the gauge
group. To express this quantitatively, let $B \in \conf(Y,K,\Phi)$ be
a connection,
let
$g$ be a
gauge transformation, and write $d(g) = (k,l) \in \Z\times \LL(G_{\Phi})$.
Then we have
\[
            \CS(B) - \CS(g(B)) =
            4\pi^{2}( 4h^{\vv} k + 2\langle \Phi,\ll \rangle).
\]

For a path $\gamma : [0,1] \to \conf(Y,K,\Phi)$, we  define the
\emph{topological energy} as twice the drop in the Chern-Simons
functional; so we can reinterpret the last equation as saying that a
path from $B$ to $g(B)$ has topological energy
\[
                \cE = 8\pi^{2}( 4h^{\vv} k + 2\langle \Phi,\ll
                \rangle).
\]
This is a formula which is familiar also for the energy of a solution
on any closed-manifold pair $(X,\Sigma)$ with $\Sigma\cdot\Sigma=0$.
For a path which formally solves the downward gradient-flow equation
for the Chern-Simons functional on $\conf(Y,K,\Phi)$, the topological
energy coincides with the modified path energy,
\[
                  \int_{0}^{1} \bigr ( \| \dot \gamma(t) \|^{2} +
                 \| \grad \CS(\gamma(t))\|^{2} \bigr) dt^{2}.
\]

From the definition of the Chern-Simons functional, it is apparent
that  critical points of $\CS$ are the flat connections in
$\conf(Y,K,\Phi)$.
The image of the critical points in the quotient space
$\bonf(Y,K,\Phi)$ can be identified with the quotient by the action of
conjugation of the space of all
homomorphisms
\begin{equation}\label{eq:homomorphisms}
            \rho : \pi_{1}(Y\sminus K) \to G
\end{equation}
with the property that the holonomy around each positively-oriented
meridian of $K$ is conjugate to $\exp(-2\pi \Phi)$. We shall write
\[
            \Crit \subset \bonf(Y,K,\Phi)
\]
for this set of critical points.

Reducible critical points of the Chern-Simons functional can be ruled
out on topological grounds, by the following criterion, whose proof
follows the same line as the $4$-dimensional version,
Proposition~\ref{prop:no-reducibles}.

\begin{proposition}\label{lem:coprime-sing-G}
    Let the components of $K$
    be $K_{1},\dots,K_{r}$. Suppose that for every fundamental
    weight $\w_{\alpha}$ and every choice of elements
    $\sigma_{1},\dots,\sigma_{r}$ in the Weyl group, the real
    cohomology class
\begin{equation}\label{eq:is-non-integral-3d}
            \sum_{j=1}^{r}(\w_{\alpha}\comp\sigma_{j})(\Phi)
            \mathrm{P.D.}[K_{j}]
\end{equation}
    is not integral. Then there are
       no reducible connections in the set of critical points
       $\Crit\subset\bonf(Y,K,\Phi)$.
\end{proposition}

Because the criterion in this proposition is referred to a few times,
we give it a name:

\begin{definition}\label{def:non-integral}
    We will say that $(Y,K,\Phi)$ satisfies the \emph{non-integral
    condition} if the expression \eqref{eq:is-non-integral-3d} is a
    non-integral cohomology class for every choice of fundamental weight $w_{\alpha}$
    and Weyl group elements
    $\sigma_{1},\dots,\sigma_{r}$.
\end{definition}

\subsection{Holonomy perturbations}
\label{subsec:holomy-perts}

We will perturb the Chern-Simons functional $\CS$ by adding a term
$f$: a real-valued function on $\conf(Y)$ invariant under the action
of the gauge group $\G(Y,K,\Phi)$. The type of perturbation that we
use is essentially the same as that used in \cite{Floer-original},
though similar constructions appear in \cite{Taubes-casson,
Donaldson-orientations} and elsewhere.

Let $q: S^{1} \times D^{2} \to Y\sminus K$ be a smooth immersion of a closed
solid torus.
Regard the circle $S^{1}$ as $\R/\Z$ and let $s$ be a corresponding
periodic coordinate.
Let $G_{P} \to Y$ be the bundle with fiber $G$ over $Y$ whose
sections are the gauge transformations of $P$, and for each $z$ in
$D^{2}$
let
\[
            \Hol_{q(\dash, z)}( B) \in (G_{P})_{q(0,z)}
\]
be the holonomy of the connection $ B$ around the corresponding
loop based at $q(0,z)$. As $z$ varies, we obtain in this way a section
$\Hol_{q}(B)$
of the bundle $q^{*}(G_{P})$ on the disk $D^{2}$.

Next suppose we have
an $r$-tuple of maps,
\[
                \bq = (q_{1},\dots, q_{r}),
\]
with $q_{j} : S^{1}\times D^{2} \to Y\sminus K$ an immersion. Suppose
further that there is some interval $[-\eta, \eta]$ such that the
restriction of $q_{j}$ to $[-\eta,\eta]\times D^{2}$
 is independent of $j$:
 \begin{equation}\label{eq:q-coincide}
                        q_{j}(s,z) = q_{j'}(s,z), \text{for all
                        $s$ with $|s|\le \eta$}.
 \end{equation}
The pull-back bundles $q^{*}_{j}(G_{P})$ are all canonically
identified with each other on the subset $[-\eta,\eta] \times D^{2}$, and
we can regard the holonomy maps as defining a section
\[
            \Hol_{\bq}(B) : D^{2} \to q_{1}^{*}(G^{r}_{P})
\]
of the $r$-fold fiber-product of the bundle $G_{P}$ pulled back to
$D^{2}$. Pick any smooth function
\[
                h : G^{r} \to \R
\]
that is invariant under the diagonal action of $G$ by the adjoint
action on the $r$ factors. Such an $h$ defines also a function on
$q_{1}^{*}(G^{r}_{P})$.  Let $\mu$ be a non-negative $2$-form supported
in the interior of $D^{2}$ and having integral $1$, and define
\begin{equation}\label{eq:cylinder-function-fmla}
                f(B) =
                \int_{D^{2}} h ( \Hol_{\bq}(B) ) \mu.
\end{equation}
A function $f$ of this sort is invariant under the gauge group action.

\begin{definition}\label{def:cylinder-function}
    A \emph{cylinder function} on $\conf(Y,K,\Phi)$ is a function \[f:
    \conf(Y,K,\Phi)\to \R\] of the form
    \eqref{eq:cylinder-function-fmla}, determined by an $r$-tuple of
    immersions as above and a $G$-invariant function $h$ on $G^{r}$.
\end{definition}

\begin{remark}
    In the transversality arguments that arise later, the important
    feature
    of the class of functions $f$ obtained in this way is that they
    separate points in the quotient space $\bonf(Y,K,\Phi)$, and also
    that they separate tangent vectors at points where the gauge
    action is free. There is a  slightly different class of
    perturbations that one can use and which serves just as well in
    the case that $G=\SU(N)$: one can drop the requirement that
    $q_{j}=q_{j'}$ on $[-\eta,\eta]\times D^{2}$, but instead put a
    more restrictive condition on $h$, namely that it be invariant
    under the action of $G$ on each of the $r$ factors separately.
    This alternative approach is laid out in detail in
    \cite{Donaldson-book}, where it is explained that such functions
    do separate points of $\bonf$ when $G=\SU(N)$: the key point is
    the following lemma.

    \begin{lemma}[\cite{Donaldson-book}]
    Suppose $h_{1}, \dots, h_{m}$ and $h'_{1},\dots,h'_{m}$ are elements of
    $\SU(N)$ and suppose that for all words $W$, the elements
    $W(h_{1},\dots, h_{m})$ and $W(h'_{1}, \dots, h'_{m})$ are
    conjugate. Then there is a $u\in \SU(N)$ such that $h'_{i} = u
    h_{i}u^{-1}$ for all $i$.
\end{lemma}

\noindent
This lemma fails for other groups: this is essentially the observation
of Dynkin \cite{Dynkin}, that two homomorphisms $f_{1},f_{2}:H\to G$
between compact Lie groups
may be \emph{linearly equivalent} without being equivalent, where
linear equivalence means that $\omega\comp f_{1}$ is equivalent to
$\omega\comp f_{2}$ for all linear representations $\omega$ of $G$.
The approach we have taken here is the one used in
\cite{Floer-original}, and works for any simple $G$.
\end{remark}

We  examine the formal gradient of such a cylinder function with
respect to our $L^{2}$ inner product on the tangent spaces of
$\conf(Y,K,\Phi)$. Let $\partial_{j} h$ be the partial derivative of
$h$ along the $j$'th factor: after trivializing the cotangent bundle
of $G$ using left-translation, we may regard this as a map
\[
                \partial_{j} h : G^{r} \to \g^{*}.
\]
Using the Killing form, we can also construct the $\g$-valued function
$(\partial_{j} h)_{\dag}$. The $G$-invariance means that this also
defines a map
\[
            (\partial_{j} h)_{\dag} : G^{r}_{P} \to \g_{P}.
\]
Let $H_{j}$ be the section of $q^{*}_{j}(\g_{P})$ on $D^{2}$ defined by
\[
                H_{j} = (\partial_{j} h)_{\dag}  (\Hol_{\bq}(B)).
\]
We extend $H_{j}$ to a section of $q^{*}_{j}(\g_{P})$ on all of
$S^{1}\times D^{2}$ by using parallel transport along the curves
$s\mapsto q(s,z)$: the resulting section $H_{j}$ has a discontinuity at $s=0$
because the parallel transport around the closed loops may be
non-trivial. The formal gradient of the cylinder function $f$, interpreted as
a $\g_{P}$-valued $1$-form on $Y\sminus K$, is then given by
\begin{equation}\label{eq:formula-for-gradient}
           *\Bigl(\sum_{j=1}^{r}  (q_{j})_{*} (H_{j} \mu) \Bigr). 
\end{equation}
Note that on $[-\eta,\eta]\times D^{2}$ we can regard each $H_{j}$ as a section of the same bundle
$q_{1}^{*}(\g_{P})$, and while each $H_{j}$ has a singularity at
$s=0$, the sum of the $H_{j}$'s does not, because of the
$G$-invariance of $h$. The above $1$-form is therefore continuous at
$q_{1}(\{0\}\times D^{2})$.

Our connections are of class $L^{2}_{m}$ away from the link $K$, and as in
\cite{Taubes-casson,K-higherrank} a short calculation shows that the
section defined by the holonomy is of the same class. So the
$\g_{P}$-valued $1$-form \eqref{eq:formula-for-gradient} is indeed in
$L^{2}_{m}$. It is also supported in a compact subset of $Y\sminus K$,
so it defines a  tangent vector to the space of connections
$\conf(Y,K,\Phi)$. As an abbreviation, let us write $\conf_{m}$ for
our space of connections modelled on $\check{L}^{2}_{m,B^{\varphi}}$,
and $\cT_{m}$ for its tangent bundle.
For
$k\le m$ we have the bundle $\cT_{k} \to \conf_{m}$ obtained by
completing the tangent bundle in the $\check{L}^{2}_{k}$ norm. 
We will
write $V$ for the formal gradient of the cylinder function $f$. The
following proposition details some of its properties.

\begin{proposition}\label{prop:cyl-functions}
    Let $f$ be a cylinder function and let $V$ be its formal gradient
    \eqref{eq:formula-for-gradient},
    regarded as a section of $\cT_{m}$ over $\conf(Y,K,\Phi)=\conf_{m}$. Then
    $V$ has the following properties:
      \begin{enumerate}
        \item  The formal gradient $V$ defines a smooth section,
        $
                        V \in C^{\infty}( \conf_{m} , \cT_{m})
        $.
        \item For any $j \le m$, the first derivative
        $
                    D V \in C^{\infty}(\conf_{m},
                    \Hom(\cT_{m},\cT_{m}))
        $
        extends to a smooth section
        \[
                    D V \in C^{\infty}(\conf_{m},
                    \Hom(\cT_{j},\cT_{j})).                    
        \]
        \item There is a constant $K$ such that
        $
                \|  V(B) \|_{L^{\infty}} \le K
        $
        for all $B$.
        \item For all $j$, there is a constant
        $K_{j}$ such that
        \[
                \|  V(B) \|_{\check{L}^{2}_{j,B^{\varphi}}} \le K_{j} \bigl(1 +
                \|B-B^{\varphi}\|_{\check{L}^{2}_{j,B^{\varphi}}}\bigr)^{j}.
        \]
        \item There exists a constant $C$ such
        that for all $B$ and $B'$, and all $p$
        with $1\le p \le \infty$, we have
        \[
                \|  V(B) - V(B') \|_{L^{p}} \le C
                \|B-B'\|_{L^{p}}.
        \]
        In particular, $V$ is continuous in the $L^{p}$ topologies.
      \end{enumerate}    
\end{proposition}

In order to work with a Banach space of perturbations $f$, we will
consider functions $f : \conf(Y,K,\Phi)\to \R$ which are obtained as the sum
of a series
\[
                f = \sum_{i=1}^{\infty} a_{i} f_{i}
\]
where the $a_{i}$ are real and $\{f_{i}\}_{i\in\mathbb{N}}$ is some
fixed countable collection of cylinder functions. We need to consider
whether such a sum is convergent. More specifically, we would like the
series of gradients $V_{i} = \grad f_{i}$ to converge to a section
\begin{equation}\label{eq:V-series}
            V = \sum_{i=1}^{\infty} a_{i} V_{i}
\end{equation}
of the tangent bundle to $\cT_{m}$, and we would like the limit $V$
to share with $V_{i}$ the properties detailed in the above proposition.
In the first part of this proposition, when we say that a section $V$
belongs to $C^{\infty}$, we do not imply that the norm of the $n$'th derivative
$D^{n}V|_{B}$  is uniformly bounded
independent of $B$. But the norm \emph{is} bounded by a function of
$\|B-B^{\varphi}\|_{\check{L}^{2}_{m,B^{\varphi}}}$: we have for each $n$  a
continuous function $h_{n}(\dash)$
such that
\[
                    \| D^{n}V|_{B}(b_{1}, \dots b_{n})
                    \|_{L^{2}_{l,B^{\varphi}}}
                    \le h_{n}(\|
                    B-B^{\varphi}\|_{\check{L}^{2}_{m,B^{\varphi}}}) \prod_{i=1}^{n}
                    \|b_{i}\|_{\check{L}^{2}_{m,B^{\varphi}}}.
\]
A similar remark applies to the derivatives of $DV$ in the norms that
appear in the second part of the proposition. Because of this, given
any countable collection of cylinder functions
$\{f_{i}\}_{i\in\mathbb{N}}$, we can find constants $C_{i}$ such that
the series \eqref{eq:V-series} converges whenever the sum
\[
                \sum C_{i} | a_{i} |
\]
converges, and such that the limit $V$ of the series is smooth.

Before proceeding further,  we shall choose a suitable countable collection of
cylinder functions $f_{i}$, sufficiently large to ensure that we can
achieve transversality. For each integer $r>0$, we choose a countable
set of $r$-tuples of immersions $q: S^{1}\times D^{2} \to Y$,
\[
                    (q_{1}^{r,j}, \dots, q_{r}^{r,j}), \quad
                    j\in\mathbb{N},                   
\]
satisfying \eqref{eq:q-coincide}
which are dense in the $C^{1}$ topology on the space of such
$r$-tuples of immersions. For each $r$, we also choose a collection of
smooth $G$-invariant functions $\{h^{r}_{k}\}_{k\in\mathbb{N}}$ on
$G^{r}$ which are dense in the $C^{\infty}$ topology.  Finally we combine
these to form a countable collection of cylinder functions
\[
                f_{j,k,r} = h^{r}_{k} (q_{1}^{r,j}, \dots,
                q_{r}^{r,j}).
\]

\begin{definition}\label{def:Pert}
    Fix a countable collection of cylinder functions $f_{i}$ and
    constants $C_{i}>0$ as above. Let $\Pert$ denote the separable Banach space of all
    real sequences $\pert=\{\pert_{i}\}_{i\in\mathbb{N}}$ such that the series
    \[
             \|\pert\|_{\Pert} \stackrel{\mathrm{def}}{=}       \sum_{i} C_{i} |\pert_{i}|
    \]
    converges. For each $\pert \in \Pert$, let $f_{\pert} = \sum \pert_{i} f_{i}$ be
    the corresponding function on $\conf(Y,K,\Phi)$, and let
    \[
                    V_{\pert} = \sum_{i} \pert_{i}V_{i}
    \]
    be the formal gradient of $f_{\pert}$ with respect to the $L^{2}$
    inner product. 
\end{definition}

What our discussion has shown is that, for suitable choice of
constants $C_{i}$, the series \eqref{eq:V-series} will converge and
the limit will also have the properties of cylinder functions that are
given in Proposition~\ref{prop:cyl-functions}. The next proposition
records this, together with the fact that the dependence of the
estimates on $\pert\in \Pert$ is as expected:

\begin{proposition}\label{prop:cyl-functions-param}
    If the constants $C_{i}$ in the definition of the Banach space $\Pert$
    grow sufficiently fast, then the family of sections $V_{\pert}$ of
    $\cT_{m}$ satisfies the following conditions.
          \begin{enumerate}
        \item \label{item:cyl-functions-param-1}
        The map
         \[
         \begin{aligned}
                                (\pert,B) &\mapsto V_{\pert}(B)
         \end{aligned}
          \]
          defines a smooth map
        $
                        V_{\bullet} \in C^{\infty}( \Pert\times
                        \conf_{m} , \cT_{m})
        $.
        \item\label{item:cyl-functions-param-2}
        For any $j \le m$, the first
        derivative in the $B$ variable,
        $
                    D V_{\bullet} \in C^{\infty}(\Pert\times \conf_{m},
                    \Hom(\cT_{m},\cT_{m}))
        $
        extends to a smooth section
        \[
                    D V_{\bullet} \in C^{\infty}(\Pert\times \conf_{m},
                    \Hom(\cT_{j},\cT_{j})).                    
        \]
        \item There is a constant $K$ such that
        $
                \|  V_{\pert}(B) \|_{L^{\infty}} \le K \|\pert\|_{\Pert}
        $
        for all $\pert$ and $B$.
        \item For all $j$, there is a constant
        $K_{j}$ such that
        \[
                \|  V_{\pert}(B) \|_{L^{2}_{j,B^{\varphi}}} \le K_{j} \| \pert \|_{\Pert}\bigl(1 +
                \|B-B^{\varphi}\|_{L^{2}_{j,B^{\varphi}}}\bigr)^{j}.
        \]
        \item There exists a constant $C$ such
        that for all $B$ and $B'$, and all $p$
        with $1\le p \le \infty$, we have
        \[
                \|  V_{\pert}(B) - V_{\pert}(B') \|_{L^{p}} \le C
               \|\pert\|_{\Pert} \|B-B'\|_{L^{p}}.
        \]
        \item The function $f_{\pert}$ whose gradient is $V_{\pert}$
        is bounded on $\conf_{m}(Y,K,\Phi)$.
      \end{enumerate}    
\end{proposition}

We will refer to a perturbation $f=f_{\pert}$ of the Chern-Simons
invariant which arises in this way as a \emph{holonomy perturbation}.

\subsection{Elliptic theory and transversality for critical points}

We fix a Banach space $\Pert$ parametrizing holonomy perturbations as
above, and we consider now the critical points of the perturbed
Chern-Simons functional $\CS+f_{\pert}$, for $\pert\in \Pert$. These critical
points are the connections $B$ in $\conf(Y,K,\Phi)$ satisfying
\begin{equation}\label{eq:perturbed-flat}
             * F_{B} + V_{\pert}(B) = 0.
\end{equation}
These equations are invariant under gauge transformation, and we
denote by
\[
            \Crit_{\pert}   \subset \bonf(Y,K,\Phi)
\]
the image of the set of critical points in the quotient space.
The familiar compactness properties of the space of flat connections modulo gauge
transformations extend to show that $\Crit_{\pert}$ is compact: we have,
more generally, the following lemma.

\begin{lemma}\label{lem:crit-compact}
    Let $\Crit_{\bullet} \subset \Pert \times \bonf(Y,K,\Phi)$ be the
    parametrized union of the critical sets $\Crit_{\pert}$ as $\pert$ runs
    through $\Pert$. Then the projection $\Crit_{\bullet} \to \Pert$ is
    proper.
\end{lemma}

\begin{proof}
    Suppose that $[B_{i}]$ belongs to $\Crit_{\pert_{i}}$ and $\pert_{i}$
    converges to $\pert$ in $\Pert$. 
    The terms $V_{\pert_{i}}(B_{i})$ are bounded in
    $L^{\infty}$, so the curvatures of the connections $B_{i}$ are
    bounded in $L^{\infty}$ also. By Uhlenbeck's theorem, we may
    assume (after replacing the $B_{i}$ by suitable gauge transforms)
    that the connection forms $B_{i}-B^{\varphi}$ are
    a bounded sequence in $\check{L}^{p}_{1}$, for all $p$.
    Uhlenbeck's theorem also allows us to find a
    finite covering of $Y$ by balls $U_{\alpha}$ (or orbifold balls
    centered at points of $K$) together with gauge
    transformations
    $g_{i,\alpha}$ such
    that the connections $B_{i,\alpha}=g_{i,\alpha}(B_{i})$
    are in $g^{\nu}$-Coulomb gauge with respect to some trivialization of
    $P|_{U_{\alpha}}$. The connection forms $B_{i,\alpha}$ on
    $U_{\alpha}$ are also bounded in $\check{L}^{p}_{1}$ for all $p$, and the
    gauge transformations $g_{i,\alpha}$ are bounded in $\check{L}^{p}_{2}$.

    The equations satisfied by $B_{i,\alpha}$ on $U_{\alpha}$ are
    \[
    \begin{aligned}
    *d B_{i,\alpha} &= - *[B_{i,\alpha}\wedge B_{i,\alpha}] -
                g_{i,\alpha} ( V_{\pert_{i}}(B_{i}))\\
                d^{*} {B_{i,\alpha}} &= 0.
                \end{aligned}
    \]
        The terms $[B_{i,\alpha}\wedge B_{i,\alpha}]$ are
    bounded in $\check{L}^{2}_{1}$, because of the continuity of the
    multiplication $\check{L}^{p}_{1} \times \check{L}^{p}_{1} \to
    \check{L}^{2}_{1}$ for $p>
    3$. The term $V_{\pert_{i}}(B_{i})$ is bounded in $\check{L}^{2}_{1}$, as is
    $g_{i,\alpha} ( V_{\pert_{i}}(B_{i}))$ therefore. On a smaller ball
    $U'_{\alpha}\subset U_{\alpha}$, these equations therefore give us
    an $\check{L}^{2}_{2}$ bound on $B_{i,\alpha}$. The bootstrapping argument
    now follows standard lines: on smaller balls $U''_{\alpha}$, the
    connections $B_{i,\alpha}$ are bounded in all $\check{L}^{2}_{j}$ norms,
    and after passing to a subsequence, the gauge transformations
    $g_{i,\alpha}$ can be patched together to form gauge
    transformations $g_{i}$ such that $g_{i}(B_{i})$ converges in the
    $\check{L}^{2}_{m}$ topology.
\end{proof}

For fixed $\pert$, the left-hand side of \eqref{eq:perturbed-flat} defines a smooth section,
\[
        B \mapsto *F_{B} + V_{\pert}(B)
\]
of the bundle $\cT_{m-1}\to \conf_{m}$. Because of the gauge invariance
of the functional, this section is everywhere orthogonal to the orbits
of the gauge group on $\conf(Y,K,\Phi)$, with respect to the $L^{2}$ inner
product. To introduce some notation to express this, let us first
write $\conf^{*}_{m}=\conf^{*}(Y,K,\Phi)$ as usual for the subset of $\conf(Y,K,\Phi)$ consisting
of irreducible connections, and let us decompose the restriction of
$\cT_{j}$ to $\conf^{*}_{m}$ as
\begin{equation}\label{eq:JK-decomp}
                \cT_{j} = \cJ_{j} \oplus \cK_{j},
\end{equation}
where $\cJ_{j,B}$ is the $\check{L}^{2}_{j}$ completion of the tangent space to
the gauge-group orbit through $B$ (i.e. the image of the map $u
\mapsto d_{B}u$) and $\cK_{j,B})$ its $L^{2}$ orthogonal complement in
$\check{L}^{2}_{j}(Y;\g_{P})$. Thus $\cK_{j,B}$ is the space of $b$ in
$\check{L}^{2}_{j}(Y;\g_{P})$ satisfying the Coulomb condition,
\[
                d^{*}_{B} b =0.
\]
On $\conf^{*}_{m}$, the decomposition \eqref{eq:JK-decomp} is a smooth
decomposition of a Banach vector bundle.
The gauge invariance of our perturbations means that $V_{\pert}$ is a
section the summand $\cK_{m} \subset \cT_{m}$ over $\conf^{*}_{m}$.

\begin{definition}
We say that a solution $B_{1} \in \conf^{*}(Y,K,\Phi)$ to the equation
\eqref{eq:perturbed-flat} is \emph{non-degenerate} if the section
$B \mapsto  *F_{B} + V_{\pert}(B)$ of the bundle $\cK_{m-1}\to
\conf^{*}_{m}$ is transverse to the zero section at $B=B_{1}$.
\end{definition}

The space of holonomy perturbations is sufficiently large to ensure
that all critical points will be non-degenerate for a suitably chosen
$\pert$:

\begin{proposition}\label{prop:residual-crit}
    There is a residual subset of the Banach space $\Pert$ such that for
    all $\pert$ in this subset, all the critical points of the perturbed
    functional $\CS + f_{\pert}$ in $\conf^{*}(Y,K,\Phi)$ are non-degenerate.
\end{proposition}

\begin{proof}
    For critical points whose stabilizer coincides with the center
    $Z(\G)$, this proposition follows from the fact that, given any
    compact (finite-dimensional) submanifold $S$ of
    $\bonf^{*}(Y,K,\Phi)$, the functions $f_{\pert}|_{S}$ are dense in
     $C^{\infty}(S)$. This argument is just as in
    \cite{Floer-original}, \cite{Donaldson-book} or \cite{KM-book}.
    For groups $G$ other than $\SU(N)$, we must deal with the fact
    that irreducible connections may have finite stabilizer larger
    than $Z(\G)$. Our holonomy perturbations are still dense in
    $C^{\infty}(S)$ for $S$ a compact sub-orbifold of
    $\bonf^{*}(Y,K,\Phi)$ however, so this extra complication can be
    dealt with as in  \cite{Wasserman}.
\end{proof}

This says nothing yet about the reducible critical points; but we will
be working eventually with configurations $(Y,K,\Phi)$ satisfying the
non-integral condition of Definition~\ref{def:non-integral}, and we
have the following version of Lemma~\ref{lem:coprime-sing-G} for small
perturbations of the equations:

\begin{lemma}\label{lem:coprime-pert}
     Suppose $(Y,K,\Phi)$ satisfies the non-integral condition.
    Then there exists $\epsilon>0$ such that for all $\pert$ with $\| \pert \|_{\Pert} \le
    \epsilon$, the critical points of $\CS + f_{\pert}$ in
    $\conf(Y,K,\Phi)$ are
    all irreducible.
\end{lemma}

\begin{proof}
    This follows from  Lemma~\ref{lem:coprime-sing-G} and the compactness
    result, Lemma~\ref{lem:crit-compact}.
\end{proof}

When a critical point $B$ is non-degenerate, its gauge orbit
$[B]$ in $\Crit_{\pert}$ is an isolated point of $\Crit_{\pert}$. It
therefore follows that if $\pert$ has norm less than $\epsilon$ and
belongs also to the residual set described in
Proposition~\ref{prop:residual-crit}, then $\Crit_{\pert}$ is a finite
subset of $\bonf(Y)$ and is contained in $\bonf^{*}(Y,K,\Phi) =
\conf^{*}(Y,K,\Phi)/\G(Y,K,\Phi)$. We record this as a proposition.

\begin{proposition}\label{prop:3d-transversality-summary}
    If $(Y,K,\Phi)$ satisfies the non-integral condition of
    Definition~\ref{def:non-integral}, then
    there exists $\epsilon > 0$ and a residual subset of the
    $\epsilon$-ball in $\Pert$, such that for all $\pert$ in this subset the
    set of critical points $\Crit_{\pert}$ in the quotient space
    $\bonf(Y,K,\Phi)$ is a finite set and consists only of non-degenerate,
    irreducible critical points.
\end{proposition}

Another way to look at the condition of non-degeneracy is to look at
the operator defined by the derivative of $\grad(\CS  + f_{\pert})$ on
$\conf(Y,K,\Phi)$, formally the Hessian of the functional. This Hessian is a
section
\[
                \Hess \in C^{\infty}(\conf_{m} , \Hom(\cT_{m}, \cT_{m-1}))
\]
and is given by
\[
                \Hess_{B}(b) = *d_{B} b + DV|_{B} (b).
\]
At a critical point $B$, the Hessian annihilates $\cJ_{B}$ and maps
$\cK_{B}$ to itself; and as an operator $\cK_{j,B} \to \cK_{j-1,B}$ it
is a compact perturbation of the Fredholm operator $*d_{B}$, because
$DV_{B}$ maps $\check{L}^{2}_{j}$ to $\check{L}^{2}_{j}$. As an unbounded self-adjoint
operator on
$\cK_{j,B}$ it has discrete spectrum: the spectrum consists of
eigenvalues, the eigenspaces are finite-dimensional and the sum of the
eigenspaces is dense. The non-degeneracy condition is the condition
that $\Hess_{B}$ is invertible, or equivalently the condition that $0$
is not an eigenvalue.

At a connection $B$ that is not a critical point of the perturbed
functional, the operator $\Hess_{B}$ does not leave invariant summands
$\cJ_{B}$ and $\cK_{B}$; and as an operator $\cT_{m,B}\to \cT_{m-1,B}$, it
is not Fredholm. To correct this, one can introduce the
\emph{extended Hessian}, which is the operator
\[
                    \EHess_{B} = 
                    \begin{bmatrix}
                        0 &  - d^{*}_{B} \\
                        -d_{B} & \Hess_{B}
                    \end{bmatrix}
\]
acting on the spaces
\[
            \EHess_{B} : \check{L}^{2}_{j}(Y;\g_{P}) \oplus \cT_{j} \to
           \check{L}^{2}_{j-1}(Y;\g_{P}) \oplus \cT_{j-1}.
\]
This is a self-adjoint Fredholm operator which varies smoothly with $B\in
\conf(Y,K,\Phi)$; it is a compact perturbation of the family of elliptic operators
\[
                \begin{bmatrix}
                        0 &  - d^{*}_{B} \\
                        -d_{B} & *d_{B}
                    \end{bmatrix}.
\]
The extended Hessian also has discrete spectrum consisting of
(real) eigenvalues of finite multiplicity; the sum of the eigenspaces
is again dense.
At a critical point $B$, the extended Hessian can be decomposed into
the direct sum of two operators, one of which is the restriction of
$\Hess_{B}$ as an operator $\cK_{j} \to \cK_{j-1}$. The other summand
is invertible at irreducible critical points. It follows that, for a
critical point $B$, the inevitability of $\EHess_{B}$ is equivalent
to the two conditions that $B$ be both irreducible and non-degenerate.

In the case that the perturbation is zero, the set of critical points
$\Crit\subset \bonf(Y,K,\Phi)$ can be identified with a space of
representations $\rho$ of the fundamental group of $Y\sminus K$, as explained at
\eqref{eq:homomorphisms}. In this situation, a representation $\rho$
determines a local coefficient system $\g_{\rho}$ on $Y\sminus
K$, with fiber $\g$. This has cohomology groups $H^{i}(Y\sminus
K;\g_{\rho})$. The following lemma (which is standard in the absence
of the knot $K$) provides a criterion for the non-degeneracy of
the corresponding connection $B$ as a critical point of $\CS$.

\begin{lemma}\label{lem:non-degenerate-rho}
    In the above situation, the kernel of $\Hess_{B}$ on $\cK_{j}$  is
    isomorphic to
    \[
                \ker: H^{1}(Y\sminus K; \g_{\rho}) \to
                    H^{1}(\underline{m} ; \g_{\rho})
    \]
    where $\underline{m}$ is any collection of loops representing the
    meridians of all the components of $K$. A critical point is
    therefore non-degenerate if and only if the above kernel is zero.
\end{lemma}

\begin{proof}
    The kernel of the Hessian on $\cK_{j}$ is isomorphic to
    $\ker(d_{B}) / \im(d_{B})$ on our function spaces on $Y$ with the
    orbifold metric. We can decompose $Y$ as a union of two pieces, one
    of which is a tubular neighborhood of $K$ and the other of which is the
    complement of a smaller neighborhood. The isomorphism of the lemma
    then follows from a Mayer-Vietoris sequence, using this
    decomposition.
\end{proof}

Suppose now that $B_{0}$ and $B_{1}$ are two irreducible,
non-degenerate critical points. Let $B(t)$ be a path in
$\conf(Y,K,\Phi)$
from $B_{0}$ to $B_{1}$. We define
\[
            \gr(B_{0}, B_{1}) \in \Z
\]
to be the spectral flow of the one-parameter family of operators
$\EHess_{B(t)}$. Because $\conf(Y,K,\Phi)$ is contractible, this number
does not depend on the path, but only on the endpoints. Now let
\[
                \beta_{0} = [B_{0}], \quad \beta_{1} = [B_{1}]
\]
be the corresponding critical points in the quotient space
$\bonf=\bonf(Y,K,\Phi)$.
The path $B(t)$ determines a path $\zeta$ from $\beta_{0}$ to
$\beta_{1}$. Let $z \in \pi_{1}(\bonf, \beta_{0},\beta_{1})$ be the
relative homotopy class of $\zeta$. The homotopy class of $z$ again
depends only on $B_{0}$ and $B_{1}$. To turn this around, if
$\beta_{0}$ and $\beta_{1}$ belong to $\Crit_{\pert} \subset
\bonf(Y,K,\Phi)$ and
are both irreducible and non-degenerate, and if $z$ is a relative
homotopy class of paths from $\beta_{0}$ to $\beta_{1}$, we define
\[
            \gr_{z}(\beta_{0},\beta_{1}) \in \Z
\]
to be equal to $\gr(B_{0}, B_{1})$, where $B_{0}$ and $B_{1}$ are the
endpoints of any path whose image in $\bonf(Y,K,\Phi)$ belongs to the
homotopy class $z$.

The fundamental group of $\bonf(Y,K,\Phi)$ is equal to the group of
components of $\G(Y,K,\Phi)$, which is described as
\eqref{eq:pi1-of-bonf} earlier. If $B$ is
a non-degenerate, irreducible critical point, and if $B'$ is obtained
from $B$ by applying a gauge transformation $g$ which is not in the
identity component of $\G(Y,K,\Phi)$, then a path from $B'$ to $B$ gives rise
to a homotopy class of closed loops $z$ based at the corresponding
point $\beta$ in $\bonf(Y,K,\Phi)$. We can compute the spectral flow around
this loop in terms of the data $(k,l) = d(g)$:

\begin{lemma}
    Let $B' = g(B)$ in $\conf(Y,K,\Phi)$, and write \[d(g) = (k,l) \in
    \Z \times L(G_{\Phi}).\]
    Then for the corresponding
    element $z$ of $\pi_{1}(\bonf(Y,K,\Phi), \beta)$ obtained from a
    path of connections from $B$ to $B'$, we have
    \[
            \gr_{z}(\beta,\beta) = 4h^{\vv} k + 4\weyl(\ll)
    \]
    where as usual $h^{\vv}$ is the dual Coxeter number of $G$ and
    $\weyl$ is the Weyl vector.
\end{lemma}

The proof of this lemma follows the expected line, reinterpreting the
spectral flow of the family of operators as the index of an operator
associated to $S^{1}\times Y$. The corresponding operator in this
context is (a perturbation of) the linearized anti-self-duality
equation with gauge fixing, so the index is the formal dimension of a
moduli space of singular instantons on the pair $(S^{1}\times Y,
S^{1}\times K)$. This relationship between the Chern-Simons functional
on $\conf(Y,K,\Phi)$ and singular instantons in dimension 4 is the
subject of the next subsection.

\subsection{The 4-dimensional equations and transversality for
trajectories}

Fix now some $\pert\in \Pert$, and write $V$ for $V_{\pert}$ and $f$ for $f_{\pert}$.
Let $B(t)$ be a path in $\conf(Y,K,\Phi)$, defined say on a bounded interval
$I\subset \R$. The path is a trajectory for the downward gradient flow
of $\CS + f$ if it satisfies the equation
\[
                \frac{d B}{dt} = - *F_{B} - V(B).
\]
The path $B(t)$ defines a  connection $A$ on the pull-back bundle over
$I\times Y$. This connection $A$ is in temporal gauge (that is, it has no $dt$
component when expressed in local trivializations obtained by
pull-back), and it has a singularity along $I\times K$ modelled on the
singular connection $A^{\varphi}$ obtained by pulling back
$B^{\varphi}$. In terms of $A$, the above equation can be written
\begin{equation}\label{eq:4d-equation}
                    F_{A}^{+} + (dt \wedge V(A))^{+} = 0,
\end{equation}
where $V(A)$ denotes the $1$-form in the $Y$ directions obtained by
applying $V$ to each $B(t)$, regarded as giving a $1$-form on $I\times
Y$.
In the form \eqref{eq:4d-equation},
the equation  is fully gauge-invariant under the
$4$-dimensional gauge group.

As an abbreviation, let us write $\hat{V}$ for the perturbing term
here,
\[
                    \hat{V}(A) = (dt \wedge V(A))^{+},
\]
so that the equations are
\begin{equation}\label{eq:perturbed-4d-short}
                F_{A}^{+} + \hat{V}(A) = 0.
\end{equation}

This perturbing term for the $4$-dimensional equations shares the same
basic properties as the perturbation $V(B)$ for the $3$-dimensional
equations. To state these, we suppose the interval $I$ is compact and
write
\[
\begin{aligned}
        Z &= I \times Y \\
        L &= I \times K
\end{aligned}
\]
so that $L$ is an embedded $2$-manifold with boundary in $Z$. We will
continue to write $P
\to Z$ for what is strictly the pull-back of $P$ from $Y$, and
$\varphi$ for the translation-invariant section of $O_{P}$ along $L$
obtained by pulling back the section $\varphi$ from $K$. 
We write $\conf_{m}(Z,L,P,\varphi)$ for the space of connections
singular connections $A = A^{\varphi}+a$ with $a$
of class $\check{L}^{2}_{m}$. As an abbreviation, and to distinguish
it from the similar space of $3$-dimensional connections
$\conf_{m}=\conf(Y,K,\Phi)$,
we write
\[
        \conff_{m} = \conf(Z,L,P,\varphi).
\]
In a similar way, we write
$\cTf_{j}$ for the $\check{L}^{2}_{j}$ completion
of the tangent bundle of $\conff_{m}$, and we write $\cS_{j} \to
\conff_{m}$ for the
(trivial) vector bundle with fiber $\check{L}^{2}_{j}(Z;
\g_{P}\otimes\Lambda^{+}(Z))$.
We assume from now on that $m$ is at least $3$, so that our
connections are again continuous on $Z\setminus L$.
Then
we have the following facts about $\hat{V}$, mirroring
Proposition~\ref{prop:cyl-functions-param}.

\begin{proposition}
    Let $A\mapsto \hat{V}(A)$ be the perturbing term for the
    $4$-dimensional equations, regarded as a section of the bundle
    $\cS_{m}$ over $\conff_{m}$. Then
      \begin{enumerate}
        \item  The section
        $\hat{V}$ is smooth:
        \[
                        \hat{V} \in C^{\infty}( \conff_{m} , \cS_{m}).
        \]
        \item For  any $j \le m$, the first derivative
        \[
                    D\hat{V} \in C^{\infty}(\conff_{m},
                    \Hom(\cTf_{m},\cS_{m}))
        \]
        extends to a smooth section
        \[
                    D \hat{V} \in C^{\infty}(\conff_{m},
                    \Hom(\cTf_{j},\cS_{j})).                    
        \]
        \item There is a constant $K$ such that
        $
                \| \hat{ V}(A) \|_{L^{\infty}} \le K
        $
        for all $A$.
        \item For all $j\le m$, there is a constant
        $K_{j}$ such that
        \[
                \|  \hat{V}(A) \|_{\check{L}^{2}_{j,A^{\varphi}}} \le K_{j} \bigl(1 +
                \|A-A^{\varphi}\|_{\check{L}^{2}_{j,A^{\varphi}}}\bigr)^{j}.
        \]
        \item There exists a constant $C$ such
        that for all $A$ and $A'$, and all $p$
        with $1\le p \le \infty$, we have
        \[
                \|  \hat{V}(A) - \hat{V}(A') \|_{L^{p}} \le C
                \|A-A'\|_{L^{p}}.
        \]
        In particular, $\hat{V}$ is continuous in the $L^{p}$ topologies.
      \end{enumerate}
      In each of these cases, the dependence on $\pert\in \Pert$ can also be
      included, as in the statement of
      Proposition~\ref{prop:cyl-functions-param}. 
\end{proposition}

For a solution $A$ in $\conff_{m}$ on a compact cylinder $Z =
[t_{0},t_{1}] \times
Y$, we define the (perturbed) topological energy
as twice the  change in the functional
$\CS + f_{\pert}$: that is,
\begin{equation}\label{eq:perturbed-energy-def1}
                \cE_{\pert}(A) =2 \bigl( (\CS + f_{\pert})(B(t_{0})) -  (\CS +
                f_{\pert})(B(t_{1})) \bigr),
\end{equation}
where $B(t)$ is the $3$-dimensional connection obtained by
restricting $A$ to $\{t\}\times Y$. Because of the last condition in
Proposition~\ref{prop:cyl-functions-param}, the perturbing term here
only affects the energy by a bounded amount. The Chern-Simons
functional is invariant only under the identity-component of the gauge
group, so $\cE_{\pert}(A)$ is not determined by knowing only the
gauge-equivalences classes of the two endpoints,
$\beta_{0}=[B(t_{0})]$ and $\beta_{1}=[B(t_{1})]$. The energy
\emph{is} determined by the endpoints $\beta_{0}$, $\beta_{1}$ in
$\bonf(Y,K,\Phi)$ and the homotopy class of the path $z$ between them
given by $[B(t)]$. Accordingly, we may write the energy as
\[
            \cE_{z}(\beta_{0},\beta_{1}).
\]

We turn next to the Fredholm theory for solutions to the perturbed
equations on the infinite cylinder. We write
\[
\begin{aligned}
        Z &= \R \times Y \\
        L &= \R \times K.
\end{aligned}
\]
Let us suppose that the holonomy perturbation is chosen as in
Proposition~\ref{prop:3d-transversality-summary}, so that the critical
points are irreducible and non-degenerate. Let $\alpha$ and $\beta$ be
two elements of $\Crit_{\pert}$ and $z$ a homotopy class of paths between
them. Let $B_{\alpha}$ and $B_{\beta}$ be corresponding elements of
$\conf(Y,K,\Phi)$, chosen so that a path from $B_{\alpha}$ to $B_{\beta}$
projects to $\bonf(Y,K,\Phi)$ to give a path belonging to the class $z$.
Let $A_{0}$ be a singular connection on the pull-back of $P$ to
the infinite cylinder $Z$, such that the restrictions of
$A_{0}$ to $(-\infty, -T]$ and $[T,\infty)$ are equal to the pull-back
of $B_{\alpha}$ and $B_{\beta}$ respectively, for some $T$. Define
\begin{equation}\label{eq:conf-z-def}
        \conf_{z}(\alpha,\beta)
        = \{\, A \mid  A - A_{0} \in \check{L}^{2}_{m,A_{0}}(Z;
        \g_{P}\otimes \Lambda^{1}(Z)) \, \}.
\end{equation}
This space depends on the choice of $A_{0}$, not just on $\alpha$,
$\beta$ and $z$. But any two choices are related by a gauge
transformation. We define $\G(Z)$ to be the group of gauge
transformations $g$ of $P$ on $Z$ satisfying
\[
                g - 1 \in \check{L}^{2}_{m+1,A_{0}}(Z; G_{P}),
\]
and we have the quotient space
\[
        \bonf_{z}(\alpha,\beta) =   \conf_{z}(\alpha,\beta) / \G(Z). 
\]
It is an important consequence of the non-degeneracy of the critical
points, that every solution $A$ to the perturbed equations on
$\R\times Y$ which has finite total energy is gauge-equivalent to a
connection in $\conf_{z}(\alpha,\beta)$, for some $\alpha$, $\beta$
and $z$.

\begin{definition}\label{def:moduli-space}
    The moduli space $M_{z}(\alpha,\beta) \subset
    \bonf_{z}(\alpha,\beta)$ is the space of
    gauge-equivalence classes of solutions to the perturbed equations,
    $F^{+}_{A} + \hat{V}(A) = 0$.
\end{definition}

Because we have assumed that all critical points are irreducible, the
configuration space $\conf_{z}(\alpha,\beta)$ consists also of
irreducible connections. The action of the gauge group therefore has
only finite stabilizers, and
$\bonf_{z}(\alpha,\beta)$ is a
Banach orbifold (or a Banach manifold in the case that $G=\SU(N)$):
coordinate charts can be obtained in the usual
way using the Coulomb condition. 
The local structure of
$M_{z}(\alpha,\beta)$ is therefore governed by the linearization of
the perturbed equations together with the Coulomb condition: at a
solution $A$ in $\conf_{z}(\alpha,\beta)$, this is the operator
\begin{equation}\label{eq:QA-operator}
                Q_{A} =    - d^{*}_{A} \oplus \bigl (d^{+}_{A} +
                    D\hat{V}|_{A}\bigr)
\end{equation}
from  $\check{L}^{2}_{m,A_{0}}(Z;\g_{P}\otimes \Lambda^{1}(Z))$  to
 $\check{L}^{2}_{m-1,A_{0}}(Z;\g_{P}\otimes
 (\Lambda^{0}\oplus\Lambda^{+})(Z))$.

This operator is Fredholm, and its index is equal to the spectral flow of the
extended Hessian:
\[
            \ind Q_{A} = \gr_{z}( \alpha,\beta).
\]

\begin{definition}
    A solution $A$ to the perturbed equations in
    $\conf_{z}(\alpha,\beta)$ is \emph{regular} if $Q_{A}$ is
    surjective. We say that the moduli space $M_{z}(\alpha,\beta)$ is
    regular if $A$ is regular for all $[A]$ in the moduli space.
\end{definition}

If the moduli space is regular, then it is a (possibly empty) smooth
orbifold of dimension $\gr_{z}(\alpha,\beta)$. At this point, one
would like to argue that for a generic choice of perturbation $\pi$,
all the moduli spaces $M_{z}(\alpha,\beta)$ are regular. However,
although such a result is true for the case of $\SU(2)$ (and is proved
in \cite{Floer-original} and \cite{Donaldson-book}), the presence of
non-trivial finite stabilizers is an obstruction to extending the
transversality arguments to general simply-connected simple groups
$G$. When the stabilizers are all equal to the center $Z(\G)$, then
the arguments from the $\SU(2)$ case carry over without change. We
therefore have:

\begin{proposition}\label{prop:4d-transversality}
    Suppose that $\pi_{0}$ is a perturbation such that  all the
    critical points in $\Crit_{\pi_{0}}$ are non-degenerate and have
    stabilizer $Z(\G)$.   Then there exists $\pert\in \Pert$ such that:
    \begin{enumerate}
        \item $f_{\pert}=f_{\pert_{0}}$ in a neighborhood of all the critical
        points of $\CS+f_{\pert_{0}}$;
        \item the set of critical points for these two perturbations
        are the same, so $\Crit_{\pert} = \Crit_{\pert_{0}}$;
        \item for all critical points $\alpha$ and $\beta$ in
        $\Crit_{\pert}$ and all paths $z$, the moduli spaces
        $M_{z}(\alpha,\beta)$ for the perturbation $\pert$
        are regular.
        \end{enumerate}
\end{proposition}

In order to proceed, we will require non-degeneracy for all critical
points and regularity for all moduli spaces. We therefore impose the
following conditions:

\begin{hypothesis}\label{hyp:stabilizers}
    We will assume henceforth that the triple $(Y,K,\Phi)$ satisfies
    the non-integral condition, and that a small perturbation $\pert$
    is chosen as in Proposition~\ref{prop:3d-transversality-summary}
    so that the critical points are irreducible and non-degenerate. We
    assume furthermore that the stabilizer of each critical point is just
    $Z(\G)$, and that the moduli spaces $M_{z}(\alpha,\beta)$ are all
    regular, as in the previous proposition.
\end{hypothesis}

In practice, we do not know how to ensure the  condition in
Hypothesis~\ref{hyp:stabilizers} that the stabilizers be $Z(\G)$
except by taking $G=\SU(N)$, in which
case it is automatic, given the other conditions. Of course, for any
given $(Y,K,\Phi)$, it is always possible that this  condition is
satisfied, as it were, ``by accident''; but from this point on we
really have $\SU(N)$ in mind. The notation we have set up for a
general simply-connected simple group $G$ is still appropriate,  and
we will continue to use it.

\subsection{Compactness and bubbles}

The basic compactness results for singular instantons on a compact
pair $(X,\Sigma)$, which we summarized in
Proposition~\ref{prop:Uhlenbeck}, can be adapted to the case of
solutions on a compact cylindrical pair
\[
 \begin{aligned}
        Z &= I \times Y \\
        L &= I \times K.
\end{aligned}
\]
The main differences from the closed case are the following. First,
in the
case of a closed manifold, the energy $\cE$ is entirely determined by
the topology of $P$ and $\varphi$. In the case of a finite cylinder,
the energy depends on the (perturbed) Chern-Simons invariants of the
restriction of the connection to the two boundary components, and
is therefore not constrained by the topology: in order to obtain a
compactness results we need to impose a bound on the energy as part of
the hypotheses. Second, since the proofs ultimately
depend on interior estimates, the hypothesis of bounded energy for a
sequence of solutions on $Z$ will only ensure that we have a
subsequence converging on some interior domain. Third, when bubbles
occur, their effect is no longer local, because of our non-local
holonomy perturbations. None of these issues are special to the case
of instantons with singularities: they all occur in the standard
construction of instanton Floer homology, and the issues surrounding
the non-local holonomy perturbations are treated in
\cite{Donaldson-book} and \cite{K-higherrank}.

In the statement of the following proposition (which corresponds to
the first parts of Proposition~\ref{prop:Uhlenbeck}), the $\check{L}^{p}_{k}$
topology refers to the topology on sections of $\g_{P}\otimes
\Lambda^{i}$ defined by using the covariant derivative of
$A^{\varphi}$ and the Levi-Civita derivative of the orbifold metric
$g^{\nu}$, just as $\check{L}^{2}_{k}$ was defined earlier.

\begin{proposition}\label{prop:Uhlenbeck-cylinder}
    Let $(Z,L)$ be the compact cylindrical pair defined above,  and let $I''
    \subset I$ be 
    a compact sub-interval contained in the interior of $I$. Let
    $A_{n}$ be a sequence of solutions to the perturbed equations in
    $\conf(Z,L,P,\varphi)$, and suppose there is a uniform bound on the energy:
    \[
                \cE_{\pert}(A_{n}) \le C, \qquad\text{for all $i$}.
    \]
    Then after passing to a subsequence, we have the following
    situation. There is an interval $I'$ with $I'' \subset I' \subset
    I$, a finite set of points $\bx$ contained in the interior of the
    sub-cylinder $Z' = I'\times Y$, and a solution $A$ to the equations
    in $\conf(Z',L', P,\varphi)$, with the following properties.
    \begin{enumerate}
        \item There is a sequence of isomorphisms of bundles $g_{n}:
        P|_{Z'\sminus\mathbf{x} } \to P$ of class
        $\check{L}^{2}_{m+1}$ such that
        \[
                            g_{n}^{*}(A_{n}) \to
                            A|_{Z'\sminus\mathbf{x}}
        \]
        in the $\check{L}^{p}_{1}$ topology
        on compact subsets of $Z'\sminus\mathbf{x}$ for all $p>1$.

        \item In the sense of measures on $Z'$, the energy densities $2
        |F_{A_{n}}|^{2}$ converge to
        \[
                           2 |F_{A}|^{2} + \sum_{x\in
                           \mathbf{x}}\mu_{x}
                           \delta_{x}
        \]
        where $\delta_{x}$ is the delta-mass at $x$ and $\mu_{x}$ are
        positive real numbers.
    \end{enumerate}
\end{proposition}

The reason for passing from a subinterval $I''$ to a larger one $I'$
in the statement above is to ensure that the set of bubble-points
$\bx$ is contained entirely in the interior of $Z'$. This means in
particular that the gauge-transformations $g_{n}$ in the statement of
the proposition are defined on the
two boundary components of $Z'$. Let us write $I'$ as $[t_{0}',
t_{1}']$, so that the boundary components are $\{t'_{0}\}\times Y$ and
$\{t'_{1}\}\times Y$. Using again the map $d$ defined at
\eqref{eq:d-map}, we can consider the elements
\[
                d( g_{n}|_{t'_{i}} ) \in \Z \oplus \LL(G_{\Phi})
\]
for $i=0,1$. If $\bx$ were empty, then these two would be equal, but
in general the difference is a topological quantity accounted for by
the failure of $g_{n}$ to extend over the punctures. This is the same
phenomenon that accounts for the difference between $(k,l)$ and
$(k',l')$ in item~\ref{item:charge-change} of
Proposition~\ref{prop:Uhlenbeck}. Combining
Proposition~\ref{prop:Uhlenbeck} with
Proposition~\ref{prop:energy-inequalities}, we therefore obtain:

\begin{proposition}
    In the situation of Proposition~\ref{prop:Uhlenbeck-cylinder}, we
    can choose the subsequence so that the elements $d( g_{n}|_{t'_{i}} )$ are
    independent of $n$ for $i=0,1$; and for each $x\in \bx$, we
    can find $(k_{x},l_{x}) \in \Z\oplus \LL(G_{\Phi})$ such that
    \[
          d( g_{n}|_{t'_{0}} ) -
          d( g_{n}|_{t'_{1}} )  = \Bigl( \sum_{x\in \bx}k_{x},
          \sum_{x\in \bx} l_{x}\Bigr).
    \]
    (If $x$ does not lie on the surface
    $L'=I'\times K$, then $\ll_{x}$ is zero.)
    The energy $\mu_{x}$ that is lost at $x$ is then given by
    \[
                    \mu_{x} = 8\pi^{2}\bigl( 4 h^{\vv} k_{x} + 2\langle
                    \Phi,l_{x}\rangle\bigr).
    \]
    Furthermore, the pairs $(k_{x}, \ll_{x})$ are subject to the
    constraints of Proposition~\ref{prop:energy-inequalities}, namely
    \[
    \begin{gathered}
    k_{x} \ge 0, \quad \text{and} \\
                    n^{\vv}_{\alpha}k_{x} + \w_{\alpha}(\ll_{x}) \ge 0
                    \end{gathered}
    \]
    for all simple roots $\alpha$. \qed
\end{proposition}

As in the case of
a closed manifold, the energy lost at the bubbles is accounted for by
solutions on the pair $(S^{4}, S^{2})$ (equipped now with a
conformally-flat orbifold
metric as in \cite{KM-gtes-I}).

The compactness results above, for solutions on a  compact cylinder,
lead in a standard way to compactness results for solutions on the infinite cylinder
$\R\times Y$ when transversality hypotheses are assumed, as in
Hypothesis~\ref{hyp:stabilizers}.
To introduce notation for this, if $z$ is not the class of a
constant path at $\alpha=\beta$, we let $\Mu_{z}(\alpha,\beta)$
denote
the quotient $M_{z}(\alpha,\beta)/\R$, where $\R$ acts by
translations. (For $\alpha=\beta$ and $z$ the constant path, we regard
$\Mu_{z}(\alpha,\beta)$ as the empty set.) We call the elements of
$\Mu_{z}(\alpha,\beta)$ the \emph{unparametrized trajectories}. By a
\emph{broken} (unparametrized) trajectory from $\alpha$ to $\beta$, we
mean a collection
\[
                [A_{i}] \in \Mu_{z_{i}}(\beta_{i-1},\beta_{i})
\]
for $i=1,\dots,l$, with $\beta_{0}=\alpha$ and $\beta_{l}=\beta$. The
case $l=0$ is allowed. We write $\Mubk_{z}(\alpha,\beta)$ for the
space of all unparametrized broken trajectories from $\alpha$ to
$\beta$ with the additional property that the composite of the paths
$z_{i}$ is in the homotopy class $z$.

In finite-dimensional Morse theory, the spaces of broken trajectories
of this sort are compact. For the instanton theory, compactness holds
only in situations where we can rule out the possibility that bubbles
may occur. Given our transversality hypotheses, we can rule out
bubbles on the grounds of the dimension of the moduli spaces involved.
In particular, from Corollary~\ref{cor:lose-4}, we deduce:

\begin{proposition}\label{prop:broken-compactness}
    If the dimension of $M_{z}(\alpha,\beta)$ is less than $4$, then
    the space of unparametrized broken trajectories
    $\Mubk_{z}(\alpha,\beta)$ is compact. In particular, if
    $\gr_{z}(\alpha,\beta)=1$, then $\Mu_{z}(\alpha,\beta)$ is a
    compact zero-dimensional manifold.
\end{proposition}

The bound of $4$ in this proposition can be improved in particular
cases, depending on the group $G$ and the choice of $\Phi$. The
correct condition in general is that $\gr_{z}(\alpha,\beta)$ is
smaller than the smallest dimension of any positive-dimensional framed moduli
space on $(S^{4},S^{2})$. See
Corollary~\ref{cor:at-least-2N} for example.

There is a significant additional question that does not arise in the
case that $K$ is absent. The compactness result that we have just
stated concerns a  single moduli space. There are only finitely many
critical points, but for each pair $(\alpha, \beta)$ there are
infinitely many possibilities for $z$. When $K$ is empty,
$\pi_{1}(\bonf(Y))$ is $\Z$ and $\gr_{z}(\alpha,\beta)$ is a
non-constant linear function of $z$: the moduli space will be empty
when $\gr_{z}(\alpha,\beta)$ is negative, and one should expect the moduli
space to be non-empty (and of large dimension) once
$\gr_{z}(\alpha,\beta)$ becomes large. When $K$ is present,
$\pi_{1}(\bonf(Y,K,\Phi))$ is larger, and knowledge of $\gr_{z}(\alpha,\beta)$
no longer determines $z$. There may be infinitely many non-empty
moduli spaces, all of the same dimension. What we do have is a
finiteness result when a bound on the energy is known. Let us again write
\[
            \cE_{z}(\alpha,\beta)
            = 2\Bigl(  (\CS + f_{\pi})(B_{\alpha}) -
            (\CS + f_{\pi})(B_{\beta})\Bigr);
\]
for the (perturbed) topological energy along a homotopy class of paths $z$.
This is the energy for any solution in the moduli
space $M_{z}(\alpha,\beta)$. For a proof of the following finiteness
result, see \cite[Proposition-something]{KM-book}.

\begin{proposition}\label{prop:energy-finite}
    Given any $C>0$, there are only finitely many $\alpha$, $\beta$
    and $z$ for which the moduli space $M_{z}(\alpha,\beta)$ is
    non-empty and has topological energy at most $C$.
\end{proposition}

For the construction of the Floer homology, the important comparison
is between the topological energy $\cE_{z}(\alpha,\beta)$ and the
spectral flow, or relative grading, $\gr_{z}(\alpha,\beta)$.
We can look at the special case where $\alpha=\beta$ so that a lift of
a path in the class $z$ gives a path of connections on $Y\sminus K$
from $B$ to $B'$, where $B'$
differs from $B$ by a gauge transformation $g\in \G(Y,K,\Phi)$. We
write $B' =
g(B)$. Let us again set
\[
            d(g) = (k,\ll) \in \Z \times \LL(G_{\Phi})
\]
as in \eqref{eq:d-map}. Then for the corresponding homotopy class $z$ of
closed paths based at $\beta=[B]$, we have
\[
                \gr_{z}(\beta,\beta) =
                       4 h^{\vv} k + 4\weyl(\ll)
\]
and
\[
                \cE_{z}(\beta,\beta) =
                         8 \pi^{2}\bigl ( 4  h^{\vv} k + 2 \langle
                         \Phi, \ll \rangle \bigr ).
\]
These formulae can  be computed, for example, by applying the
dimension and energy formulae for the closed manifold $S^{1}\times Y$
containing the embedded surface $S^{1}\times K$. In the case that
$\Phi$ satisfies the monotone condition
(Definition~\ref{def:monotone}), these two linear forms in $k$
and $\ll$ are proportional. Since there are only finitely many
critical points in all, we see:

\begin{lemma}
    If $\Phi$ satisfies the monotone condition, then there is a
    constant $C_{0}$ such that  for all $\alpha$, $\beta$ and $z$, we have
    \[
                    \Bigl | \cE_{z}(\alpha,\beta)-  8 \pi^{2}
                    \gr_{z}(\alpha,\beta)\Bigr| \le C_{0} .
    \]
\end{lemma}

From Proposition~\ref{prop:energy-finite} we now deduce:

\begin{corollary}\label{cor:energy-finite-monotone}
     If $\Phi$ satisfies the monotone condition, then
    given any $D>0$, there are only finitely many $\alpha$, $\beta$
    and $z$ for which the moduli space $M_{z}(\alpha,\beta)$ is
    non-empty and has formal dimension at most $D$.
\end{corollary}

\subsection{Orientations}
\label{subsec:orientations}

If $\alpha$ and $\beta$ are not necessarily critical points, we can
still construct the operator $Q_{A}$ from an arbitrary $A$
corresponding to a path $\zeta$ joining $\alpha$ to
$\beta$. The operator is
Fredholm if the extended Hessian is invertible at both $\alpha$ and
$\beta$. Under these circumstances, let us define
\[
            \Lambda_{\zeta}(\alpha,\beta)
\]
to be the (two-element) set of orientations for the determinant line
of the Fredholm operator $Q_{A}$. As $\zeta$ varies in the paths
belonging to a particular homotopy class of paths $z$, the family of
determinant lines of the corresponding operators $Q_{A}$ forms an
orientable real line bundle over the space of paths: this
orientability can be deduced from the corresponding statement in the
case of a closed manifold, Proposition~\ref{prop:orientable-sing}.
An orientation
for any one determinant line in this connected family therefore
determines an orientation for any other.  Thus it makes sense to write
\[
             \Lambda_{z}(\alpha,\beta)
\]
in place of $\Lambda_{\zeta}(\alpha,\beta)$, for $z$ a homotopy class
of paths from $\alpha$ to $\beta$. If $z'$
is a homotopy class of paths from $\beta$ to $\beta'$, then there is a
natural composition law,
\[
            \Lambda_{z}(\alpha,\beta) \times
            \Lambda_{z'}(\beta,\beta') \to \Lambda_{z'\comp
            z}(\alpha,\beta')  .
\]
(Note that our notation for a composite path puts the first path on
the right.) Because of the requirement that the Hessian is invertible
at the two end-points, the two-element set $\Lambda_{z}(\alpha,\beta)$
cannot be thought of as depending continuously on $\alpha$ and $\beta$
in $\bonf(Y,K,\Phi)$.

A priori, $\Lambda_{z}(\alpha,\beta)$ depends on $z$, not just on
$\alpha$ and $\beta$; but we can specify a rule, compatible with the
composition law, that identifies $\Lambda_{z}(\alpha,\beta)$ and
$\Lambda_{z'}(\alpha,\beta)$ for different homotopy classes $z$ and
$z'$. This can be done, for example, using excision to transfer the
question to a closed pair $(X,\Sigma)$ and then using the
constructions which were used to compare orientations in
Proposition~\ref{prop:orientable-sing}.  This observation allows us to
write $\Lambda(\alpha,\beta)$, omitting the $z$.

Because $G$ is simple, the bundle $P$ admits a product connection
$B^{0}$ for which $\varphi$ is parallel. We add a singular term in the
standard way, to obtain a connection $B^{\varphi}$ with a
codimension-two singularity; the monodromy of this connection lies in
the one-parameter subgroup generated by $\Phi$. We let
$\theta^{\varphi}$ denote the corresponding point in
$\bonf(Y,K,\Phi)$. This point is neither irreducible or
non-degenerate, so we cannot define $\Lambda(\theta^{\varphi},\alpha)$
as above because the operator $Q_{A}$ will not be Fredholm as it
stands. To remedy this, we
 we can regard $Q_{A}$ as acting weighted Sobolev space, on which this
operator is Fredholm. That is, we choose a connection
$A$ in $\conf_{\loc}$ from $B^{\varphi}$ to $B_{\alpha}$ and define
$\Lambda(\theta^{\varphi},\alpha)$ as the set of orientations of the
determinant line of the operator $Q_{A}$ acting in the topologies
\[
                   Q_{A} : e^{-\epsilon t}\check{L}^{2}_{m,A_{0}}
                   \to e^{-\epsilon t}\check{L}^{2}_{m-1,A_{0}}
\]
on the infinite cylinder. Here $\epsilon$  is a small positive
constant, smaller than the smallest positive
eigenvalue of the extended Hessians $\theta$ and $\alpha$.

Now that we have a basepoint $\theta^{\varphi}$, we can define a
$2$-element set 
\[
            \Lambda(\alpha) = \Lambda(\theta^{\varphi}, \alpha).
\]
 We could
equally well define $\Lambda(\alpha)$ as
$\Lambda(\alpha,\theta^{\varphi})$
(with the same weighted Sobolev spaces), because of the composition
law and the fact that $\Lambda(\alpha,\alpha)$ is canonically trivial.

With this
understood, the composition law for the orientation lines gives us a
map
\[
            \Lambda(\alpha) \times \Lambda(\beta) \to
            \Lambda(\alpha,\beta).
\]
If $\alpha$ and $\beta$ are now critical points and $[A]$ is a
solution of the equations belonging to the (regular) moduli space
$M_{z}(\alpha,\beta)$, then $\Lambda(\alpha,\beta)$ is isomorphic to
the set of orientations of the moduli space at $[A]$. Using the above
composition law, we can turn this round and say that an orientation of
$M_{z}(\alpha,\beta)$ at $[A]$ determines an isomorphism
$\Lambda(\alpha) \to \Lambda(\beta)$.

In particular, we can consider the case that
$\gr_{z}(\alpha,\beta)=1$. In this case, the moduli space of
unparametrized trajectories $\Mu_{z}(\alpha,\beta)$ is a  finite set
of points, and $M_{z}(\alpha,\beta)$ is a finite set of copies of
$\R$, acted on by the translations of the cylinder. Thus
$M_{z}(\alpha,\beta)$ is canonically oriented. To be quite specific,
if $\tau_{t}$ denotes the translation $(s,y) \mapsto (s+t,y)$ of
$\R\times Y$, we make $\R$ act on $M_{z}(\alpha,\beta)$ by $[A]\mapsto
\tau^{*}_{t}[A]$, and we use this to give each orbit of $\R$ an
orientation. For each $[A]$ in $M_{z}(\alpha,\beta)$, we therefore
obtain an isomorphism
\begin{equation}\label{eq:epsilon-sign}
                    \epsilon[A] : \Lambda(\alpha) \to \Lambda(\beta).
\end{equation}

\subsection{Floer homology}
\label{subsec:Floer-basic}

We can now define the Floer homology groups. The situation is that we
have a compact, connected, oriented $3$-manifold $Y$ with an oriented
knot or link $K\subset Y$, a choice of simple, simply-connected Lie
group $G$ and a $\Phi$ in the fundamental Weyl chamber with
$\theta(\Phi)<1$. A Riemannian metric $g^{\nu}$ with an orbifold
singularity along $K$ is given. We continue to suppose that the non-integrality
condition (Definition~\ref{def:non-integral}) holds and that a
perturbation $\pert\in \Pert$ is chosen so as to satisfy
Hypothesis~\ref{hyp:stabilizers}. We also need to suppose that $\Phi$
satisfies the monotone condition, Definition~\ref{def:monotone}.

For a $2$-element set $\Lambda=\{\lambda,\lambda'\}$ we use
$\Z\Lambda$ to mean the infinite cyclic group obtained
from the rank-2 abelian group $\Z\lambda \oplus \Z\lambda'$ by
imposing the condition $\lambda = -\lambda'$. Thus a choice of element
of $\Lambda$ determines a generator for $\Z\Lambda$.
We define $C_{*}(Y,K,\Phi)$ to be the
free abelian group
\begin{equation}\label{eq:chain-complex}
            C_{*}(Y,K,\Phi) = \bigoplus_{\beta\in \Crit_{\pert}}
            \Z\Lambda(\beta),
\end{equation}
If $\gr_{z}(\alpha,\beta)=1$ and $[\breve A]$ denotes the $\R$-orbit of
some $[A]$ in $M_{z}(\alpha,\beta)$, then from
\eqref{eq:epsilon-sign} above we obtain an isomorphism
\[
                \epsilon[\breve A] : \Z\Lambda(\alpha) \to
                \Z\Lambda(\beta).
\]

Combining all of these, we define
\[
                \partial : C_{*}(Y,K,\Phi) \to  C_{*}(Y,K,\Phi)
\]
by
\begin{equation}\label{eq:boundary-map-def}
                \partial = \sum_{(\alpha,\beta,z)} \sum_{[\breve{A}]}
                 \epsilon[\breve{A}]
\end{equation}
where the first sum runs over all triples with
$\gr_{z}(\alpha,\beta)=1$.

That the above sum is finite depends on the monotonicity condition.
The point is that for any pair $(\alpha,\beta)$, there will be
infinitely many homotopy classes of paths $z$ with
$\gr_{z}(\alpha,\beta)=1$ (as long as $K$ is non-empty).
Thus the first sum in the definition of $\partial$ has an
infinite range. The monotone condition, however, ensures that only
finitely many of the $1$-dimensional moduli spaces
$M_{z}(\alpha,\beta)$ will be non-empty: this is the statement of
Corollary~\ref{cor:energy-finite-monotone}.

Based as usual on a gluing theorem and consideration of the compactification of moduli spaces
$\Mu_{z}(\alpha,\gamma)$ with $\gr_{z}(\alpha,\gamma)=2$, one shows
that $\partial\comp\partial =0$. It is important here that the moduli
spaces of broken trajectories $\Mubk_{z}(\alpha,\gamma)$ are compact
when $\gr_{z}(\alpha,\gamma)=2$, as follows from
Proposition~\ref{prop:broken-compactness}.

\begin{definition}\label{def:basic-I}
   When the non-integrality and transversality assumptions of
   Hypothesis~\ref{hyp:stabilizers} holds and $\Phi$ satisfies the
   monotone condition, we define
   $\I_{*}(Y,K,\Phi)$ to be the homology of the complex
   $(C_{*}(Y,K,\Phi),\partial)$.
\end{definition}

Since $\gr_{z}(\alpha,\beta)$ taken modulo $2$ is independent of the
path $z$, we can regard $\I_{*}(Y,K,\Phi)$ as having an affine grading by
$\Z/2$. For particular choices of $G$ and $\Phi$, the greatest common
divisor of $\gr_{z}(\beta,\beta)$, taken over all closed paths, may
be a proper multiple of $2$, in which case 
$\I_{*}(Y,K,\Phi)$ has an affine $\Z/(2d)$-grading for $d>1$. For
example, if $G=\SU(N)$ and $\Phi$ has just two distinct eigenvalues,
then the homology is graded by $\Z/(2N)$.

Rather than being left as a relative (i.e.~affine) grading,
the mod $2$ grading can be made canonical. For a critical point
$\alpha$, the grading of $\alpha$ mod $2$ can be defined as the mod
$2$ reduction of $\gr_{z}(\theta^{\varphi},\alpha)$, where
$\theta^{\varphi}$ is the reducible configuration constructed earlier
and $z$ is any homotopy class of paths. The result is independent of
the choices made.

\subsection{Cobordisms and invariance}
\label{subsec:cobordims-1}

The Floer group $\I_{*}(Y,K,\Phi)$ depends only on $(Y,K)$ as a smooth
oriented pair and on the choice of $\Phi$: it is independent of the
remaining choices made. These choices include the choice of Riemannian
metric and the perturbation $\pert$: changing either of these may
change the set of critical points that form the generators of the
complex. More subtly, the choice of cut-off function involved in the
construction of the base connection $B^{\varphi}$ may effect the
$2$-element set $\Lambda(\alpha)$ used in fixing signs.  As in Floer's
original approach, the independence of the Floer groups on these
choices can be seen as a consequence of a more general property,
namely the fact that a cobordism between pairs gives rise to a
homomorphism on Floer homology.

To say this more precisely, let $(Y_{0},K_{0})$ and $(Y_{1},K_{1})$
be two pairs. By a \emph{cobordism} between them we will mean a connected,
oriented manifold-with-boundary, $W$, containing  a properly embedded
oriented surface-with-boundary, $S$, together with an
orientation-preserving diffeomorphism of pairs
\[
                r : ( \bar{Y}_{0}, \bar{K_{0}}) \amalg
                (Y_{1},K_{1}) \to (\partial W, \partial S).
\]                
If $(W,S)$ and $(W'S')$ are two cobordisms between the same pairs, then an
isomorphism between them means a diffeomorphism between the underlying
manifolds  commuting with $r$. Isomorphism classes of cobordisms can
be composed in the obvious way, and in this manner we obtain a
category, whose objects are the pairs $(Y,K)$ and whose morphisms are
the isomorphism classes of cobordisms. If $(W_{1},S_{1})$ is a cobordism from
$(Y_{0},K_{0})$ to $(Y_{1},K_{1})$ and $(W_{2},S_{2})$
is a cobordism from $(Y_{1},K_{1})$ to $(Y_{2},K_{2})$,
we denote by
\begin{equation}\label{eq:composite-cobord}
            (W ,S)= ( W_{2} \comp W_{1}, S_{2} \comp S_{1})
\end{equation}
the composite cobordism from $(Y_{0},K_{0})$ to $(Y_{2},K_{2})$.

We adopt from \cite{KM-book} the appropriate definition of a
\emph{homology orientation} for a cobordism $W$ from $Y_{0}$ to $Y_{1}$:
a homology orientation $o_{W}$ is a choice of orientation for the line
    \[
        \Lambda^{\max}
        H^{1}(W;\R) \otimes \Lambda^{\max} I^{+}(W) \otimes
        \Lambda^{\max} H^{1}(Y_{1};\R)
    \]
where $I^{+}(W)$ is a maximal positive-definite subspace for the
non-degenerate quadratic pairing on the image of $H^{2}(W,\partial
W;\R)$ in $H^{2}(W;\R)$. (Note that the links $K_{i}$ and the
$2$-dimensional cobordism $S$ are not involved here, and are omitted
from our notation.) This definition can be made to look less
arbitrary by regarding this as an orientation for the determinant line
of the operator $-d^{*}\oplus d^{+}$ on the cylindrical-end manifold
obtained from $W$, acting on weighted Sobolev spaces with a consistent
choice of weights. There is a composition law for homology
orientations: if $W=W_{2}\comp W_{1}$ and homology orientations
$o_{W_{i}}$ are given, we can construct a homology orientation
$o_{W_{2}}\comp o_{W_{2}}$ for $W$. This is most easily seen from the
second description of what a homology orientation is. We thus have a
modified category in which the morphisms are 
cobordisms of pairs, $(W,S)$, equipped with homology orientations,
up to isomorphism.

Let $(W,S)$ be a cobordism from $(Y_{0}, K_{0})$ to $(Y_{1},
K_{1})$. Suppose that each $Y_{i}$ is equipped with a Riemannian
metric and that perturbations $\pert_{i}$ are chosen satisfying
Hypothesis~\ref{hyp:stabilizers}. We continue to suppose also that
$\Phi$ satisfies the monotone condition. Let base connections
$B_{i}^{\varphi_{i}}$ be chosen for each. In this case, we have Floer homology groups
$  \I_{*}(Y_{i},K_{i},\Phi)$, for $i=0,1$,
which depend a priori on the choices made. Let us temporarily
denote this collection of choices (of metric, perturbation and base
connection) by $\pertstuff_{i}$, and so write the
groups as
\[
                \I_{*}(Y_{i},K_{i},\Phi)_{\pertstuff_{i}}, \qquad i=0,1.
\]
The fact that cobordisms give rise to maps can be stated as follows.

\begin{proposition}\label{prop:functor}
    Suppose that $\Phi$ satisfies the monotone condition.
    For $i=0,1$, let $(Y_{i},K_{i})$ be pairs as above, and suppose
    that Hypothesis~\ref{hyp:stabilizers} holds for both.
    Let $\pertstuff_{i}$ be choices of Riemannian metric,
    connection $B^{\varphi}$ and perturbation as above. Let
    $(W,S)$ be a cobordism from $(Y_{0},K_{0})$ to $(Y_{1},K_{1})$,
    and let
    a homology orientation $o_{W}$ for the cobordism $W$ be given.
    Then $(W,S,o_{W})$ gives rise to a homomorphism
    \begin{equation}\label{eq:cobordism-map}
       \I_{*}(W,S,\Phi, o_{W}) :   \I_{*}(Y_{0},K_{0},\Phi)_{\pertstuff_{0}}
                 \to  \I_{*}(Y_{1},K_{1},\Phi)_{\pertstuff_{1}}
    \end{equation}
    which depends only on the isomorphism class of the cobordism with
    its homology orientation.  Furthermore, composition of cobordisms
    becomes composition of maps and the trivial product cobordism
    gives the identity map.
\end{proposition}

\begin{remark}
    The choice of homology orientation $o_{W}$ affects only the
    overall sign of the map $\I_{*}(W,S,\Phi,o_{W})$, and affects it
    non-trivially only if the dimension of $G$ is odd:
    \cf~Proposition~\ref{prop:orientable-sing}.
\end{remark}

In particular, by taking $W$ to be a cylinder and setting
\[
 (Y_{0},K_{0})=(Y_{1},K_{1})=(Y,K),
\]
 we see that the Floer group
$\I_{*}(Y,K)_{\pertstuff}$ is independent of the auxiliary choices
$\pertstuff$, up to canonical isomorphism. Usually, we omit mention of
 $o_{W}$ and $\pertstuff$ from our notation (just as we have already
 silently omitted $r$), and we simply
write
\[
       \I_{*}(W,S,\Phi) :   \I_{*}(Y_{0},K_{0},\Phi)
                 \to  \I_{*}(Y_{1},K_{1},\Phi)
\]
with the unstated understanding that the identifications $r$ (when
needed)
are implied and that $o_{W}$ is needed to fix the
overall sign of this map if the dimension of $G$ is odd.

The proof of Proposition~\ref{prop:functor} follows standard lines,
and can be modelled (for example) on the arguments from \cite{KM-book}.
We content ourselves here with some remarks about the construction of
the maps $\I_{*}(W,S,\Phi)$.

For $i=0,1$, let $\beta_{i}$ be a  critical point in $\bonf(Y_{i},
K_{i},\Phi)$. On the pair $(W,S)$, let us consider a $G$-bundle $P$
equipped with a section $\varphi$ of $O_{P}$ along $S$ and
corresponding connection $A$ with singularity along $S$, subject to
the constraint that the restrictions of $A$ to the two ends should
define singular connections belonging the gauge-equivalences classes
of $\beta_{0}$ and $\beta_{1}$.  There is an obvious notion of a
continuous family of such data, $(P_{t}, \varphi_{t}, A_{t})$
parametrized by any space $T$, and we can therefore consider the set of
deformation-classes of such data. We will refer to such an equivalence
class as a \emph{path from $\beta_{0}$ to $\beta_{1}$ along the
cobordism $(W,S)$}. In the case of a cylindrical cobordism, such a
path is the same as a homotopy class of paths from $\beta_{0}$ to
$\beta_{1}$ in $\bonf(Y,K,\Phi)$. If $S$ has any closed components,
then different paths along $(W,S)$ may also be distinguished by having
different monopole charges on the closed components. If $(W,S)$ is a
composite cobordism, as in \eqref{eq:composite-cobord}, and if $z_{1}$
and $z_{2}$
are paths along $(W_{1},S_{1})$ and $(W_{2},S_{2})$ from $\beta_{0}$
to $\beta_{1}$ and from $\beta_{1}$ to $\beta_{2}$ respectively, then
there is a well-defined composite path along $(W,S)$, obtained by
choosing any identification of the two bundles on $Y_{1}$ respecting
the sections $\varphi_{i}$ and the connections.

\begin{remark}
There is a small point
to take note of here. By assumption, the critical point $\beta_{1}$,
like all critical points, is irreducible and has stabilizer $Z(G)$.
When forming the composite path by identifying the two bundles along
$Y_{1}$, there is therefore a $Z(G)$'s worth of choice in how the
identification is made. Despite this choice, the composite path is
well-defined, because the automorphisms of the connection on $Y_{1}$
extend to the $4$-manifolds.
\end{remark}

Let $W^{+}$ be the  manifold obtained by attaching
cylindrical ends to the two boundary components of $W$, and let this
manifold be given a Riemannian metric $g^{\nu}_{W}$ which is a product
metric on each of the two cylindrical pieces. Let $S^{+}\subset W^{+}$
be obtained similarly from $S$.

Let critical points $\beta_{i}$ in $\bonf(Y_{i}, K_{i},\Phi)$ be given
for
$i=0,1$, and let $z$ be a path along $(W,S)$ from $\beta_{0}$ to
$\beta_{1}$. Let $(P_{W}, \varphi_{W}, A_{W})$ be a representative for
$z$, and extend this data to the cylindrical ends by pull-back.
Imitating the definition of $\conf_{z}(\alpha,\beta)$ from
\eqref{eq:conf-z-def}, we define a configuration space of singular
connections $\conf_{z}(W,S,\Phi; \beta_{0},\beta_{1})$ as the space of
all $A$ differing from $A_{W}$ by a term belonging to
$\check{L}^{2}_{m}$, and we write
$\bonf_{z}(W,S,\Phi;\beta_{0},\beta_{1})$ for the corresponding
quotient space.

Let $\pert_{0}$ and $\pert_{1}$ be the chosen holonomy perturbations on
$Y_{0}$ and $Y_{1}$ respectively. We perturb the $4$-dimensional
equations on $W^{+}$  by adding a term supported on the
cylindrical ends: this term will be a $t$-dependent holonomy
perturbation $\pi_{W}$ equal to $\pert_{i}$ on the two ends.
In more detail, in a collar $[0,1) \times Y_{0}$ of one of the
boundary components $Y_{0}\subset W$, the perturbed equations take the
form
\[
            F^{+}_{A} + \beta(t)\hat{U_{0}}(A) +
            \beta_{0}(t)\hat{V}_{0}(A) =0
\]
where, as in \eqref{eq:perturbed-4d-short}, $\hat{V}_{0}$ is the
perturbing term defined by $\pert_{0}\in \Pert$ and $\hat{U}_{0}$ is
defined by a choice of an auxiliary element of $\Pert$. The cut-off
function $\beta$ is supported in the interior of the interval, while
$\beta_{0}$ is equal to $1$ near $t=0$ and equal to $0$ near $t=1$.
(This choice of perturbation follows \cite[Section~24]{KM-book}.) 
We write
\[
            M_{z}(W,S,\Phi;\beta_{0},\beta_{1})
            \subset  \bonf_{z}(W,S,\Phi;\beta_{0},\beta_{1})
\]
for the moduli space of solutions to the perturbed equations on
$W^{+}$. For generic choice of auxiliary perturbation $\hat{U}_{i}$ on
the two collars, the moduli space is cut out transversely by the
equations and (under our standing assumptions of
Hypothesis~\ref{hyp:stabilizers}) is a smooth manifold. A choice of
homology orientation $o_{W}$ and an element of $\Lambda(\beta_{0})$
and $\Lambda(\beta_{1})$ determines an orientation of the moduli
space. As in the closed case, if $G$ is even-dimensional, then $o_{W}$
is not needed. The map \eqref{eq:cobordism-map} is defined in the
usual way by counting with sign the points of all zero-dimensional
moduli spaces $M_{z}(W,S,\Phi;\beta_{0},\beta_{1})$. As in the
definition of the boundary map $\partial$, the monotonicity condition
ensures that this is a  finite sum, because for fixed $\beta_{0}$ and
$\beta_{1}$, the dimension of the moduli space corresponding to a path
$z$ along $(W,S)$ is an affine-linear function of the
topological energy.

\subsection{Local coefficients}
\label{subsec:local-coeffs}

There is a standard way in which the construction of Floer homology
groups can be generalized, by introducing a local system of
coefficients, $\Gamma$, on the configuration space (in this case, the
configuration space $\bonf(Y,K,\Phi)$ of singular connections modulo
gauge transformations on the $3$-manifold). Thus we suppose that for
each point $\beta$ in the configuration space we have an abelian group
$\Gamma_{\beta}$ and for each homotopy class of paths $z$ from
$\alpha$ to $\beta$ and isomorphism $\Gamma_{z}$ from
$\Gamma_{\alpha}$ to $\Gamma_{\beta}$ satisfying the usual composition
law. If we make the same assumptions as before (the conditions of
Hypothesis~\ref{hyp:stabilizers} and the monotonicity condition,
Definition~\ref{def:monotone}), then we can modify the definition of
the chain group $C_{*}(Y,K,\Phi)$ by setting
\[
            C_{*}(Y,K,\Phi;\Gamma) = \bigoplus_{\beta\in \Crit_{\pert}}
            \Z\Lambda(\beta) \otimes \Gamma_{\beta}
\]
and taking the boundary map to be
\begin{equation}\label{eq:boundary-Gamma}
                     \partial = \sum_{(\alpha,\beta,z)} \sum_{[\breve{A}]}
                 \epsilon[\breve{A}] \otimes \Gamma_{z}.
\end{equation}
The homology of this complex, $\I_{*}(Y,K,\Phi;\Gamma)$ is the Floer
homology with coefficients $\Gamma$.

If we are given two pairs, $(Y_{0}, K_{0})$ and $(Y_{1}, K_{1})$ with
local systems $\Gamma^{0}$ and $\Gamma^{1}$, and if $(W,S)$ is a
cobordism between the pairs, then we have a natural notion of
morphism, $\Delta$, 
of local systems along $(W,S)$: such a $\Delta$ assigns to each path
$z$ from $\beta_{0}$ to $\beta_{1}$
along $(W,S)$ (in the sense of the previous subsection) a
homomorphism
\begin{equation}\label{eq:Delta-z}
            \Delta_{z} : \Gamma_{\beta_{0}}^{0} \to
            \Gamma_{\beta_{1}}^{1}
\end{equation}
respecting the composition maps with paths in
$\bonf(Y_{i},K_{i},\Phi)$ on the two sides. (See \cite{KM-book}, for
example.) Using such a morphism $\Delta$, we can adapt the definition
of the map $\I_{*}(W,S,\Phi)$ in an obvious way to obtain a homomorphism
\[
                \I_{*}(W,S,\Phi;\Delta)
                : \I_{*}(Y_{0}, K_{0}, \Phi;\Gamma^{0})
                \to \I_{*}(Y_{1}, K_{1}, \Phi;\Gamma^{1}).
\]

To give an example, we begin with a standard local system
$\Gamma^{S^{1}}$ on the
circle $S^{1}$, regarded as $\R/\Z$, defined as follows. We write
$R$ for the ring of finite Laurent series with integer coefficients in
a variable $t$. This is the group ring $\Z[\Z]$, and we can regard it
as lying inside the group ring $\Z[\R]$: the ring of formal finite
series
\[
                    \sum_{x\in \R} a_{x} t^{x}.
\]
For each $\lambda$ in $\R$, we have an $R$-submodule $t^{\lambda}R
\subset \Z[\R]$ generated by the element $t^{\lambda}$: this is the
$R$-module of all finite series of the form
\[
                    \sum_{x\in \lambda + \Z} a_{x} t^{x}.
\]
As $\lambda$ varies in $\R/\Z$ these form a local system of
$R$-modules, $\Gamma^{S^{1}}$ over $S^{1}$: the map
$\Gamma^{S^{1}}_{z}$ corresponding to a path $z$ is given by
multiplication by $t^{\lambda_{1}-\lambda_{0}}$ if $z$ lifts to a path
in $\R$ from $\lambda_{0}$ to $\lambda_{1}$. If we are now given a circle-valued function
\[
            \mu : \bonf(Y,K,\Phi)    \to S^{1} =\R/\Z                
\]
then we can pull back the standard local system $\Gamma^{S^{1}}$ to
obtain a local system
\[
                \Gamma^{\mu} = \mu^{*}(\Gamma^{S^{1}})
\]
on $\bonf(Y,K,\Phi)$.

This construction can be applied using a class
of naturally-occurring circle-valued functions on the configuration
space of singular connections.
These functions can be defined, roughly speaking, by taking the
holonomy of a connection $B$ along a longitudinal curve close to a
component of the
link $K$ and applying a character of $G_{\Phi}$. To say this more
precisely, 
we choose a framing of the link $K\subset Y$, so has to have
well-defined coordinates on the tubular neighborhood, up to isotopy,
identifying the neighborhood with $D^{2}\times K$. Suppose first that
$K$ has just one component, and for each
sufficiently small $\epsilon
> 0$, let $T_{\epsilon}$ be the torus obtained as the product of the
circle of radius $\epsilon$ in $D^{2}$ with knot $K$. Use the
coordinates to identify $T_{\epsilon}$ with $S^{1}\times K$. If $B$ is
a connection in $\conf(Y,K,\Phi)$, then by restricting to
$T_{\epsilon}$ we obtain in this way a sequence of $G$-connections on
$S^{1}\times K$; and the definition of the space $\check{L}^{2}_{m}$
in which we work guarantees that these have a  well-defined limit, up
to gauge transformation, which is a flat connection $B_{0}$ on $S^{1}\times
K$. The holonomy of $B_{0}$ along a curve belonging to the $S^{1}$
factor is $\exp(\varphi)$, and the holonomy along the longitudinal
curve belonging to the $K$ factor lies in the commutant.
Choose a character
\[
                    s : G_{\Phi} \to U(1)
\]
and let
\[
                    \w : \g_{\Phi} \to \R
\]
be the corresponding weight, so that $s(\exp(x)) = e^{2\pi i \w(x)}$.
We can apply $s$ to the holonomy of $B_{0}$ along the
longitudinal curve to obtain a well-defined element of $U(1)$,
depending only on the gauge-equivalence class of $B$. Thus we obtain
from $s$ a function
\begin{equation}\label{eq:construction-of-mu}
                \mu_{s} : \bonf(Y,K,\Phi) \to U(1) = \R/\Z
\end{equation}
by applying $s$ to the holonomy along the longitudinal curve. In this
way we obtain a local system $\Gamma^{\mu_{s}}$ by pull-back.

The choice of framing of $K$ is essentially immaterial. The set of
framings is an affine copy of $\Z$;
and if we change the chosen framing of $K$ by $1$, then $\mu_{s}$ is changed by
the addition of the constant $\w(\Phi)$ mod $\Z$. The
corresponding local systems are canonically isomorphic, via
multiplication by $t^{\w(\Phi)}$.

If $K$ has more than one
component, we can apply this construction to each one, perhaps using
different characters $s$, and form the product. Alternatively, one
could define a local system over a ring of Laurent series in a number
of variables $t_{i}$, one for each component of $K$.

In the above construction, the reason for taking such a
specifically-defined function $\mu_{s}$, rather than a general
circle-valued function belonging to the same homotopy class, is that
the naturality inherent in the construction leads to a Floer homology
group that is a topological invariant of the pair, rather than a group
that is an invariant only up to isomorphism. The point is that if we
have a cobordism of pairs, $(W,S)$, with a chosen framing of a tubular
neighborhood of $S$ (or at least of the components of $S$ having
non-empty boundary), then we obtain a natural morphism $\Delta$
between the corresponding local systems associated to the framed knots
at the two ends. The map $\Delta_{z}$ corresponding to a path $z$
along $(W,S)$ can be defined as follows. Fix data $(P,\varphi,A)$ on
$W$ corresponding to $z$. For each small positive
$\epsilon$, we have a copy of $S^{1}\times S$ in $W$, as the boundary
of the $\epsilon$-neighborhood of $S$ in its framed tubular
neighborhood $D^{2}\times S$, and we therefore obtain connections
$A_{\epsilon}$ on $S^{1}\times S$. The limit of these connections is a
connection $A_{0}$ on $S^{1}\times S$ whose curvature $2$-form has the
$S^{1}$ direction in its kernel. Thus $A_{0}$ gives a $G_{\Phi}$-connection
on $S^{1}\times S$; and applying the character $s$ we obtain a $U(1)$-connection
$s(A_{0})$  on
$S^{1}\times S$. For any $p$ in $S^{1}$, we have a parallel copy of
$S$ as $\{p\}\times S$, and the map $\Delta_{z}$
can then be defined as multiplication by $t^{\nu}$, where
\begin{equation}\label{eq:local-coeff-cobordism-int}
            \nu = \frac{i}{2\pi} \int_{\{p\}\times S} F_{s(A_{0})}.
\end{equation}
Because the curvature $2$-form of $s(A_{0})$ annihilates the circle
directions, we  see that we could have taken any section of
$S^{1}\times S$ instead of the constant section $\{p\}\times S$, and
the above integral would be unchanged. So in the end, the map
$\Delta_{z}$ is independent of the choice of framing of $S$. 

Local systems can also be made use of to define Floer groups in the
case that $\Phi$ does not satisfy the monotone condition. When $\Phi$
is not monotone, the sum \eqref{eq:boundary-Gamma} which defines the boundary
operator may have infinitely many non-zero terms; but the sum can
still be made sense of if each $\Gamma_{\alpha}$ is a topological
group and the local system is  such that the series converges. A
typical instance of such a construction replaces the ring $R$ of
finite Laurent series which we used above by the ring of Laurent
series that are infinite in one direction.

\subsection{Non-simple groups}
\label{subsec:unitary-3d}

We have been considering instanton Floer homology in the case that
$G$ is a simple group. When discussing instanton moduli spaces, we saw in
section~\ref{subsec:unitary-4d} how the definitions are readily
adapted to the case which of a non-simple group such as the unitary
group. We now carry this over to the Floer homology setting.
We again suppose that $G$ has a simply-connected commutator subgroup.
We write $Z(G)$ for its center and $\bar{Z}(G)$ for $G/[G,G]$. Unlike
the case in which $G$ itself is simply-connected, it is no longer the
case that a $G$-bundle $P\to Y$ must be trivial: its isomorphism type
is determined by the $\bar{Z}(G)$-bundle $\gdet(P)$, or
equivalently by the characteristic class $c=c(P)$ in $H^{2}(Y;L(G))$ of
\eqref{eq:c-char-class}.

To preserve the functoriality of the Floer homology groups, we need to
adopt the alternative viewpoint for the configuration space and gauge
group which we mentioned briefly in subsection~\ref{subsec:unitary-4d}.
We fix $\bar{Z}(G)$-bundle $\delta\to Y$ with an isomorphism $q :
\gdet(P)\to \delta$, and we fix a connection $\Theta$ in $\delta$. As
before, we let $\Theta^{\varphi}$ denote the corresponding singular
connection in $\delta$ (equation \eqref{eq:Theta-varphi}),
and we construct a space $\conf(Y,K,\Phi)_{\delta}$ 
of singular connections, $B$,  with the constraint that
$\gdet(B) = q^{*}(\Theta^{\varphi})$.
The gauge group $\G(Y,K,\Phi)$
consists of gauge transformations $g$ of class $\check{L}^{2}_{m+1}$
with $\gdet(g)=1$ and we have a quotient space
$\bonf(Y,K,\Phi)_{\delta}$.

The construction of the Floer groups then proceeds as before, with
straightforward modifications of the same type as we dealt with in
section~\ref{subsec:unitary-4d}. We deal with some of these
modifications in the next few paragraphs.

\paragraph{The Chern-Simons functional.}

The appropriate Chern-Simons functional on $\conf(Y,K,\Phi)_{\delta}$
in the present setting is the one which ignores the central component
of the connection: it can be defined by the same formula
\eqref{eq:CS-formula} as before, if we understand that the inner
products in \eqref{eq:CS-formula} are defined using the semi-definite
Killing form. Critical points of the unperturbed Chern-Simons
functional on $\conf(Y,K,\Phi)_{\delta}$ are singular connections $B$
such that the induced connection $\bar{B}$ with structure group $G/Z(G)$ in the
adjoint bundle is flat. The formal gradient flow lines of this
functional correspond to connections $A$ in temporal gauge on the
cylinder with the property that $\bar{A}$ is anti-self-dual.

\paragraph{The non-integral condition.}
The most important change involves
the non-integrality condition, Definition~\ref{def:non-integral},
which we used to rule out reducible critical points and which formed
part of our standing Hypothesis~\ref{hyp:stabilizers}. In the case
that $G$ is not simple, the corresponding condition can be read off
from the $4$-dimensional version,
Proposition~\ref{prop:no-reducibles-2}:

\begin{definition}\label{def:non-integral-non-simple}
     Let the components of $K$ again be $K_{1}$,\dots,$K_{r}$.
     For non-simple groups $G$,
    we will say that the bundle $P$ on
    $(Y,K,\Phi)$ satisfies the \emph{non-integral
    condition} if, in the notation of section~\ref{subsec:unitary-4d}, the expression
\begin{equation*}
            w_{\alpha}(c(P)) +
            \sum_{j=1}^{r}(\bar{w}_{\alpha}\comp\sigma_{j})(\bar\Phi)
            \mathrm{P.D.}[K_{j}]
\end{equation*}
    is a
    non-integral cohomology class for every choice of fundamental weight $w_{\alpha}$
    and Weyl group elements
    $\sigma_{1},\dots,\sigma_{r}$.
\end{definition}

The simplest example in which this non-integrality holds is the case
corresponding to Corollary~\ref{cor:example-coprime}, in which $G$ is
the unitary group $U(N)$ and all the components of $K_{i}$ are
null-homologous: in this case, the non-integral condition is
equivalent to saying that the pairing of $c_{1}(P)$ with some integral
homology class in $Y$ is coprime to  $N$.

As previously, we need to suppose that this non-integrality condition
holds and that further, as in Hypothesis~\ref{hyp:stabilizers}, the
stabilizer in $\G(Y,K,\Phi)_{\delta}$ of every critical point is exactly
$Z(G)\cap [G,G]$, rather than some larger finite group
(a condition which is automatic in the non-integral
case if $G=U(N)$).  Under these conditions, and when $\Phi$ satisfies
the monotone condition \eqref{eq:monotone-non-simple},
we will arrive at a Floer homology group
\[
                \I_{*}(Y,K,\Phi)_{\delta}
\]
depending on the choice of bundle $\bar{Z}(G)$-bundle $\delta$.

\paragraph{Holonomy perturbations.}

The definition of holonomy perturbations does not need any changes in the case
of more general $G$. The basic ingredient is still a choice of
function
\[
            h : G^{r} \to \R
\]
invariant under the diagonal action of $G$, acting by the adjoint
representation on each factor. Holonomy perturbations still separate
points in the quotient space $\bonf(Y,K,\Phi)_{\delta}$. Note that the
choice of connection $\Theta$ is involved in the construction, because
we are taking the holonomy of a $G$-connection $B$ in the bundle $P$
which satisfies $\gdet(B)=\Theta^{\varphi}$. If we chose $h$ so that it was
pulled back from $(G/Z(G))^{r}$, then the choice of $\Theta$ would
again become irrelevant; but functions $h$ of this sort are not a
large enough class, as they do not allow our holonomy perturbations to
separate points and tangent vectors in $\bonf(Y,K,\Phi)_{\delta}$.

\paragraph{Orientations.}

In the case of a simple group $G$, we defined a $2$-element set
$\Lambda(\alpha,\beta)$ for a pair of configurations $\alpha$ and
$\beta$; and we then defined $\Lambda(\alpha)$ as being
$\Lambda(\theta^{\varphi}, \alpha)$, where $\theta^{\varphi}$ was a
specially chosen connection. The important features of our choice of
$\theta^{\varphi}$ were first that $\theta^{\varphi}$ was reducible
and second that, although the construction depended on details such as
a choice of cut-off function, any two choices differed by a small
isotopy, so that an essentially unique path connects any two choices.

When $P$ is not simple and $\gdet(P)$ is non-trivial, we do not
have a distinguished gauge-equivalence class of \emph{trivial} connections in
$P$ from which to construct $\theta^{\varphi}$, but we can instead
proceed as we did in section~\ref{subsec:unitary-4d}. We fix again a
homomorphism
$\fe : \bar{Z}(G) \to T $ which is right-inverse to $\gdet$ (see
\eqref{eq:gdet-inverse}). As in the $4$-dimensional case, we obtain
a $G$-connection $\fe(\Theta)$ on a
bundle isomorphic to $P$, with $\gdet(\fe(\Theta))=\Theta$. After adding
the singular term
along $K$, we obtain a distinguished
gauge-equivalence class of connections, $\theta^{\varphi}$, in
$\bonf(Y,K,\Phi)_{\delta}$. Once $\fe$ is fixed, this gauge-equivalence class
depends only on the details of how the singular term is constructed,
through the choice of cut-off function for example. This puts us  in
a  position to define $\Lambda(\alpha)$ as we did before.

\paragraph{Cobordisms.}

Let $(W,S)$ now be a cobordism of oriented pairs, and write its two
boundary components as $(Y_{i}, K_{i})$ for $i=0,1$, so that
\[
                \partial(W,S) = (\bar{Y}_{0},\bar{K}_{0}) \amalg
                 (\bar{Y}_{1},\bar{K}_{1}). 
\]
(In the slightly more categorical language that we used earlier in
section~\ref{subsec:cobordims-1}, we are supposing here that the
identification map $r$ is the identity.)
Let $\delta_{W}$ be a $\bar{Z}(G)$-bundle on $W$, and let write
\[
            \delta_{i} = \delta_{W}|_{Y_{i}}, \qquad{i=0,1}.
\]
Fix $G$-bundles $P_{0}$ and $P_{1}$ on $Y_{0}$ and $Y_{1}$
with isomorphisms
$q_{i}:\gdet(P_{i})\to \delta_{i}$. Let $\Theta_{0}$ and $\Theta_{1}$
be chosen connections in $\delta_{0}$ and $\delta_{1}$.

We wish to show how the data $(W,S,\delta_{W})$ (together with a
homology orientation of $W$) gives rise to a homomorphism from
$\I_{*}(Y_{0}, K_{0},\Phi)_{\delta_{0}}$ to $\I_{*}(Y_{1},
K_{1},\Phi)_{\delta_{1}}$.
The first step is to extend our previous
notion of a ``path along $(W,S)$'' between critical points $\beta_{0}$ and
$\beta_{1}$ belonging the configuration spaces
$\bonf(Y_{i},K_{i},\Phi)_{\delta_{i}}$ for $i=0,1$. To do this, we let
$\beta_{i}$ be represented by singular connections $B_{i}$ on $P_{i}
\to Y_{i}$ and we define a path $z$ to be defined by data consisting
of:
\begin{itemize}
\item a bundle $P\to W$;
\item an isomorphism $q_{W} : \gdet(P)\to \delta_{W}$;
\item a reduction of structure group defined by a
section $\varphi$ of $O_{P}$ along $S$; and
\item an isomorphism $R$ from $(P_{0}\amalg P_{1})$ to
$P|_{\partial
W}$,
respecting the reduction of structure group along $K_{i}$ and such
that $\gdet(R)$ fits into the obvious
commutative diagram involving  the other maps
on the $\bar{Z}(G)$-bundles -- a condition which appears as
\[q_{W}\comp \gdet(R) = (q_{0}\amalg q_{1}).\]
\end{itemize}

For any path $z$ in this sense, we can construct a moduli space,
generalizing our earlier $M_{z}(W,S,\Phi; \beta_{0},\beta_{1})$. To do
this, we use the data $P$, $P_{i}$, $B_{i}$ and $R$ to construct a
bundle $P^{+}$ on the cylindrical-end manifold $W^{+}$, together with
a connection on the two cylindrical ends (obtained by pulling back the
$B_{i}$). We extend this connection arbitrarily to a connection
$A_{W}$ on the whole of
$W^{+}$, with a singularity along $S^{+}$, and we write
$\Theta^{\varphi}_{W}$ for
$\gdet(A_{W})$. We can then define a configuration space
$\conf_{z}(W,S,\Phi;\beta_{0},\beta_{1})_{\delta_{W}}$ and quotient
space $\conf_{z}(W,S,\Phi;\beta_{0},\beta_{1})_{\delta_{W}}$ using
singular connections $A$ satisfying $\gdet(A)=\Theta^{\varphi}_{W}$ and with
$A-A_{W}$ of class $\check{L}^{2}_{m}$. Introducing perturbations as
before, we arrive at a moduli space
$M_{z}(W,S,\Phi;\beta_{0},\beta_{1})_{\delta_{W}}$. The task of
orienting this moduli space is the same as the case of a simple group
$G$, with slight modifications drawn from
section~\ref{subsec:unitary-4d}, so that when homology orientation of
$W$ is given together with elements of $\Lambda(\beta_{0})$ and
$\Lambda(\beta_{1})$, the moduli space is canonically oriented.

We summarize the situation with a proposition, generalizing
Proposition~\ref{prop:functor} to the case of non-simple groups:

\begin{proposition}\label{prop:functor-non-simple}
    Suppose that $\Phi$ satisfies the monotone condition
    \eqref{eq:monotone-non-simple}.
    Let $(W,S)$ be a cobordism with boundary the two pairs $(Y_{i},
    K_{i})$ as above, let $\delta_{W}$ be a  $\bar{Z}(G)$-bundle and
    let $\delta_{i}$ be its restriction to $Y_{i}$.
    Suppose that the non-integrality
    condition Definition~\ref{def:non-integral-non-simple} holds at
    both ends and 
    Hypothesis~\ref{hyp:stabilizers} holds, so that the Floer groups
    $\I_{*}(Y_{i},K_{i},\Phi)_{\delta_{i}}$ are defined.
    Then, after choosing a homology orientation $o_{W}$, there is a
    well-defined homomorphism
     \begin{equation}\label{eq:I-W-with-delta}
                I(W,S,\Phi)_{\delta_{W}} :
                   \I_{*}(Y_{0},K_{0},\Phi)_{\delta_{0}}
                   \to \I_{*}(Y_{1}, K_{1},\Phi)_{\delta_{1}}.
    \end{equation}
    These homomorphisms satisfy the natural composition law when
    a cobordism $(W,S)$ is decomposed as the union of two pieces and
    $\delta_{W}$ is restricted to each piece.
\end{proposition}

There is an important point about the naturality of this construction
that is worth spelling out in more detail. Rather than taking
$\delta_{i}$ to be the restriction of $\delta_{W}$ to the boundary
component $Y_{i}$, we could have taken the $\bar{Z}(G)$-bundles
$\delta_{0}$ and $\delta_{1}$ to have been given in advance, in which
case it is more natural to regard the necessary data on the cobordism
$W$ as consisting of
\begin{itemize}
    \item a $\bar{Z}(G)$-bundle $\delta_{W}\to W$; and
    \item a pair of isomorphisms $\tilde{r} = (\tilde{r}_{0},
    \tilde{r}_{1})$ from $\delta_{W}|_{Y_{i}}$ to $\delta_{i}$,
    $i=0,1$.
\end{itemize}
In this setting, the induced map between the two instanton homology
groups $\I_{*}(Y_{i},K_{i},\Phi)_{\delta_{i}}$ does depend on the
choice of isomorphisms $\tilde{r}_{i}$. The following corollary of
Proposition~\ref{prop:functor-non-simple} makes essentially the same
point:

\begin{corollary}
    Given $(Y,K)$ and a $\bar{Z}(G)$-bundle $\delta$ satisfying as
    usual the conditions of Hypothesis~\ref{hyp:stabilizers}, the
    group of components of the
    bundle automorphisms of $\delta\to Y$ acts on the the instanton
    homology group $\I_{*}(Y,K,\Phi)_{\delta}$.
\end{corollary}

\begin{proof}
    Given an automorphism $g$ of $\delta$, consider the cobordism
    $(W,S)$ that is the cylinder  $[0,1]\times(Y,K)$ and the bundle
    $\delta_{W}$ which is the pull-back. We can
identify $\delta_{W}$ with $\delta$ by using the identity map at the
boundary component $\{0\}\times Y$ and the map $g$ at the other
boundary component $\{1\}\times Y$. From the data $(W,S)$ and
$\delta_{W}$ with these identifications, $(r_{0},r_{1})=(1,g)$,
we obtain a homomorphism from $\I_{*}(Y,K,\Phi)_{\delta}$ to itself.
\end{proof}

Some of the automorphisms of $\delta$ act trivially on the Floer
homology:

\begin{proposition}
    Suppose that $G$ is constructed from $G_{1}=[G,G]$ as in
    \eqref{eq:G-from-G1} and that Condition~\ref{cond:G-conditions}
    holds. If $G_{1}$ is $\SU(N)$, then suppose additionally that $G$
    is $U(N)$ and that $\fe$ is standard. (See
    Proposition~\ref{prop:epsilon-signs}.)

    Then
    an automorphism $g : \delta\to\delta$ of the $\bar{Z}(G)$-bundle
    $\delta\to Y$ acts trivially on $\I_{*}(Y,K,\Phi)_{\delta}$ if $g$ has
    the form $\gdet(f)$ for some $Z(G)$-valued automorphism $f : P \to
    P$ of the corresponding bundle $P \to $Y.
\end{proposition}

\begin{proof}
    Take $W$ to be the cylinder $[0,1]\times Y$, and
    let $\delta_{W}$ be as in the proof of the previous corollary.
    If $g=\gdet(f)$, then we can describe $\delta_{W}$ as
    $\delta_{1}\otimes \gdet(\epsilon)$, where $\delta_{1}$ is the
    pull-back bundle $[0,1]\times\delta$ and $\epsilon\to W$ is a
    $Z(G)$ bundle equipped with a trivialization at each boundary
    component of $W$. That is, we take $\epsilon$ to be $W\times
    Z(G)$, with the trivialization $1$ at $\{0\}\times Y$ and $f$ at
    $\{1\}\times Y$.

    We are therefore left to compare two maps from
    $\I_{*}(Y,K,\Phi)_{\delta}$ to itself: the first is the
    identity map, arising from the product cobordism $W$ with
    $\delta_{1}$; and the second is the map obtained from $W$ with
    $\delta_{1}\otimes \gdet(\epsilon)$. This is essentially the same
    situation as the construction of the map $\mu_{\epsilon}$ in
    \eqref{eq:mu-epsilon}. In particular, the moduli spaces that are
    involved in defining the two maps are identical, and the only
    question is whether the zero-dimensional moduli spaces arise with
    the same sign.
   Proposition~\ref{prop:epsilon-signs} tells us that
    $\mu_{\epsilon}$ preserves orientation in all cases except
    $\SU(N)$ and $E_{6}$. For the $\SU(N)$ case, $\mu_{\epsilon}$
    still preserves orientation because the cobordism $W$ has even intersection form on its
    relative homology. For the case of $E_{6}$, an examination of the
    proof of Proposition~\ref{prop:epsilon-signs} shows that
    orientation depends on a term $\weyl(\bar{e})\cupprod
    \weyl(\bar{e})$; so again, the even intersection form ensures that
    $\mu_{\epsilon}$ is orientation-preserving.
\end{proof}

The bundle automorphisms of $\delta$ are the maps $Y\to \bar{Z}(G)$
and the group of components is $H^{1}(Y;\pi_{1}(\bar{Z}(G)))$. The
above Proposition tells us that the image of
$H^{1}(Y;\pi_{1}({Z}(G))$ acts trivially. Under the hypotheses of
Condition~\ref{cond:G-conditions}, we have a short exact sequence
\[
                   \pi_{1}( Z(G)) \to \pi_{1}(\bar{Z}(G)) \to Z(G_{1})
\]
in which the first two groups are free abelian and the first map is
multiplication by $p$. From the corresponding long exact sequence in
cohomology, we learn that the largest group that may act effectively on
$\I_{*}(Y,K,\Phi)_{\delta}$ via this construction is isomorphic to the subgroup
\[
           \cH \subset H^{1}(Y; Z(G_{1}))
\]
consisting of the elements with integer lifts:
\[
            \cH = \im \Bigl (H^{1}(Y; \pi_{1}(\bar{Z}(G))) \to H^{1}(Y;
            Z(G_{1}))\Bigr).
\]

As explained in section~\ref{subsec:unitary-4d} (where we treated the
$4$-dimensional case), we can also regard the automorphisms of
$\delta$ as defining, rather directly, automorphisms of the
configuration space $\bonf(Y,K,\Phi)_{\delta}$. Were it the case that
the holonomy perturbations could be chosen to be invariant under the
action of $\cH$ while still achieving the necessary transversality,
then we would have a more direct way of understanding the action of
$\cH$ on the instanton homology: the action  on the set of
critical points would give rise to an action of $\cH$ on the chain complex
$C_{*}(Y,K,\Phi)$ by chain maps. Although perturbations cannot always
be chosen so as to realize the action in this way, the following
situation does arise in some cases:

\begin{proposition}\label{prop:invariant-pert}
    Suppose that a subgroup $\cH'\subset\cH$ acts freely on the set of
    critical points for the unperturbed Chern-Simons functional $\CS$
    in $\bonf(Y,K,\Phi)_{\delta}$. Then a holonomy perturbation
    $\pert$ can be found, as in
    Proposition~\ref{prop:3d-transversality-summary}, that is
    invariant under $\cH'$, such that the critical point set
    $\Crit_{\pert}$ is non-degenerate, the action of $\cH'$ on
    $\Crit_{\pert}$ is still free, and the moduli spaces
    $M_{z}(\alpha,\beta)$ are all regular.
\end{proposition}

\begin{proof}
    A cylinder function arising from a collection of loops and a map
    $h:G^{r}\to\R$ will be invariant under the action of $\cH'$ on
    $\bonf(Y,K,\Phi)_{\delta}$ provided that $h$ is invariant under an
    associated action of $\cH'$ on $G^{r}$ (given by multiplications
    by central elements). This
    gives us a means to construct invariant perturbations, and
    the statements about the critical point set are straightforward.
    For the moduli spaces $M_{z}(\alpha,\beta)$, we note that, by
    unique continuation, once the action on $\Crit_{\pert}$ is known
    to be free, it must also be free on the subset of
    $\bonf(Y,K,\Phi)$ consisting of all points lying on gradient
    trajectories between critical points. Once the action is known to
    be free here, the transversality arguments go through without
    change.
\end{proof}

\section{Classical knots and variants}
\label{sec:classical}

\subsection{Summing with a 3-torus}

Take $G$ to be the group $U(N)$, let $Y$ be any closed, oriented
$3$-manifold and $K\subset Y$ an oriented knot or link. Let $y_{0}$ be
a base-point in $Y\setminus K$, and let an oriented frame in
$T_{y_{0}}Y$ be chosen. Using the base-point and frame, we can form
the connected sum $Y\# T^{3}$ in a manner that makes the result unique
to within a canonical isotopy class of diffeomorphisms. The knot or
link $K$ now becomes a knot or link in the connected sum. Any
topological invariant that we define for the pair $(Y\# T^{3}, K)$
becomes an invariant of the original pair $(Y,K)$ together with its
framed basepoint.

Regard the
$3$-torus as a product, $S^{1}\times T^{2}$, and let
$\delta_{1} \to T^{3}$ be a $U(1)$ bundle with $c_{1}(\delta_{1})$
Poincar\'e dual to $S^{1}\times \{\mathrm{point}\}$. Extend
$\delta_{1}$ trivially to the connected sum $Y\# T^{3}$, and call the
resulting $U(1)$-bundle $\delta$.  If $P$ is a $U(N)$ bundle on $Y\#
T^{3}$ whose determinant is $\delta$, then $P$ satisfies the
non-integral condition, Definition~\ref{def:non-integral-non-simple},
for any choice of $\Phi$ in the positive Weyl chamber.
(See the remarks immediately following that definition.) The remaining
condition in Hypothesis~\ref{hyp:stabilizers} (that the stabilizer of
each critical point is just $Z(\G)$) is automatically satisfied in the
case of $U(N)$ once the non-integrality condition holds. There is
therefore a well-defined Floer homology group (in the notation of
section~\ref{subsec:Floer-basic}),
\[
        \I_{*}(Y\# T^{3}, K, \Phi)_{\delta}
\]
for any choice $\Phi$ in the positive Weyl chamber which satisfies the
monotone condition. 

Let us consider a particular choice of $\Phi$ in this context, namely
\[
                \Phi = \mathrm{diag}(i/2, 0, \dots, 0).
\]
This $\Phi$ satisfies the monotone condition, and its orbit in
$\g=\u(N)$ is
a copy of $\CP^{N-1}$. The group element $\exp(2\pi  \Phi)$ has order
$2$: it is
\begin{equation}\label{eq:reflection}
           \mathrm{diag}(-1, 1, \dots, 1).         
\end{equation}
We introduce a notation for the corresponding instanton homology
group:

\begin{definition}\label{def:basic-with-T3-def}
    We write $\TI{N}{Y,K}$ for the instanton homology group
    \[\I_{*}(Y\# T^{3},K,\Phi)_{\delta}\] in the case that $G=U(N)$, with
    $\delta$ and $\Phi$ as above. In the case that $Y$ is $S^{3}$, we
    simply write $\TI{N}{K}$; and in the case that $N=2$, we write
    $\tI_{*}(Y,K)$ or $\tI_{*}(K)$. The group $\TI{N}{Y,K}$ has an
    affine grading by $\Z/(2N)$ (see the remark following
    Definition~\ref{def:basic-I}).
\end{definition}

\begin{remark}
If $\Phi$ is changed by the addition of a central element of $\u(N)$,
then the resulting instanton homology is essentially unchanged. Thus,
we could equally well have taken
$\Phi$ to be the element
\begin{equation}\label{eq:Phi-primed}
  \Phi' =  i \mathrm{diag} (1/2,0,\dots,0) - (i/(2N))\mathrm{diag}(1,1,\dots,1)
\end{equation}
in $\su(N)$,
so that
\begin{equation}\label{eq:reflection-primed}
        \exp(2\pi \Phi') = e^{-\pi i /N}\mathrm{diag}(-1,1,\dots,1)
        \in \SU(N).
\end{equation}
\end{remark}

To examine what comes of this definition, let us begin by looking at
$\TI{N}{\emptyset}$, i.e.~the case of the the empty link in $S^{3}$.
In other words, we are looking at $\I_{*}(T^{3})_{\delta_{1}}$.
Let $a$, $b$ and $c$ be standard generators for the fundamental group
of $T^{3}$, with $a$ and $b$ generating the fundamental group of the
$T^{2}$ factor in $T^{3} = T^{2}\times S^{1}$. Let $p\in T^{2}$ be a
point not lying on the $a$ or $b$ curves, and let $D$ be a small disk
around $p$. We can take the line bundle $\delta_{1}$ to be pulled back
from $T^{2}$, with a trivialization on the complement of $D\times
S^{1}$.
Let
$\Theta_{1}$ be a connection in the line bundle $\delta_{1}\to T^{3}$
that is also pulled back from $T^{2}$. We can
take the $\Theta_{1}$ to respect the trivialization of $\delta_{1}$ on
the complement $D\times S^{1}$, so that the
curvature of $\Theta_{1}$ is a  $2$-form supported in that
neighborhood. If $A \in \bonf(T^{3})_{\delta_{1}}$ is a critical point
for the Chern-Simons functional, then the restriction of $A$ to
$T^{2}\setminus D$ is a flat $\SU(N)$ connection whose holonomy around
$\partial{D}$ is the central element $e^{2\pi i/N}$ in
$\SU(N)$. In this way, the critical points correspond to conjugacy
classes of triples
$\{h(a),h(b),h(c)\}$ in $\SU(N)$ (the holonomies around the three
generators) satisfying
\[
\begin{aligned}{}
[h(a),h(b)] &= e^{2\pi i/N} \\
[h(a),h(c)] &= 1\\
[h(b),h(c)] &= 1.
\end{aligned}
\]
There are $N$ different
solutions to these conditions (see also \cite{K-higherrank}): the
element $h(c)$ can be any of the $N$ elements of the center of $\SU(N)$,
and up to
the action of $\SU(N)$, we must have
\begin{equation}\label{eq:T3-solutions}
\begin{aligned}
h(a) &= \epsilon \begin{bmatrix}
                1 & 0 & 0 & \cdots& 0\\
                0 & \zeta  & 0 & \cdots&0\\
                0 & 0 & \zeta^{2} & \cdots &0\\
                  &   &        &   \ddots&0 \\
               0   &   0&   0 &  0 & \zeta^{N-1}
            \end{bmatrix}
            \\
            h(b) &= \epsilon \begin{bmatrix}
                0 & 0 & 0 & \cdots&   1\\
                1 & 0  & 0 & \cdots&0\\
                0 & 1 & 0 & \cdots &0\\
                  &   &        &   \ddots&0 \\
               0   &   0&   0 &  1 & 0
            \end{bmatrix}
            \\
            h(c) &= \zeta^{k} \begin{bmatrix}
                1 & 0 & 0 & \cdots& 0\\
                0 & 1  & 0 & \cdots&0\\
                0 & 0 & 1 & \cdots &0\\
                  &   &        &   \ddots&0 \\
               0   &   0&   0 &  0 & 1
            \end{bmatrix},\qquad k \in \{ 0,1,\dots,N-1\}.
           \end{aligned} 
    \end{equation}
Here, $\zeta=e^{2\pi i/N}$, and
$\epsilon$ is $1$ if $N$ is odd and an $N$th root of $-1$ if $N$
is even.

Thus the Chern-Simons functional has exactly $N$ distinct critical
points in $\bonf(T^{3})_{\delta_{1}}$. These critical points are
irreducible (as they must be, on account of the coprime condition); and
it is shown in
\cite{K-higherrank} that these $N$ critical points are non-degenerate.
We can now describe the critical points in the case of a general
$(Y,K)$.

\begin{proposition}\label{prop:N-copies-crit}
    For any oriented pair $(Y,K)$, the set of critical points of the
    Chern-Simons functional on $\bonf(Y\# T^{3}, K,\Phi)_{\delta}$
    consists of $N$ disjoint copies of the space of representations
    \begin{equation}\label{eq:representation-space}
            \rho : \pi_{1}(Y\setminus K)\to \SU(N) 
    \end{equation}
    satisfying the condition that, for each oriented meridian $m$ of
    the knot or link $K$, the element $\rho(m)$ is conjugate to
    \eqref{eq:reflection-primed}.
\end{proposition}

\begin{proof}
    Rather than consider $\bonf(Y\# T^{3}, K,\Phi)_{\delta}$ as in the
    statement of the proposition, we may alternatively consider the
    isomorphic space
    \[
            \bonf(Y\# T^{3}, K,\Phi')_{\delta}
    \]
    with the alternative element $\Phi' \in
    \su(N)$ from \eqref{eq:Phi-primed}. As in the discussion of
    $T^{3}$ above, the
    critical points in $\bonf(Y\# T^{3}, K,\Phi')_{\delta}$ correspond
    to flat
    $\SU(N)$ connections on the complement of $K \amalg ( D\times
    S^{1})$
    in $Y\# T^{3}$, such that the holonomy around $\partial D$ is
    $e^{2\pi i /N}$ and $\rho(m)$ is in the conjugacy class of
    \eqref{eq:reflection-primed} for all oriented meridians $m$. The
    fundamental group of the complement in this connected sum is a
    free product, so the result follows from the fact that there are
    $N$ such flat connections on the $T^{3}$ summand, all of which are
    irreducible: see the discussion surrounding
    \eqref{eq:connected-sum-reps} in the introduction.
\end{proof}

\begin{corollary}
    In the case of the $3$-sphere, and an oriented classical knot or
    link $K \subset S^{3}$,  the set of critical points of the
    Chern-Simons functional on $\bonf(S^{3}\# T^{3}, K,\Phi)_{\delta}$
    consists of $N$ disjoint copies of the space of representations
    \begin{equation}\label{eq:knot-complement-rho}
            \rho : \pi_{1}(S^{3}\setminus K)\to U(N)  
    \end{equation}
    satisfying the condition that, for each oriented meridian $m$ of
    the knot or link $K$, the element $\rho(m)$ is conjugate to
    \eqref{eq:reflection}.
\end{corollary}

\begin{proof}
    In the case of $S^{3}$, the meridians generate the first homology
    of the complement of the link, and the space of $U(N)$
    representations that appears here is therefore identical to the
    space of $\SU(N)$ representations in the proposition above.
\end{proof}

As a special case, we have:

\begin{corollary}
For the unknot $K$ in $S^{3}$,
     the set of critical points  \[ \Crit\subset \bonf(S^{3}\# T^{3},
     K,\Phi)_{\delta} \]
    consists of $N$ disjoint copies $\CP^{N-1}$. Furthermore, the
    Chern-Simons function is Morse-Bott: at points of $\Crit$, the
    kernel of the Hessian is equal to the tangent space to $\Crit$.
\end{corollary}

\begin{proof}
Since the fundamental group of the knot complement is $\Z$, a
homomorphism $\rho$ as in the previous corollary is determined by the
image of a meridian $m$ in the conjugacy class \eqref{eq:reflection}.
This conjugacy class in $U(N)$ can
be identified with $\CP^{N-1}$ by sending an element to its
$(-1)$-eigenspace. The kernel of the Hessian can be computed from
Lemma~\ref{lem:non-degenerate-rho}, to establish the Morse-Bott
property.
\end{proof}

We introduce a  notation for the space of representations that appears
in the previous proposition:

\begin{definition}\label{def:Rep-def}
    We write
    \[
                \Rep(Y,K,\Phi') \subset \Hom( \pi_{1}(Y\setminus K) ,
                \SU(N) ) 
    \]
    for the space of homomorphisms $\rho$ such that, for all meridians
    $m$, the element $\rho(m)$ is conjugate to
    \eqref{eq:reflection-primed}.
\end{definition}

Recall from section~\ref{subsec:unitary-3d} that the set of critical
points is acted on by the group $\cH$, which in our present case is
the subgroup
\[
                    \cH \subset H^{1}(Y\# T^{3} ; \Z/N)
\]
consisting of elements with integer lifts.  Let
\[
\begin{aligned}
\cH' &\cong \Z/N \\
\cH' &\subset \cH
\end{aligned}
 \]
be the subgroup of $H^{1}(T^{3}; \Z/N)$
consisting of elements which are non-zero only on the
generator $c$ in $\pi_{1}(T^{3})$ and are zero on the generators $a$
and $b$.

\begin{lemma}
The action of this copy, $\cH'$, of $\Z/N$ on the set of critical
points, $\Crit$, in $\bonf(Y\# T^{3}, K,\Phi)_{\delta}$ is to cyclically permute
the $N$ copies of $\Rep(Y,K,\Phi')$ that comprise
$\Crit$ according to the description in
Proposition~\ref{prop:N-copies-crit}.
\end{lemma}

\begin{proof}
    From our description of $\Crit$, it follows that it sufficient to
    check the lemma for the case of the $N$ critical points in
    $\bonf(T^{3})_{\delta_{1}}$ (i.e.~the case that $Y$ is $S^{3}$ and
    $K$ is empty). These critical points are described in \eqref{eq:T3-solutions}, and
    the action of  $\cH'\cong\Z/N$ is to multiply $h(c)$ by
    the $N$th roots of unity.
\end{proof}

We can compare the values of the Chern-Simons functional on the $N$
copies of $\Rep(Y,K,\Phi')$ in $\Crit$. For example, in the case of
the unknot, because $\CP^{N-1}$ is connected, the functional is
constant on each copy and we can compare the $N$ values. The general
case is the next lemma. 

\begin{lemma}\label{lem:H-relative-CS}
    Let $\bs \in \cH'$ be the generator that evaluates to $1$ on the
    generator $c$ in $T^{3}$. Then for any $\alpha=[A]$ in
    $\bonf(Y\#T^{3}, K, \Phi)_{\delta}$ we have
    \begin{equation}\label{eq:CS-group-action}
                \CS(\bs(\alpha)) - \CS(\alpha) = -16\pi^{2}(N-1)
    \end{equation}
    modulo the periods of the Chern-Simons functional.
\end{lemma}

\begin{proof}
    The calculation reduces to a calculation for a connection $[A]$ on
    $T^{3}$ and its image under $\bs$. Pull back the $U(N)$ bundle
    $P_{1}$ (with $\gdet(P_{1})=\delta_{1}$) to the cylinder
    $[0,1]\times T^{3}$. Identify the two ends to form $S^{1}\times
    T^{3}$, gluing the bundle $P_{1}$ using an automorphism  $f$ of
    $P_{1}$ with $\gdet(f)=u$. Let $P\to S^{1}\times T^{3}$ be the
    resulting $U(N)$ bundle. We have
    \[
                    c_{1}(P) = \mathrm{PD}\bigl(
                            [S^{1}\times c] + [a\times b] 
                    \bigr)
    \]
    modulo multiples of $N$.
    The change in the Chern-Simons
    functional is half the ``topological energy'' $\cE(P)$ on
    $S^{1}\times T^{3}$, so
    \[
    \CS(\bs(\alpha)) - \CS(\alpha) =
          - 8\pi^{2} p_{1}(\g_{P})[S^{1}\times T^{3}],
    \]
     from \eqref{eq:energy-non-simple}. Using the relation
     \eqref{eq:p1-versus-c1-squared}, we obtain
     \begin{equation}\label{eq:c1-squared-calc}
         \CS(\bs(\alpha)) - \CS(\alpha) =
          - 8\pi^{2} c_{1}(P)^{2}[S^{1}\times T^{3}]
     \end{equation}
     modulo periods, and the final result follows from the above
     formula for $c_{1}(P)$, which gives
     \[
                   c_{1}(P)^{2}[S^{1}\times T^{3}] = 2\pmod{N}.  
     \]
\end{proof}

We can also consider the relative grading for the pair of
points $\alpha$ and
$\bs(\alpha)$ in $\bonf(Y\# T^{3},K,\Phi)_{\delta}$ along a suitable
chosen path $z$:
\[
                \gr_{z}(\bs(\alpha),\alpha) \in \Z.
\]
Note that this relative grading is easy to interpret unambiguously,
even when the Hessian at $\alpha$ has kernel, essentially because the
Hessians at $\alpha$ and $\bs(\alpha)$ are isomorphic operators.

\begin{lemma}\label{lem:H-relative-gr}
    Let $z$ be a path in $\bonf(Y\# T^{3},K,\Phi)_{\delta}$ along
    which a single-valued lift of the Chern-Simons functional
    satisfies \eqref{eq:CS-formula}. Then along this path we have
    \[
          \gr_{z}(\bs(\alpha),\alpha) = - 4(N-1).      
    \]
\end{lemma}

\begin{proof}
    Concatenating $z$ with its image under the maps $\bs$,
    $\bs^{2}$, \dots, $\bs^{N-1}$, we obtain a closed loop, along
    which the total energy $\cE$ is $-32\pi^{2}(N-1)N$. From the
    monotone relationship between the dimension and energy, we have
    that the spectral flow along the closed loop is $-4(N-1)N$. The
    spectral flow along each part is therefore $-4(N-1)$, because each
    part makes an equal contribution.
\end{proof}

According to Proposition~\ref{prop:invariant-pert}, we can choose a
holonomy perturbation $\pert$ which is invariant under $\cH'$ while
still making  the critical point set non-degenerate and the moduli
spaces regular. If follows that we have an action of $\cH'$ on
$\TI{N}{Y,K}$ resulting from this geometric action on the
configuration
space. (Without using an invariant perturbation, the action can still
be defined -- using cobordisms --
by the procedure described in subsection~\ref{subsec:unitary-3d}.)
Since the grading in $\TI{N}{Y,K}$ is only defined modulo $2N$, we can
interpret the last lemma
above as saying that the action of $\bs$ gives an automorphism of 
$\TI{N}{Y,K}$ of degree $4$:
\[
                \bs_{*} : \tI^{N}_{j}(Y,K) \to
                \tI^{N}_{j+4}(Y,K).
\]
(The subscript is to be interpreted mod $2N$.)

Whereas the construction using cobordisms only gives us an action on
the homology, the geometric action on the configuration space gives us
an action on the chain complex.  So, rather than consider the action
of $\cH'\cong\Z/N$ on the instanton homology group, we can instead
consider dividing $\bonf(Y\# T^{3}, K ,\Phi)_{\delta}$ by the action
of $\cH'$ and then taking the Morse homology. We can interpret
Proposition~\ref{prop:invariant-pert} as telling us that $\pert$ can
be chosen so that the Morse construction works appropriately on
$\bonf(Y\# T^{3}, K ,\Phi)_{\delta}/\cH'$. The relative grading on the
Morse complex is defined mod $4$ if $N$ is even, and mod $2$ if $N$ is
odd.

\begin{definition}\label{def:bar-version}
    We  define $\bar\tI\mathstrut^{N}_{*}(Y,K)$ to be the homology of the
    quotient chain complex $\bar{C}_{*}(Y\# T^{3},K,\Phi)_{\delta}$:
    the quotient of $C_{*}(Y\# T^{3},K,\Phi)_{\delta}$ by the action
    of $\Z/N$.  We write $\bar\tI\mathstrut^{N}_{*}(K)$ in the case
    that $Y$ is the $3$-sphere.
\end{definition}

We can calculate this group in the case of the unknot.

\begin{proposition}\label{prop:basic-unknot}
    For the unknot $K$ in $S^{3}$, we have
    \[
    \begin{aligned}
    \bar\tI\mathstrut^{N}_{*}(S^{3},K) &\cong
             H_{*}(\CP^{N-1};\Z) \\
              &\cong \Z^{N},\\
                  \tI\mathstrut^{N}_{*}(S^{3},K) &\cong
             H_{*}(\CP^{N-1} \amalg \dots \amalg \CP^{N-1};\Z),
             \quad\text{\textrm{($N$ copies)}}\\
              &\cong \Z^{N^{2}}.
    \end{aligned}
\]
\end{proposition}

\begin{proof}
    The group $\bar\tI\mathstrut^{N}_{*}(S^{3},K)$ is the homology of
    $\bar{C}_{*}(Y\# T^{3},K,\Phi)_{\delta}$, and this chain complex
    has generators corresponding to the points of $\Crit_{\pert}/\cH'$,
    for suitable holonomy perturbation $\pert$. Before perturbation,
    $\Crit_{\pert}$ consists of $N$ copies of $\CP^{N-1}$ and $\cH'$
    is a copy of $\Z/N$ which permutes the $N$ copies cyclically. So
    $\Crit/\cH'$ is a single copy of $\CP^{N-1}$.

    Choose a holonomy perturbation $\pert_{1}$ which is invariant under
    $\cH'$ and is such that the
    corresponding function $f_{1}$ on $\bonf(S^{3}\# T^{3},
    K)_{\delta}$ has the property that $f_{1}|_{\Crit}$ is a standard
    Morse function with even-index critical points on each copy of
    $\CP^{N-1}$. Then set $\pert_{\epsilon}= \epsilon\pert_{1}$ and
    take $\epsilon$ a small, positive quantity. Because the
    Chern-Simons functional is Morse-Bott, an application of the
    implicit function theorem and the compactness theorem for critical
    points shows that, for $\epsilon$ sufficiently small, the
    perturbed critical set $\Crit_{\pert_{\epsilon}}/\cH'$ consists of
    exactly $N$ critical points. As $\epsilon$ goes to zero, these
    converge to the $N$ critical points of $f_{1}|_{\CP^{N-1}}$, and
    the relative grading of the points in $\Crit_{\pert_{\epsilon}}$
    along paths in the neighborhood of $\Crit$ is equal to the
    relative Morse grading of the corresponding critical points of
    $f_{1}|_{\CP^{N-1}}$. It follows that, for this perturbation, the complex has $N$
    generators, all of which are in the same grading mod $2$.
\end{proof}

In the case $N=2$, the invariant $\bar\tI_{*}(K)$ of classical knots
appears to resemble Khovanov homology in the simplest cases. As
mentioned in the introduction, it is natural to ask whether we have
\[
            \bar\tI_{*}(K) \cong \kh(K)
\]
for the $(2,p)$, $(3,4)$ and $(3,5)$ torus knots, for example.

\subsection{Bridge number and representation varieties}

For a  knot $K$ in $S^{3}$, the knot group (i.e.~the fundamental group
of the knot complement) is generated by the conjugacy class of the
meridian. If we choose meridional elements $m_{1},\dots, m_{k}$ which
generate the knot group, then
a homomorphism $\rho$ as in 
\eqref{eq:knot-complement-rho} from the knot group to $U(N)$ is
entirely determined by $k$ elements
\[
                A_{i} = \rho(m_{i})
\]
in the conjugacy class of the reflection \eqref{eq:reflection}: or
equivalently, $k$ points in $\CP^{N-1}$.  The $-1$-eigenspaces of the
reflections $A_{i}$ will span at most a  $k$-plane in $\C^{N}$. It follows that
$\rho$ is conjugate in $U(N)$ to a representation whose image lies in
$U(k)\subset U(N)$.

In this sense, the problem of describing the space of representations
$\rho$ stabilizes at $N=k$. For larger $N$, the homomorphisms $\rho$
to $U(N)$ are obtained from the homomorphisms to $U(k)$ by conjugating
by elements of the larger unitary group.  An upper bound for the
number $k$ of meridians that are needed to generate the knot group is
the bridge number of the knot. So for a $k$-bridge knot, the problem
of computing the set of critical points $\Crit$ in
$\bonf(S^{3}\#T^{3},K,\Phi_{\delta})$ for the group $U(N)$ can be
reduced to the corresponding problem for $U(k)$ (though the critical
point sets are not the same).

The simplest example after the unknot is the trefoil, a $2$-bridge
knot. The group is generated by a pair of meridians, and for $N=2$ a
representation $\rho$ is therefore determined by a pair of points in
$\CP^{1}\cong S^{2}$. The relation between the generators in the knot
group implies that these two points in $S^{2}$ either coincide or make
an angle $2\pi/3$. The set of all such representations for $N=2$ is
therefore parametrized by one copy of $S^{2}$ and one copy of $SO(3)$.
When $N$ is larger, we have essentially the same classification: a
representation is determined by a pair of points in $\CP^{N-1}$, and
either these coincide, or they lie at angle $2\pi/3$ from each other
along the unique $\CP^{1}$ that contains them both. These two
components are a copy of $\CP^{N-1}$ itself and a copy of the unit
sphere bundle in $T\CP^{N-1}$ respectively. The authors conjecture
that for the trefoil, $\bar\tI^{N}_{*}(K)$ is isomorphic to the direct
sum of the homologies of these two components of the critical set of
the unperturbed functional.

For a  general knot $K$,
the critical set $\Crit$, after perturbation, determines the set of
generators of the complex that computes $\tI_{*}^{N}(K)$.
It would be interesting to know whether there is any sort of
stabilization that occurs for the \emph{differentials} in the complex,
as $N$ increases.  The situation is reminiscent of the large-$N$ stabilization
for Khovanov-Rozansky homology that is discussed in
\cite{Dunfield-Gukov-Rasmussen, Rasmussen-differentials}.

\subsection{A reduced variant}

In the construction of $\tI_{*}^{N}(Y,K)$, the important feature of
the manifold $T^{3}$ with which we formed the connected sum was that,
for a suitable choice of $U(N)$ bundle $P_{1}\to T^{3}$, the
corresponding set of critical points $\Crit(T^{3}, P_{1})$ was just a
finite set of reducibles (a single orbit of the finite group
$\cH'\cong\Z/N$). Rather than $T^{3}$, we can consider the pair
$(S^{1}\times S^{2}, L)$, where $L\subset S^{1}\times S^{2}$ is the
$(N+1)$-component link
\[
                L = S^{1} \times \{ p_{0}, \dots p_{N} \}.
\]
Let $P_{0} \to S^{1}\times S^{2}$ be the trivial $\SU(N)$ bundle and
let $\Phi' \in \su(N)$ be the element \eqref{eq:Phi-primed}. In the
space of singular connections $\bonf(S^{1}\times S^{2}, L, \Phi')$,
consider again the set of critical points:
\[
                \Crit(S^{1}\times S^{2}, L, \Phi')
                \subset \bonf(S^{1}\times S^{2}, L, \Phi').
\]
The pair $(S^{1}\times S^{2}, L)$ with this choice of $\Phi'$ fits the
hypotheses of Corollary~\ref{cor:irreducibles-comp}, and it follows
that the critical set consists only of irreducible flat connections.

\begin{lemma}
    The critical set $\Crit(S^{1}\times S^{2}, L, \Phi')$ consits of
    exactly
    $N$ non-degenerate, irreducible points. These form a single orbit
    of the group $\cH = H^{1}(S^{1}\times S^{2}; \Z/N)$.
\end{lemma}

\begin{proof}
    The critical set parametrizes conjugacy classes of homomorphisms
    \[
             \rho:   \pi_{1}(S^{1}\times S^{2} \setminus L) \to \SU(N)
    \]
    such that $\rho$ of each meridian of $L$ is conjugate to
    $\exp(\Phi')$. The fundamental group is a product, with a $\Z$
    factor coming from the $S^{1}$. The
    lemma will follow if we can
    show that there is only a single, irreducible, conjugacy class of
    homomorphisms
    \[
                   \sigma :  \pi_{1}( S^{2} \setminus \{p_{1},\dots
                   p_{N+1}\}) \to \SU(N)
    \]
    such that $\sigma$ sends the linking circle of each $p_{j}$ into
    the conjugacy class of $\exp(\Phi')$. The classification of such
    homomorphisms $\sigma$ can be most easily achieved by using the
    correspondence with stable parabolic bundles on the Riemann sphere
    with $(N+1)$ marked points. In this instance, the relevant
    parabolic bundles are holomorphic bundles $\cE\to \CP^{1}$ of rank
    $N$ and degree $0$, equipped with a  distinguished line $\cL_{i}$
    in the fiber $\cE_{p_{i}}$ for each $i$. The appropriate stability
    condition for such a parabolic bundle $(\cE, \{\cL_{i}\})$ is
    that, for every proper holomorphic subbundle $\cF\subset \cE$, we should have
    \[
                    \#\{ \, i \mid
                         \cL_{i}\subset \cF_{p_{i}} \,\}
                         + 2 \deg(\cF) \le \mathrm{rank}(\cF).
    \]
    The only solution to these constraints is to take $\cE$ to be the
    trivial bundle $\mathcal{O} \otimes \C^{N}$ and to take $\cL_{i}$ to be
    $\mathcal{O}_{p_{i}}\otimes L_{i}$, where the lines $L_{i}$ define $N+1$
    points in general position in $\CP^{N-1}$. There is therefore a
    single homomorphism $\sigma$ satisfying the given conditions.
\end{proof}

Now let $(Y,K)$ be an arbitrary pair, and let $k_{0}$ be a basepoint
on $K$. Choose a framing of $K$ at $k_{0}$, and use this framing
data to form the connected sum of pairs
\[
                 (\hat Y ,\hat K) =   (Y,K) \# (S^{1}\times S^{2}, L),
\]
connecting the component of $K$ containing $k_{0}$ to the component
$S^{1}\times \{p_{0}\}$ of $L$. We define a \emph{reduced} version of
the framed instanton homology by setting
\[
            \Ir_{*}^{N}(Y,K) = \I_{*}(\hat Y , \hat K, \Phi').
\]
Like $\tI_{*}^{N}(Y,K)$, this group has an affine grading by
$\Z/(2N)$.

The definition is such that, for the case that $Y$ is $S^{3}$ and $K$
is the unknot, the pair $(\hat Y, \hat K)$ is simply $(S^{1}\times
S^{2}, L)$. The lemma above thus tells us that, in this case, the
complex that computes $\Ir^{N}_{*}$ has $N$ generators. The relative
grading of these generators is even, and we therefore have
\[
                \Ir_{*}^{N}(S^{3},\mathrm{unknot}) = \Z^{N}.
\]

The set of critical points in $\bonf(\hat Y, \hat K,\Phi')$ for a
general $(Y,K)$ can be described by a version of
Proposition~\ref{prop:N-copies-crit}. Let us again write
$\Rep(Y,K,\Phi')$ for the space of homomorphisms described in
Definition~\ref{def:Rep-def}. As a basepoint for the fundamental group
$\pi_{1}(Y\setminus K)$ let us choose the push-off the the chosen point
$k_{0}\in K$ using the framing. We then have a preferred meridian,
$m_{0}$, in $\pi_{1}(Y\setminus K)$ linking $K$ near this basepoint,
and hence an evaluation map taking values in the conjugacy class
$C(\exp\Phi')$ of the element $\exp(\Phi')$:
\begin{equation}\label{eq:ev}
\begin{aligned}
\ev : \Rep(Y,K,\Phi') &\to C(\exp\Phi') \\
        \rho &\mapsto \rho(m_{0}).
\end{aligned}
\end{equation}
The counterpart to Proposition~\ref{prop:N-copies-crit} is then:

\begin{proposition}
     For any oriented pair $(Y,K)$, the set of critical points of the
    Chern-Simons functional on $\bonf(\hat Y, \hat K,\Phi')$
    consists of $N$ copies of the fiber of the evaluation map
    \eqref{eq:ev}. \qed
\end{proposition}

Thus, for example, in the case of the trefoil, the set of critical
points consists of $N$ points and $N$ copies of the sphere $S^{2N-3}$.

As in the case of $\tI^{N}_{*}$, it is possible to pass to the
quotient of the configuration space by the action of the cyclic group
$\cH'$. The result is a variant, $\bar\Ir^{N}_{*}(Y,K)$, which is
isomorphic to $\Z$ in the case of the unknot. In the quotient space
$\bonf(\hat Y, \hat K,\Phi')/\cH'$, the set of critical points
consists of just one copy of the fiber of the evaluation map $\ev$
above.

Rasmussen \cite{Rasmussen-observations} has observed that the reduced
Khovanov homology coincides with Heegaard-Floer knot homology group
$\knotfloer(K)$, defined by Ozsv\'ath and Szab\'o in \cite{OS-knot-hf}
and by Rasmussen in \cite{Rasmussen-knot-hf}, for many knots, but not
for the $(4,5)$ torus knot. It would therefore be interesting to have
some data comparing the reduced group $\bar\Ir_{*}(K,k_{0})$ with
$\knotfloer(K)$.

\subsection{Longitudinal surgery}

Another variant briefly mentioned in the introduction is to begin with a
null-homologous knot $K$ in an arbitrary $Y$, and form the pair
$(Y_{K}, K_{0})$, where $Y_{K}$ is the $3$-manifold obtained by
$0$-surgery (longitudinal surgery), and $K_{0}$ is the core of the
solid torus used in the surgery. The knot $K_{0}$ represents a
primitive element in the first homology of the manifold $Y_{K}$.
We can apply our basic construction to this pair (without taking a
further connected sum), with $G=\SU(N)$ as usual. To avoid reducibles,
we can take $\Phi$ to have just two distinct eigenvalues with coprime
multiplicity (see Corollary~\ref{cor:irreducibles-examples}). In
particular,  we may again take the $\Phi'$ given in
equation~\eqref{eq:Phi-primed}. We make a definition to cover this
case:

\begin{definition}
    For a null-homologous knot $K$ in a $3$-manifold $Y$, we define
    $\Il^{N}_{*}(Y,K)$ to be the result of applying the standard
    construction $\I_{*}$ to the oriented pair $(Y_{K},K_{0})$, with
    $G=\SU(N)$ and $\Phi$ given by \eqref{eq:Phi-primed}.

    For $N=2$, we just write $\Il_{*}(Y,K)$, and if $Y=S^{3}$ we drop
    the $Y$ from our notation. \qed
\end{definition}

The complement of $K_{0}$ in $Y_{K}$ is homeomorphic to the original
complement of $K$ in $Y$, and the meridian of $K_{0}$ corresponds to
the longitude of $K$. Thus we see that the set of critical points of
the Chern-Simons function in the configuration space for
$(Y_{K},K_{0})$ can be identified with the space of conjugacy classes
of homomorphisms
\begin{equation}\label{eq:longitude-homomorphisms}
                    \rho : \pi_{1}(Y\sminus K) \to \SU(N)
\end{equation}
satisfying the constraint that $\rho$ maps the longitude of $K$ to an
element conjugate to $\exp(2\pi i \Phi)$. In the case of the unknot in
$S^{3}$, the group is $\Z$ and the longitude represents the identity
element, so the set of critical points is empty. For the unknot
therefore, the group $\Il^{N}_{*}(K)$ is zero. The same applies to an
``unknot'' in any $Y$, i.e.~a knot that bounds a disk.

 In part
because of the results of \cite{KM-wittenp} and \cite{KM-dehn}, it is
natural to conjecture that $\Il_{*}(Y,K)$ is zero only if $K$ is an
unknot.
An examination of the
representation variety, and a comparison with Casson's work
\cite{Akbulut-McCarthy}, suggests that the Euler characteristic of
$\Il_{*}(K)$ should be $2\Delta''_{K}(1)$, where $\Delta_{K}$ is the
symmetrized Alexander polynomial. In the case of a torus knot $K$, the
representation variety of homomorphisms $\rho$ satisfying the
longitudinal constraint
\[
            \rho(\mathrm{longitude}) \sim 
            \begin{pmatrix}
                i & 0 \\
                0 & -i
            \end{pmatrix}
\]
consists of exactly $2\Delta''_{K}(1)$ points, and it would
follow that
\[
                \Il_{*}(K) = \Z^{2\Delta''_{K}(1)}
\]
for all torus knots. (For the $(2,p)$ torus knot with $p>3$, it would
follow from this that $\Il_{*}(K)$ is not isomorphic to Eftekhary's
longitude variant of Heegaard Floer homology \cite{Eftekhary}; because
for the $(2,p)$ knot, the quantity $\Delta''_{K}(1)$ grows
quadratically with $p$, while the rank of Eftekhary's knot invariant
grows linearly.) 

\section{Filtrations and genus bounds}

In this section, we take up a theme from the introduction and explore
how a  lower bound for the slice genus of a knot can be obtained from
(a variant of) instanton Floer homology. In doing so, we will see that
the formal outline of the construction can be made to resemble closely
the work of Rasmussen in \cite{Rasmussen-slice}, while the underlying
mechanism of the proof draws on \cite{K-obstruction}.

From this point on, we will work exclusively with the group $G=\SU(2)$
and the element $\Phi=i \diag(1/4, -1/4)$ in the fundamental alcove
(this is the only balanced case for the group $\SU(2)$). We will write
$\tI_{*}(K)$ for the framed instanton homology of a  classical knot,
as described in Definition~\ref{def:basic-with-T3-def},
and $\tI_{*}(Y,K)$ when the knot is in a $3$-manifold other than
$S^{3}$. To keep the notation to a minimum, we will treat only
classical knots to begin with, though little changes when we generalize
to knots in other $3$-manifolds.

\subsection{Laurent series and local coefficients}

Let $K \subset S^{3}$ be an oriented knot, and let us write
\[
                \bonf(K) = \bonf(S^{3} \# T^{3}, K,\Phi)_{\delta}.
\]
for the configuration space that is used in defining the framed
instanton Floer homology group, Definition~\ref{def:basic-with-T3-def}.
As in section~\ref{subsec:local-coeffs},
we will consider the Floer homology of $K$
coming from a system of local
coefficients on the configuration space $\bonf(K)$.
Specifically, we let $\mu: \bonf\to U(1)=\R/\Z$ be a circle-valued
function arising from the standard character $G_{\Phi} \to U(1)$,
using the holonomy construction of \eqref{eq:construction-of-mu}; and
we let $\Gamma = \Gamma^{\mu}$ be the corresponding local system
pulled back from $\R/Z$, as described in
section~\ref{subsec:local-coeffs}. Thus $\Gamma$ is a local system of
free modules of rank $1$ over the ring
\[
\begin{aligned}
R &= \Z[\Z] \\
   &= \Z[t^{-1}, t].
\end{aligned}
\]
Thus for each knot $K$, we have a finitely-generated
$R$-module
\[
                \tI_{*}(K;\Gamma).
\]
We should recall at this point that the construction of the local
system $\Gamma^{\mu}$ really depends on a choice of framing $n$ for the
knot $K$, but that the local systems arising from different choices of
framings are all canonically identified (see
section~\ref{subsec:local-coeffs} again). So we should regard $\Gamma$
as the common identification of a collection of local systems
$\Gamma^{\mu_{n}}$ as $n$ runs through all framings.

A cobordism of pairs, $(W,S)$, from $(S^{3}, K_{0})$ to
$(S^{3},K_{1})$ gives rise to a homomorphism of the corresponding
Floer groups, by the recipe described at
\eqref{eq:local-coeff-cobordism-int}. (The way we have set it up, a framing of the normal
bundle to  $S$ is
needed in order to define the map, but the resulting map is
independent of the choices made.) We will abbreviate our notation for
the map and write
\[
                \psi_{W,S} : \tI_{*}(K_{0};\Gamma) \to
                \tI_{*}(K_{1};\Gamma).
\]
A homology-orientation of $W$ is needed to fix the sign of the map
$\psi_{W,S}$.

From the argument of Proposition~\ref{prop:basic-unknot} we obtain:

\begin{proposition}\label{prop:local-coeff-unknot}
    For the unknot $U$ in $S^{3}$, the Floer group $\tI_{*}(U;\Gamma)$
    is a free $R$-module of rank $4$.
\end{proposition}

The definition of the local coefficient system $\Gamma$ and the maps
induced by a  cobordism are related to the monopole number.
To understand this relationship, consider in general two different cobordisms of
pairs, 
$(W, S)$ and $(W', S')$ from $(S^{3}, K_{0})$ to $(S^{3}, K_{1})$,
(with both $S$ and $S'$ being embedded surfaces, not immersed). Let
$\beta_{0}$ and $\beta_{1}$ be critical points in $\bonf(K_{0})$ and
$\bonf(K_{1})$ respectively, and let $z$ and $z'$ be paths from
$\beta_{0}$ to $\beta_{1}$ along $(W,S)$ and $(W', S')$ respectively.
Corresponding to $z$ and $z'$, the local system gives us maps of the
form \eqref{eq:Delta-z}, which in this case means we have
\[
            \Delta_{z}, \Delta_{z'} : t^{\mu(\beta_{0})} R \to
             t^{\mu(\beta_{1})}  R.
\]
Both of these maps are multiplication by a certain (real) power of
$t$:
\[
\begin{aligned}
            \Delta_{z} &= t^{\nu(z)} \\
            \Delta_{z'} &= t^{\nu(z')}.
\end{aligned}
\]
We can express the difference between $\nu(z)$ and $\nu(z')$
in topological terms. The surfaces $S$ and $S'$ are not closed, so
there is not a well-defined monopole number $l$ for the classes $z$
and $z'$; but there is  a well-defined relative monopole number: we
can write
\[
                d(z-z') = (k,l)
\]
where $k$ is the relative instanton number and $l$ is the relative
monopole number (both are integer-valued). There is a also a ``relative''
self-intersection number $S^{2}-(S')^{2}$  (the self-intersection number
of the union of $S$ and $-S'$ in $W \cup (-W')$. In these terms, we
have
\[
                \nu(z) - \nu(z') = -l + (1/4)(S^{2}-(S')^{2}).
\]
This is essentially the formula (17) of \cite{K-obstruction}, which
expresses the curvature integral which defines $\nu$ in terms of the
topological data. Our $\nu$ corresponds to $-\lambda$ in
\cite{K-obstruction}. If we fix a reference cobordism $(W_{*}, S_{*})$
and a path $z_{*}$ along it from $\beta_{0}$ to $\beta_{1}$, then
the contribution involving $\beta_{0}$ and $\beta_{1}$ to the map
$\psi_{W,S}$ can be written (in the style of 
definition \eqref{eq:boundary-map-def}) as
\begin{equation}\label{eq:explanation-of-psiWS}
                \sum_{z} \sum_{[\breve{A}]}
                 \epsilon[\breve{A}] \, \Delta_{z} =
                 \sum_{z} \sum_{[\breve{A}]}
                 \epsilon[\breve{A}] \, t^{-l +
                 (1/4)(S^{2}-(S_{*})^{2}) +
                 \nu_{z_{*}}}.
\end{equation}

\subsection{Immersed surfaces and canonical isomorphisms}

A cobordism of pairs from $(Y_{0}, K_{0})$ to $(Y_{1}, K_{1})$, as
considered so far, consists of a $4$-dimensional cobordism $W$ and an
\emph{embedded} surface $S$ with $\partial S = K_{1} - K_{0}$. It is
convenient to follow \cite{K-obstruction} and consider also
\emph{immersed} surfaces $S$. We will always consider only smoothly
immersed surfaces $f: S \looparrowright W$ with normal crossings
(transverse double-points), all of which should be in the interior of
$W$. As is common, we often omit mention of $f$ and confuse $S$ with
its image in $W$. By blowing up $W$ at each of the double points (forming a
connected sum with copies of $\bar\CP^{2}$) and taking the proper
transform of $S$ (\cf \cite{K-obstruction}, we have a canonical
procedure for replacing any such immersed cobordism $(W,S)$ with an
embedded version, $(\tilde W, \tilde S)$. We then \emph{define} a map
\[
                \psi_{W,S}: \tI_{*}(K_{0};\Gamma) \to
                \tI_{*}(K_{1};\Gamma)
\]
corresponding to the immersed cobordism by declaring it to be equal to
the map obtained from its resolution:
\[
                    \psi_{W,S}:= \psi_{\tilde{W},\tilde{S}}.
\]

Now suppose that $f_{0}:S\to W$ and $f_{1}:S \to W$ are two
immersions in $W$ with normal crossings, and suppose that they are
homotopic as maps relative to the boundary.  Then the image $S_{1}$ of
$f_{1}$ can be
obtained from $S_{0}=f(S_{0})$ by a sequence of standard moves, each of which
is either
\begin{enumerate}
\setcounter{enumi}{-1}
    \renewcommand{\theenumi}{(\arabic{enumi})}
    \item \label{item:basic-0}
           an ambient isotopy of the image of the immersed surface in
    $W$,
    \item \label{item:basic-1}
           a twist move introducing a positive double point,
    \item  \label{item:basic-2}
          a twist move introducing a negative double point,
    \item  \label{item:basic-3}
          a finger move introducing two double points of opposite
    sign,
\end{enumerate}
or the inverse of one of these \cite{Freedman-Quinn}. To analyze the
relation between the maps $\psi_{W,S_{0}}$ and $\psi_{W,S_{1}}$, we
must therefore analyze the effect of each of these types of elementary
changes to an immersed surface. This was carried out in
\cite{K-obstruction} for the case of closed surfaces in a  closed
$4$-manifold, and the same arguments work as well in the relative
case, leading to the following result.

\begin{proposition}\label{prop:basic-moves}
    Let $S$ be obtained from $S'$ by one of the three basic moves
    \ref{item:basic-1}--\ref{item:basic-3}. Then the maps
    \[
\begin{aligned}
    \psi_{W,S'} : \tI_{*}(K_{0},\Gamma) &\to \tI_{*}(K_{1},\Gamma) \\
    \psi_{W,S} : \tI_{*}(K_{0},\Gamma) &\to \tI_{*}(K_{1},\Gamma) \\
\end{aligned}
    \]
    are related by, respectively,
    \begin{enumerate}
\setcounter{enumi}{0}
    \renewcommand{\theenumi}{(\arabic{enumi})}
    \item \label{item:result-1}
           $\psi_{W,S} = (t^{-1}-t)\psi_{W,S'}$,
    \item  \label{item:result-2}
          $\psi_{W,S} = \psi_{W,S'}$ (no change), and
    \item  \label{item:result-3}
          $\psi_{W,S} = (t^{-1}-t)\psi_{W,S'}$ (the same as case
          \ref{item:result-1}).
\end{enumerate}
\end{proposition}

\begin{remark}
    The three cases of this proposition can be summarized by saying
    that the map $\psi_{W,S}$ acquires a factor of $(t^{-1}-t)$ for
    every positive double point that is introduced.
\end{remark}

\begin{proof}
As indicated above, this is essentially Proposition~3.1 of
\cite{K-obstruction}. In that proposition, the monopole number $l$
contributes to the power of $t$ in the coefficients of the map
$\psi_{W,S}$, according to the formula
\eqref{eq:explanation-of-psiWS}. The other contribution to the
exponent is the self-intersection number of the proper transform of
the immersed surface $S$, which
changes by $-4$ when $S$ acquires a positive double point and is
unaffected by negative double-points.
\end{proof}

Let us now return to situation where we have two homotopic maps
$f_{i}:S\to W$ with images $S_{i}$, $i=0,1$.
If we form the ring $R'$ by inverting the element $(t^{-1}-t)$ in $R$,
so
\[
                R' = R[ (t^{-1}-t)^{-1} ],
\]
and if we write
\[
\psi'_{W,S}=\psi_{W,S}\otimes 1 : \tI_{*}(K_{0};\Gamma)\otimes_{R} R' \to
\tI_{*}(K_{1};\Gamma)\otimes_{R} R'
\]
then the proposition tells us:

\begin{corollary}\label{cor:differ-by-unit}
    If $S_{0}$ and $S_{1}$ are the images of homotopic immersions as
    above, then the maps
    \[
           \psi'_{W,S_{i}}: \tI_{*}(K_{0};\Gamma)\otimes_{R} R' \to
           \tI_{*}(K_{1};\Gamma)\otimes_{R} R'
     \]
     differ by multiplication by a unit. More specifically, if $\tau(S_{i})$ is
     the number of positive double-points in $S_{i}$, then we have
     \[
                \psi'_{W,S_{1}} = (t^{-1}-t)^{\tau(S_{1}) -
                \tau(S_{0})} \psi'_{W,S_{0}}.
     \]     
\end{corollary}

\begin{corollary}
    For any two knots $K_{0}$ and $K_{1}$, we have
    \[
                   \tI_{*}(K_{0};\Gamma)\otimes_{R} R'
                   \cong \tI_{*}(K_{1};\Gamma)\otimes_{R} R' 
    \]
\end{corollary}

\begin{proof}
     In the cylindrical cobordism $W=[0,1]\times S^{3}$, let $S$ be
     any immersed annulus from $K_{0}$ to $K_{1}$, and let $\bar{S}$
     be any annulus from $K_{1}$ to $K_{0}$. The concatenation of
     these immersed annular cobordisms, in either order, give annular
     immersed cobordisms from $K_{0}$ to $K_{0}$ and from $K_{1}$ to
     $K_{1}$. These composite annuli are each homotopic to a trivial
     product annulus; so the composite maps
     \[
                    \psi'_{W,S} \circ \psi'_{W,\bar{S}}
     \]
     and
      \[
                    \psi'_{W,\bar{S}} \circ \psi'_{W,S}
     \]
     are both the identity map, and it follows that $\psi'_{W,S}$ and
     $\psi'_{W,\bar{S}}$ are isomorphisms.
\end{proof}

Extending this line of thought a little further, we see that the group
$\tI_{*}(K;\Gamma)\otimes_{R}R'$ is not just independent of $K$ up to
isomorphism, but up to \emph{canonical} isomorphism.  That is, if
$K_{0}$ and $K_{1}$ are any two knots, we may choose any immersed
annulus $S$ from $K_{0}$ to $K_{1}$ in the 4-dimensional product
cobordisms $W=[0,1]\times S^{3}$ and construct the isomorphism
\[
                    (t^{-1}-t)^{-\tau(S)} \psi'_{W,S}:
                    \tI_{*}(K_{0};\Gamma)\otimes_{R} R'
                   \to \tI_{*}(K_{1};\Gamma)\otimes_{R} R' .
\]
This isomorphism is independent of the choice of annulus $S$. In
particular, for any knot $K$, the $R'$-module
$\tI_{*}(K;\Gamma)\otimes_{R}R'$ is canonically isomorphic to the $R'$-module
arising from the unknot. From
Proposition~\ref{prop:local-coeff-unknot} we therefore obtain:

\begin{corollary}\label{cor:canonical-iso}
    For any knot $K$ in $S^{3}$, there is a canonical isomorphism
    \[
                   \Psi:   ( R' )^{4} \to \tI_{*}(K;\Gamma)\otimes_{R}R'.
    \]
\end{corollary}

\subsection{Filtrations and double-point bounds}

The inclusion $R\to R'$ gives us a canonical copy of $R^{4}$ in
$(R')^{4}$, the image of $\tI_{*}(U;\Gamma)$ in $\tI_{*}(U;\Gamma)\otimes_{R}R'$ for
the unknot $U$. We define an increasing filtration of
$\tI_{*}(K;\Gamma)\otimes_{R}R'$,
\[
                \cdots \subset \cF^{-1}(K)\subset
                \cF^{0}(K)
                \subset \cF^{1}(K) \subset \cdots
\]
by first setting
\[
                \cF^{0}(K) = \Psi(R^{4}) \subset \tI_{*}(K;\Gamma)\otimes_{R}R',
\]
where $\Psi$ is the canonical isomorphism of
Corollary~\ref{cor:canonical-iso},
and then defining
\begin{equation}\label{eq:filtration}
            \cF^{i}(K) = (t^{-1}-t)^{-i} \cF^{0}(K).
\end{equation}
Although $\cF^{0}(U)$ is the image of $\tI_{*}(U;\Gamma)$ in the tensor
product for the case of the unknot, this does not hold for a general
knot. We make the following definition, modelled on similar
constructions in \cite{Rasmussen-slice, OS-tau, Lobb}.

\begin{definition}
    For any knot $K$ in $S^{3}$ we define  $\vr(K)$ to be the smallest
    integer $i$ such that $\cF^{-i}(K)$ is contained in the image of
    $\tI_{*}(K;\Gamma)$ in $\tI_{*}(K;\Gamma)\otimes_{R}R'$.
\end{definition}

To see that the definition makes sense, choose an immersed annular
cobordism from the unknot $U$ to $K$, and let $\tau$ be the
number of positive double points in this annulus. As $R$-submodules of
$\tI_{*}(K;\Gamma)\otimes_{R}R'$, we have, from the definitions,
\[
            \begin{aligned}
                \cF^{-\tau}(K) &= \Psi( (t^{-1}-t)^{\tau}
                R^{4}) \\
                &= \psi'_{W,S}\bigl( \tI_{*}(U;\Gamma) \bigr) \\
                &\subset \im \big( \tI_{*}(K;\Gamma) \to
                \tI_{*}(K;\Gamma)\otimes_{R}R')
            \end{aligned}
\]
where the last inclusion holds because passing to the ring $R'$
commutes with the maps $\psi_{W,S}$ and $\psi'_{W,S}$ induced by the
cobordism. From this observation and the definition of $\vr(K)$, we
obtain
\[
                  \vr(K) \le \tau.
\]
Since the annulus $S$ was arbitrary, we have:

\begin{theorem}\label{thm:double-point-bound}
    Let $K$ be a knot in $S^{3}$ and let $D$ be an immersed disk with
    normal crossings in the 4-ball, with boundary $K$. Then the number
    of positive double points in $D$ is at least $\vr(K)$. \qed
\end{theorem}

\subsection{Algebraic knots}

The lower bound for the number of double points in an immersed disk,
given by Theorem~\ref{thm:double-point-bound}, is sharp for the case
of an algebraic knot (a knot arising as the link of a singularity in
a complex plane curve, such as a torus knot). The reason this is so
comes down to the same mechanisms that were involved in
\cite{KM-gtes-I,KM-gtes-II} and \cite{K-obstruction}, where singular
instantons were used to obtain bounds on unknotting numbers and slice
genus.

To explain this, we recall some background from
\cite{KM-gtes-I,K-obstruction}. Let
$(X,\Sigma)$ be a closed pair, with $\Sigma$ connected for
simplicity,  and let $P\to X$ be a $U(2)$ bundle.
Suppose that $c_{1}(P)$ satisfies the non-integrality condition, that
\[
            \tfrac{1}{2}c_{1}(P) \pm \tfrac{1}{4}[\Sigma]
\]
is not an integer class, for either choice of sign. Denoting by $k$
the instanton number of $P$, we have for any choice of monopole number
$l$ a moduli space $M_{k,l}(X,\Sigma)_{\delta}$, which we label by $k$
$l$ and $\delta$, where $\delta$ is the line bundle $\det P$. If the
formal dimension of this moduli space is zero, then there an integer
invariant
\[
            q^{\delta}_{k,l}(X,\Sigma) \in \Z.
\]
(A homology orientation is needed as usual to fix the sign.) In
\cite{K-obstruction}, these integer invariants are combined into a
Laurent series (with only finitely many non-zero terms):
\[
                    R^{\delta}(X,\Sigma)(t) = 2^{-g(\Sigma)} \sum_{(k,l):\dim
                    M_{k,l}=0} t^{-l}q^{\delta}_{k,l}(X,\Sigma).
\]
The normalizing factor $2^{-g}$ was convenient in \cite{K-obstruction}
but is not significant here. The definition of
$q^{\delta}_{k,l}(X,\Sigma)$ and $R^{\delta}(t)$ is extended to the
case of immersed surfaces with normal crossings by blowing up.

Suppose now that $(X,\Sigma)$ decomposes along a $3$-manifold $Y$,
meeting $\Sigma$ transversely in a knot $K$. Suppose also that the
restriction of $P$ to $(Y,\Sigma)$ satisfies the non-integrality
condition. Then there is a gluing formula which expresses the Laurent
series $R^{\delta}(X,\Sigma)(t)$ as  a pairing between a cohomology
class and a homology class in the Floer group
\[
                \I_{*}(Y,K,P;\Gamma).
\]
The coefficient system $\Gamma$ is the one we have been using, and
keeps track of the power of $t$.

Suppose now that we have a cobordism $(W,S)$ (with $S$ immersed
perhaps) from $(Y_{0}, K_{0})$ to $(Y_{1},K_{1})$, giving us a map
\[
            \psi : \I_{*}(Y_{0},K_{0},P;\Gamma) \to
            \I_{*}(Y_{1},K_{1},P;\Gamma).
\]
Suppose we wish to show that the image of $\psi$ is not contained in the
image of multiplication by $(t^{-1}-t)$. From the functorial
properties and the gluing formulae, it will be sufficient if we can
find a closed pair $(X,\Sigma)$ (together with a $U(2)$ bundle $P$)
such that:
\begin{figure}
    \begin{center}
        \includegraphics[scale=0.90]{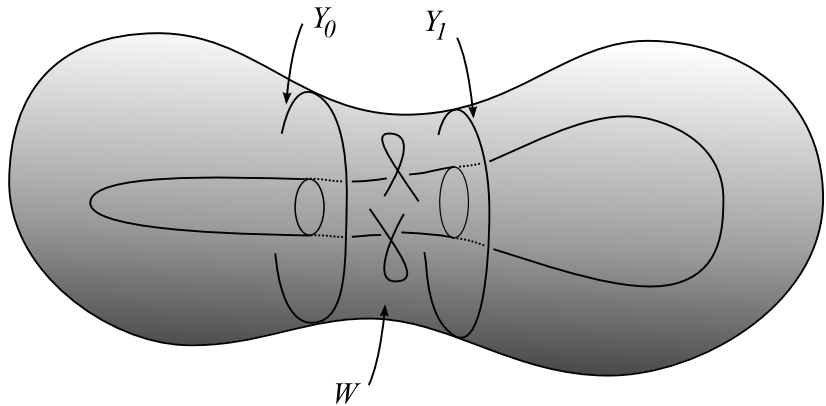}
    \end{center}
    \caption{A closed pair $(X,\Sigma)$ separated by $(W,S)$, with $S$
    immersed.}
\end{figure}
\begin{itemize}
    \item $(X,\Sigma,P)$ contains $(W,S,P)$ as a separating subset,
    as indicated in the figure;  and
    \item the Laurent series $R^{\delta}(X,\Sigma)(t)$ does not vanish
    at $t=1$.
\end{itemize}
Summarizing this discussion, we therefore have:

\begin{proposition}
    Suppose $(X,\Sigma)$ is a  pair (with $\Sigma$ perhaps immersed)
    such that, for some $\delta$, the finite Laurent series
    $R^{\delta}(X,\Sigma)(t)$ is non-vanishing at $t=1$. Suppose that
    $(X,\Sigma)$ has a  decomposition as shown, and suppose:
    \begin{itemize}
        \item $Y_{0}\cong Y_{1}\cong S^{3} \# T^{3}$;
        \item $W$ the 4-dimensional product cobordism;
        \item $c_{1}(\delta)$ is dual to a standard circle in $T^{3}$;
        \item $K_{0}$, $K_{1}$ arise from classical knots in $S^{3}$,
        with $K_{0}$ the unknot;
        \item the surface $S$ arises from an immersed annulus in
        $[0,1]\times S^{3}$ with $\tau$ double positive points.
    \end{itemize}
    Then $\vr(K_{1})=\tau$ and the bound of
    Theorem~\ref{thm:double-point-bound} is sharp for $K_{1}$. \qed
\end{proposition}

Consider now the $4$-torus $T^{4}$ as a complex surface, containing an
algebraic curve $C$ with a unibranch singularity at a point $p$. Let
$B_{1}$ be a small ball around $p$ so that the
curve $C$ meets $\partial B_{1}$  in a knot $K_{1}$. In a $C^{\infty}$
manner, we can alter $C$ in the interior of $B_{1}$ to obtain an immersed
surface $\tilde C$, so that the part of $\tilde C$ that is in the
interior of $B_{1}$ is isotopic to a complex-analytic immersed disk with
$\tau$ positive double-points. Let $B_{0}\subset B_{1}$ be a smaller
$4$-ball, meeting $\tilde C$ in a standard embedded disk, so that the
part of $\tilde C$ that lies between $B_{0}$ and $B_{1}$ is an
immersed annulus with $\tau$ double points. Let $T$ be a real $2$-torus in
$T^{4}$ disjoint from $C \cup B$, and let $\delta$ be a line bundle
with $c_{1}(\delta)[T]=1$. Let $Y_{1}\cong S^{3}\# T^{3}$ be obtained
as an internal connected sum of $\partial B_{1}$ with the boundary of a
tubular neighborhood of $T$. Similarly, let $Y_{0}$ be obtained as the
internal connected sum of $\partial B_{0}$ with the boundary of a
smaller tubular neighborhood of $T$. Then the pair $(T^{4}, \tilde C)$ has
a decomposition as shown in the diagram, satisfying the itemized
conditions of the theorem above.

To show that $\vr(K_{1})=\tau$ for this algebraic knot $K_{1}$, we
therefore need only show that the corresponding Laurent series
$R^{\delta}(T^{4},\tilde C)(t)$ is non-zero at $t=1$. Because of the
results of \cite{K-obstruction}, this is equivalent to showing that
$R^{\delta}(T^{4},\Sigma)(1)$ is non-zero, where $\Sigma$ is a smooth
algebraic curve (embedded in the abelian surface). Using the results
of \cite{KM-gtes-I} however, we can calculate this Laurent series. We
are free to arrange that $\Sigma$ has odd genus, that
$c_{1}(\delta)[\Sigma]$ is zero, and that $c_{1}(\delta)^{2}=0$. The
terms in the series come from the invariants $q_{k,l}^{\delta}$ with
\[
                2k + l - \tfrac{1}{2}(g-1) = 0,
\]
and from \cite{KM-gtes-I} we learn that
\[
                q_{k,l}^{\delta} = 
                \begin{cases}
                    2^{g} q_{0}^{\delta}(T^{4}), & \text{$k=0$ and
                    $l=(g-1)/2$,} \\
                    0, &\text{otherwise,}
                \end{cases}
\]
where $q^{\delta}_{0}(T^{4})$ is the Donaldson invariant of $T^{4}$,
which is $2$. The Laurent series is therefore a non-zero multiple of a
certain power of $t$, and in particular is non-zero at $t=1$, as
required.

\subsection{The involution on the configuration space}

As we have defined it, the group $\tI_{*}(K;\Gamma)$ is a free $R$-module of
rank $4$ in the case that $K$ is the unknot. The four generators come
from the two $2$-spheres that make up the set of critical points of
the unperturbed Chern-Simons functional on $\bonf(K)$. However, there
is an involution on $\bonf(K)$, interchanging these two copies: this
is the action of the cyclic group $\cH'$ of order $2$. Recall that,
with $\Z$ coefficients, we defined $\bar\tI_{*}(K)$ (in
Definition~\ref{def:bar-version}, where we dealt with $\SU(N)$ for
arbitrary $N$) by passing to the
quotient $\bonf(K)/\cH'$ and taking the Morse theory in this quotient.

Because the local coefficient system $\Gamma$ on $\bonf(K)$ is
actually pulled back from the quotient $\bonf(K)/\cH'$, we can adapt
this construction to define an $R$-module
\begin{equation}\label{eq:bar-Gamma-version}
            \bar\tI_{*}(K;\Gamma)
\end{equation}
for any knot $K$. For a suitable choice of perturbation,
the complex that computes $\bar\tI_{*}(K;\Gamma)$ is
the quotient of the complex that computes $\tI_{*}(K;\Gamma)$ by an
involution that acts freely on the generators. In the case of the
unknot, this Floer group would be $R^{2}$ instead of $R^{4}$. Little
else in our discussion would need to be changed.

\subsection{Genus bounds}

Let
\[
\begin{aligned}
        f: S &\to W \\
        f': S' &\to W
\end{aligned}
\]
be two immersions with transverse self-intersections, having as
boundary the same knots $K_{0}\subset Y_{0}$ and $K_{1}\subset Y_{1}$.
We have seen that if $S=S'$ and $f\simeq f'$ relative to the
boundary, then the two resulting maps $\tI_{*}(K_{0};\Gamma)\to
\tI_{*}(K_{1};\Gamma)$ differ only by factors of $(t^{-1}-t)$
(Proposition~\ref{prop:basic-moves} and
Corollary~\ref{cor:differ-by-unit}). Another situation to consider is
the case that $S'$ is obtained from $S$ by adding a handle: forming an
internal connected sum with a $2$-torus contained in a ball in $W$.

The effect of adding a handle in this way was examined for the case
of closed pairs $(X,\Sigma)$ in \cite{K-obstruction}. In our present
context, the relevant calculation is the following. Let $W$ be the
$4$-dimensional product cobordism $[0,1]\times S^{3}$, and let
$S_{1}\subset W$ be a cobordism from the unknot to the unknot and
having genus $1$. This gives rise to a homomorphism
\[
            \psi_{1} : \tI_{*}(U;\Gamma) \to \tI_{*}(U;\Gamma).
\]
If we pass to the ``bar'' version of the Floer groups
\eqref{eq:bar-Gamma-version}, then we have also a map
\[
         \bar\psi_{1} : \bar\tI_{*}(U;\Gamma) \to
         \bar\tI_{*}(U;\Gamma).
\]
The group $\bar\tI_{*}(U;\Gamma)$ is a free $R$-module of rank $2$,
with generators in different degrees mod $4$. We can therefore
identify it with $R\oplus R$ with an ambiguity consisting of
multiplication by units on each summand. The map $\bar\psi_{1}$ must
be off-diagonal in this basis, because its degree is $2$ mod $4$; so we have a map
\[
                \bar\psi_{1} : R \oplus R \to R\oplus R
\]
of the form
\begin{equation}\label{eq:psi-matrix}
                    \bar\psi_{1} = 
                    \begin{pmatrix}
                        0 & p(t) \\
                        q(t) & 0
                    \end{pmatrix}
\end{equation}
for certain Laurent polynomials $p(t)$ and $q(t)$, well-defined up to
units.  From \cite{K-obstruction} we know the effect of adding
\emph{two} handles to a surface, and from that we deduce the relation
\[
                    p(t)q(t) = 4 (t - 2 + t^{-1}),
\]
or in other words
\begin{equation}\label{eq:barpsi-squared}
\begin{aligned}
\bar\psi_{1}^{2} &= 4(t - 2 + t^{-1}) \\
            &= 4 t^{-1}(t-1)^{2}.
\end{aligned}
\end{equation}
as an endomorphism of $\bar\tI_{*}(U;\Gamma)$.

At this point, because of the factor of $4$ in the above formula, we
shall pass to rational coefficients rather than integer coefficients:
without change of notation, let us redefine $R$ as $\Q[t^{-1}, t]$. We
again define $R'$ by inverting $(1-t)$ and $(1+t)$ in $R$.

Suppose now that we have an embedded cobordism $S$ of genus $g$ from the unknot
$U$ to $K$, inside $[0,1]\times S^{3}$. This gives rise to a map
\[
            \bar\psi_{S} : \bar\tI_{*}(U;\Gamma) \to  \bar\tI_{*}(K;\Gamma)
\]
and similarly
\[
            \bar\psi'_{S} : \bar\tI_{*}(U;\Gamma)\otimes_{R}R' \to
            \bar\tI_{*}(K;\Gamma)\otimes_{R}R'.
\]
The surface $S$ is homotopic to an
immersed cobordism $S^{+}$ which is a composite of two parts: the
first part is $g$ copies of the standard genus-1 cobordism $S_{1}$
from $U$ to $U$; and the second part is an immersed annulus. Let
$\tau$ be the number of positive double points in the immersed
annulus. From Corollary~\ref{cor:differ-by-unit} and the definition of
the canonical isomorphism $\Psi$, we obtain
\[
            \bar\psi'_{S} =
             \Psi\circ (\bar\psi'_{1})^{g} 
\]
where $\bar\psi'_{1}$ is the map defined by \eqref{eq:psi-matrix}.
This provides a constraint on the genus $g$: it must be that
$\Psi\circ (\bar\psi'_{1})^{g}$ carries $R\oplus R$ into the image of
$ \bar\tI_{*}(U;\Gamma) $ in $ \bar\tI_{*}(U;\Gamma)\otimes_{R}R'$.
This constraint gives us a lower bound for $g$, just as we obtained a
lower bound $\vr(K)$ for the number of double points previously.
For algebraic knots again, the bound will be sharp.

It is not inconceivable that, by working over $\Z$ and paying
attention to the factor $4$ above instead of passing to $\Q$, one
could obtain a stronger bound for $g$ in some cases, but the authors
have no evidence one way or the other. 

\bibliographystyle{abbrv}
\bibliography{yaft}

\end{document}